\documentclass[10 pt]{amsart}

\usepackage{amsmath}
\usepackage{amssymb}
\usepackage{mathrsfs}
\usepackage{ stmaryrd }
\usepackage[all,cmtip]{xy}
\usepackage{pigpen}
\usepackage{graphicx}
\usepackage{enumitem}
\usepackage{hyperref}
\usepackage{verbatim}
\numberwithin{equation}{subsubsection}
\newtheorem{introthm}[subsection]{Theorem}

\newtheorem{introconjecture}[subsection]{Conjecture}

\newtheorem{introprp}[subsection]{Proposition}
\newtheorem{prop}[subsubsection]{Proposition}
\newtheorem{thm}[subsubsection]{Theorem}

\newtheorem*{thm*}{Theorem}
\newtheorem{lemma}[subsubsection]{Lemma}
\newtheorem*{lem*}{Lemma}
\newtheorem{cor}[subsubsection]{Corollary}

\theoremstyle{definition}

\theoremstyle{remark}

\newtheorem*{assump*}{Assumption}

\theoremstyle{definition}
\newtheorem{definition}[subsubsection]{Definition}

\theoremstyle{remark}
\newtheorem{remark}[subsubsection]{Remark}

\newcommand{\R}{\mathbb{R}}  % The real numbers.

\newcommand{\C}{\ensuremath{\mathbb{C}}}

\newcommand{\A}{\ensuremath{\mathbb{A}}}

\newcommand{\der}{\ensuremath{\mathrm{der}}}
\newcommand{\ab}{\ensuremath{\mathrm{ab}}}

\newcommand{\Fss}{\ensuremath{\mathrm{ss}}}

\newcommand{\Sh}{\ensuremath{\mathrm{Sh}}}
\newcommand{\Q}{\ensuremath{\mathbb{Q}}}
\newcommand{\xbar}{\ensuremath{{\overline{x}}}}

\renewcommand{\O}{\ensuremath{\mathcal{O}}}
\newcommand{\WD}{\ensuremath{\mathrm {WD}}}

\newcommand{\Spec}{\ensuremath{{\mathrm{Spec\,}}}}
\newcommand{\Spf}{\ensuremath{\mathrm{Spf}}}
\newcommand{\Spd}{\ensuremath{\mathrm{Spd}}}
\newcommand{\Spa}{\ensuremath{\mathrm{Spa}}}
\newcommand{\Res}{\ensuremath{\mathrm{Res}}}
\newcommand{\GL}{\ensuremath{\mathrm{GL}}}

\newcommand{\GSp}{\ensuremath{\mathrm{GSp}}}

\newcommand{\SL}{\ensuremath{\mathrm{SL}}}

\newcommand{\SIsom}{\ensuremath{\underline{\mathrm{Isom}}}}
\newcommand{\Isoc}{\ensuremath{\mathrm{Isoc}}}
\newcommand{\lps}{[\![}
\newcommand{\rps}{]\!]}

\newcommand{\Adm}{\ensuremath{\mathrm{Adm}}}

\newcommand{\fkm}{\ensuremath{\mathfrak{m}}}

\newcommand{\fkM}{\ensuremath{\mathfrak{M}}}

\newcommand{\fkS}{\ensuremath{\mathfrak{S}}}

\newcommand{\ps}{\ensuremath{[\![u]\!]}}
\newcommand{\ls}{\ensuremath{(\!(u)\!)}}
\newcommand{\bbA}{\ensuremath{\mathbb{A}}}

\newcommand{\bbC}{\ensuremath{\mathbb{C}}}
\newcommand{\bbD}{\ensuremath{\mathbb{D}}}

\newcommand{\bbF}{\ensuremath{\mathbb{F}}}
\newcommand{\bbG}{\ensuremath{\mathbb{G}}}

\newcommand{\bbL}{\ensuremath{\mathbb{L}}}

\newcommand{\bbN}{\ensuremath{\mathbb{N}}}

\newcommand{\bbP}{\ensuremath{\mathbb{P}}}
\newcommand{\bbQ}{\ensuremath{\mathbb{Q}}}
\newcommand{\bbR}{\ensuremath{\mathbb{R}}}
\newcommand{\bbS}{\ensuremath{\mathbb{S}}}

\newcommand{\bbZ}{\ensuremath{\mathbb{Z}}}

\newcommand{\bfE}{\ensuremath{\mathbf{E}}}

\newcommand{\bfG}{\ensuremath{\mathbf{G}}}
\newcommand{\bfH}{\ensuremath{\mathbf{H}}}

\newcommand{\bfX}{\ensuremath{\mathbf{X}}}

\newcommand{\rmE}{\ensuremath{\mathrm{E}}}
\newcommand{\rmF}{\ensuremath{\mathrm{F}}}

\newcommand{\rmH}{\ensuremath{\mathrm{H}}}

\newcommand{\rmK}{\ensuremath{\mathrm{K}}}

\newcommand{\scrE}{\ensuremath{\mathscr{E}}}

\newcommand{\scrG}{\ensuremath{\mathscr{G}}}

\newcommand{\scrI}{\ensuremath{\mathscr{I}}}

\newcommand{\scrP}{\ensuremath{\mathscr{P}}}

\newcommand{\scrS}{\ensuremath{\mathscr{S}}}

\newcommand{\scrV}{\ensuremath{\mathscr{V}}}

\newcommand{\scrZ}{\ensuremath{\mathscr{Z}}}
\newcommand{\Mloc}{\ensuremath{\mathbb{M}^{\mathrm{loc}}}}
\newcommand{\calA}{\ensuremath{\mathcal{A}}}
\newcommand{\calB}{\ensuremath{\mathcal{B}}}
\newcommand{\calC}{\ensuremath{\mathcal{C}}}
\newcommand{\calD}{\ensuremath{\mathcal{D}}}
\newcommand{\calE}{\ensuremath{\mathcal{E}}}
\newcommand{\calF}{\ensuremath{\mathcal{F}}}
\newcommand{\calG}{\ensuremath{\mathcal{G}}}
\newcommand{\calH}{\ensuremath{\mathcal{H}}}

\newcommand{\calK}{\ensuremath{\mathcal{K}}}
\newcommand{\calL}{\ensuremath{\mathcal{L}}}
\newcommand{\calM}{\ensuremath{\mathcal{M}}}
\newcommand{\M}{\ensuremath{\mathcal{M}}}
\newcommand{\N}{\ensuremath{\mathcal{N}}}
\newcommand{\calN}{\ensuremath{\mathcal{N}}}
\newcommand{\calO}{\ensuremath{\mathcal{O}}}
\newcommand{\calP}{\ensuremath{\mathcal{P}}}

\newcommand{\calS}{\ensuremath{\mathcal{S}}}

\newcommand{\calV}{\ensuremath{\mathcal{V}}}

\newcommand{\calY}{\ensuremath{\mathcal{Y}}}

\newcommand{\brE}{\ensuremath{\breve{E}}}

\newcommand{\brQ}{\ensuremath{\breve{\mathbb{Q}}_p}}

\newcommand{\brZ}{\ensuremath{\breve{\mathbb{Z}}_p}}

\newcommand{\et}{\ensuremath{\mathrm{\acute{e}t}}}

\newcommand{\Hom}{\ensuremath{\mbox{Hom}}}

\newcommand{\Gal}{\ensuremath{\mbox{Gal}}}
\newcommand{\Rep}{\ensuremath{\mathrm {Rep}}}
\newcommand{\Mod}{\ensuremath{\mathrm{Mod}}}

\newcommand{\Conj}{\ensuremath{\mathrm{Conj}}}

\newcommand{\po}{\ar@{}[dr]|{\text{\pigpenfont R}}}
\newcommand{\pb}{\ar@{}[dr]|{\text{\pigpenfont J}}}
\begin{document}

%%
%% The title of the paper goes here.  Edit to your title.
%%

%%
%% Now edit the following to give your name and address:
%% 

\title[Strongly compatible systems associated to semistable abelian varieties]{Strongly compatible systems associated to semistable abelian varieties}
\author{Mark Kisin and Rong Zhou}
\date{\today}

	%\makeatletter
%\@namedef{subjclassname@2020}{\textup{2020} Mathematics Subject Classification}
%\makeatother

%\keywords{}
%\subjclass[2020]{}
\address{Department of Mathematics, Harvard University, Cambridge, USA, MA 02138}
\email{kisin@math.harvard.edu}

\address{Department of Pure Mathematics and Mathematical Statistics, University of Cambridge, Cambridge, UK, CB3 0WA}
\email{rz240@dpmms.cam.ac.uk}

\begin{abstract}
	
	We prove a motivic refinement of a result of Weil, Deligne and Raynaud on the existence of strongly compatible systems associated to abelian varieties. More precisely,  given
 an abelian variety $A$ over a number field $\rmE\subset \bbC$, we prove that after replacing $\rmE$ by a finite extension, the action of  $\mathrm{Gal}(\overline{\rmE}/\rmE)$ on the $\ell$-adic cohomology $\rmH^1_{\mathrm{\acute{e}t}}(A_{\overline{\rmE}},\bbQ_\ell)$  gives rise to a strongly compatible system of $\ell$-adic representations valued in the Mumford--Tate group $\bfG$ of $A$. This involves an independence of $\ell$-statement for the Weil--Deligne representation associated to $A$ at places of semistable reduction, extending previous work of \cite{KZ} at places of good reduction. 
\end{abstract}
\maketitle
\tableofcontents

\section{Introduction}

Let $X$ be a proper, smooth algebraic variety over a number field $\rmE \subset \mathbb C.$ The Tannakian fundamental group of the 
Hodge structure $\rmH^i(X_{\mathbb C},\Q)$ is a connected reductive $\mathbb Q$-group called the {\em Mumford-Tate group.} 
One expects that, after replacing $\rm E$ by a finite extension, for each prime $\ell$, 
the action of the absolute Galois group $\Gamma_{\rmE}:=\mathrm{Gal}(\bar \rmE/\rmE)$ on the $\ell$-adic cohomology 
$\rmH^i(X_{\bar{\rm E}}, \bbQ_\ell)$ gives rise to a representation 
$$ \rho_{X,\ell}^{\bfG}:  \Gamma_{\rmE} \rightarrow \bfG(\bbQ_\ell).$$ 

The representations $\rho^{\bfG}_{X,\ell}$ are expected to form a {\em compatible system} (cf.~\cite[12.6]{Serre}), 
in the following sense: For $v\nmid \ell$ a prime of $\rmE,$ restricting $\rho_{X,\ell}$ to the decomposition group at $v$ 
gives rise to a representation of the local Weil--Deligne group 
 $$\rho^{\WD,\bfG}_{X,\ell,v}:\WD_v\rightarrow \bfG(\bbQ_\ell).$$ 
 Fixing an isomorphism $i_\ell:\bar\bbQ_\ell\cong \bbC,$ we may view $\rho^{\WD,\bfG}_{X,\ell,v}$ as a $\bfG(\bbC)$-valued representation. 
 If $\rho_1,\rho_2:\WD_v\rightarrow \bfG(\bbC)$ are $\bfG$-valued Weil--Deligne representations, we write $\rho_1\sim_{\bfG}\rho_2$ if there exists $g\in \bfG(\bbC)$ such that $g^{-1}\rho_1g=\rho_2$.  We say that $\rho_1$ is defined over $\bbQ$ if for all $\sigma\in \mathrm{Aut}(\bbC/\bbQ)$, we have $\sigma\circ \rho_1\sim_{\bfG}\rho_1$. Then one has the following conjecture (cf.~\cite{TateNTB}, \cite{Fontaine},\cite[Introduction]{Noot2}) 
 \begin{introconjecture}\label{conj:mainconj} 
There exists a $\bfG$-valued Weil--Deligne representation $\rho_{X,v}^{\WD,\bfG}$ defined over $\bbQ$ such that  
$$\rho_{X,v}^{\WD,\bfG}\sim_{\bfG} \rho_{X,\ell,v}^{\WD,\bfG}\ \  \text{ for all primes $v\nmid \ell$}.$$ 
 \end{introconjecture}
One can even extend this to places $v|\ell$ using isocrystals to define $\rho_{A,\ell,v}^{\WD,\bfG}.$ 
For $v$ such that $\rho_{X,\ell}^{\bfG}$ is unramified at $v,$ 
the conjecture amounts to the independence of $\ell$ of the $\bfG$-conjugacy class of $\rho_{X,\ell}^{\bfG}(\mathrm{Frob}_v),$ 
where $\mathrm{Frob}_v$ is the geometric Frobenius at $v.$ 
 
 Now suppose that $X=A$ an abelian variety. Then a fundamental result of Deligne \cite{De1} asserts the existence of the representations 
 $\rho_{X,\ell}^{\bfG}.$ In our previous paper \cite{KZ} we showed Conjecture \ref{conj:mainconj} for places $v$ of good reduction for $A,$ and 
 of residue characteristic $> 2$  (see also Noot in \cite{Noot} and \cite{Ki3} for some previous partial results). The main goal of this paper is to 
 generalize these results to places $v$ where $A$ has semistable reduction, and also to include a version of the result for $v|\ell:$ 
 
 \begin{introthm}\label{introthm: main}
	Let $v$ be a place where $A$ has semistable reduction. There exists a $\bfG$-valued Weil--Deligne representation $\rho_{A,v}^{\WD,\bfG}$ defined over $\bbQ$ such that  
	$$\rho_{A,v}^{\WD,\bfG}\sim_{\bfG} \rho_{A,\ell,v}^{\WD,\bfG}\ \  \text{ for all primes $\ell$ (including $v|\ell$).}$$ 
\end{introthm}
 
 A partial result in the direction of the theorem was previously proved by Noot  \cite{Noot2}. In his result the equivalence relation $\sim_{\bfG}$ is replaced by a coarser relation $\sim_{\bfG}'$, and there is an additional assumption that the Frobenius element is weakly neat.
If one instead considers the composite of $\rho_{A,v}^{\WD,\bfG}$ with the standard representation $\bfG \rightarrow \GL_{2g},$ 
arising from the $\ell$-adic Tate module of $A$ (here $g=\dim A$), then the analogue of Theorem \ref{introthm: main} is due to Raynaud \cite{GRR}. 
In fact for $\GL_{2g}$-valued representations one can even drop the semi-stability assumption by using a base change argument 
to reduce to the semistable case, cf. \cite[Lemma 1]{Saito}.
 It is not clear to us how to do this for $\bfG$-valued representations; 
this seems to be related to the failure of the group $\bfG$ to be \emph{acceptable} in the sense of \cite{Larsen}.

Although the main point of the theorem is to deal with the case of semi-stable reduction, let us mention that the case of good reduction 
with $v|\ell$ allows us to verify \cite[Hypothesis 2.3.1]{vHoftenOrdinary}, which implies quite a general version of the Chai-Oort Hecke orbit conjecture 
for Shimura varieties of Hodge type.

If one assumes the system $\rho^{\bfG}_{X,\ell}$ in Conjecture \ref{conj:mainconj} arises from an automorphic form $\Pi$ on the 
dual group $\check{\bf G},$ then Conjecture \ref{conj:mainconj} can be regarded as a consequence of local-global compatibility 
in the Langlands correspondence. Turning this into an actual proof seems out of reach of available techniques. 
Instead we reduce the compatibility in Theorem \ref{introthm: main} to an analogous compatibility for 
{\em function fields}, where one can apply the results of Lafforgue \cite{Laf}, and also \cite{Abe} on existence of compatible systems on 
curves over finite fields. This ultimately relies on local-global compatibility for function fields. 
The reduction uses in an essential way the arithmetic properties of the Shimura variety $\Sh_{\rmK}(\bfG,X)$ 
associated to $\bfG,$ which we now explain. 

To simplify the discussion, we restrict to the case $v\nmid \ell$ for the rest of this introduction. 
For $\rmK\subset \bfG(\bbA_f)$ a compact open  subgroup, $\Sh_{\rmK}(\bfG,X)$ is defined over a number field $\bfE$ contained in $\rmE$ and parameterizes abelian varieties whose Mumford--Tate groups are contained in $\bfG$. The integral model $\scrS_{\rmK}(\bfG,X)$ is a scheme defined over the ring of integers of the completion of $\bfE$  at a place $v'$ lying below $v,$ and  is equipped with a natural $\bfG(\bbQ_\ell)$-torsor $\bbL_\ell.$  
For a finite extension $\kappa$ of the residue field $\kappa(v')$ and $x\in \scrS_{\rmK}(\bfG,X)(\kappa)$, the geometric $|\kappa|$-Frobenius 
$\text{\rm Frob}_x,$ acting on the geometric stalk of $\bbL_\ell$ at $x$ gives rise to an element $\gamma_{x,\ell}\in \Conj_{\bfG}(\bbQ_\ell),$ 
where $\Conj_{\bfG}$ denotes the variety of semi-simple conjugacy classes in $\bfG.$ 
By considering the abelian variety $A$ as a point $x_A\in \Sh_{\rmK}(\bfG,X)(\rmE)$, Theorem \ref{introthm: main} for places of good reduction follows from 
\begin{introthm}\label{introthm: interior}
Let $(\bfG,X,\rmK)$ be a strongly admissible triple (see Definition \ref{def: strongly admissible}) and assume $\bfG_{\bbQ_p}$ is quasi-split. 
Then for all $\kappa/\kappa(v')$ finite and $x\in \scrS_{\rmK}(\bfG,X)(\kappa)$, there exists $\gamma_x\in \Conj_{\bfG}(\bbQ)$ such that 
$\gamma_x=\gamma_{x,\ell}$ for all $\ell$ with $v\nmid\ell.$
\end{introthm}

Although it is not always true that one can extend $(\bfG, X)$ to a strongly admissible triple, the general case of Theorem \ref{introthm: main} 
can be reduced to this one. The theorem was proved for $p>2$ and the more restrictive notion of \emph{strongly acceptable} triple in \cite{KZ}; see Definition 4.2.2 of \emph{loc. cit.}. 
In this paper, we make use of advances in the study of $p$-adic shtukas and Shimura varieties \cite{PRshtukas}, \cite{GLX} to give a completely different proof, and in particular extend the result to the case $p=2$. 
In particular, we show that each isogeny class in  $\scrS_{\rmK}(\bfG,X)(\bar\kappa)$ contains a element which lifts to a special point, generalizing a result of Gleason--Lim--Xu \cite[Corollary 1.4]{GLX} and verifying \cite[Conjecture 1]{KMS}.

To prove Theorem \ref{introthm: main} for places of bad reduction, which is the main new innovation of this paper, we make use of 
the  toroidal compactifications $\scrS_{\rmK}^\Sigma(\bfG,X)$ of $\scrS_{\rmK}(\bfG,X)$ constructed in \cite{Keerthi}.
Our abelian variety $A$ with semistable reduction, then corresponds to a point $x_A\in\scrS_{\rmK}^\Sigma(\bfG,X)(\calO_{\rmE_v}),$ whose special fiber lies on the boundary. We show that there exists a smooth $\O_{\rmE_v}$-curve $\mathcal C$ equipped with a point $\tilde c\in \calC(\calO_{\rmE_v})$ and a map 
$\tilde\pi:\mathcal C \rightarrow \scrS_{\rmK}^\Sigma(\bfG,X)$ such that 
\begin{itemize}
\item $\tilde\pi^{-1}(\scrS_{\rmK}(\bfG,X)) \subset \mathcal C$ is the complement of a point $\tilde b \in \mathcal C(\O_{E_v})$  with non-empty generic fiber.
\item The point $\tilde{\pi}(\tilde c)\in \scrS_{\rmK}^\Sigma(\bfG,X)(\calO_{\rmE_v})$ is equal to  $x_A.$
\end{itemize}
Now let $C = \mathcal C\otimes \kappa(v),$ and $c \in C(\kappa(v))$ the reduction of $\tilde c.$ 
We denote by $\WD_c$ the Weil--Deligne group of the local field obtained by completing $C$ at $c.$ 
As $A$ has semistable reduction at $v$, the associated Weil--Deligne representation $\rho_{A,\ell,v}^{\WD,\bfG}$ is unipotently ramified, and 
 Frobenius semisimple (we will call this (URFS) for short) and hence is determined by a pair 
$(s,N)\in \bfG(\bbC)^{\mathrm{ss}}\times \mathrm{Lie}\bfG_{\bbC}$ satisfying $\mathrm{Ad}(s)N=qN$, where $q=|\kappa(v)|.$  Here $\bfG(\bbC)^{\mathrm{ss}}$ denotes the set of semi-simple elements of $\bfG(\bbC)$.

The pair $(s,N)$ also corresponds to the representation $\rho^{\WD,\bfG}_{\ell,c}$ of $\WD_c,$ obtained from $\tilde\pi^*(\bbL_{\ell}).$ 
This allows us to reduce Theorem \ref{introthm: main} to a statement comparing Weil--Deligne representations for a local 
field of {\em equal characteristic $p$}. Namely, it suffices to show for $v\nmid \ell,\ell'$ we have 
\begin{equation} \tag{$\dagger$}\rho^{\WD,\bfG}_{\ell,c} \sim_{\bfG} \rho^{\WD,\bfG}_{\ell',c}. 
\end{equation}
Our next observation is that the condition $(\dagger)$ can be reduced to the case of $\GL_n$-representations: 

\begin{introprp}\label{introprop: WD conj}
	Let $\rho_1, \rho_2:\WD_v\rightarrow\bfG(\bbC)$ be $\bfG$-valued Weil--Deligne representations which are (URFS). Then 
	$\rho_1\sim_{\bfG}\rho_2$ if and only if for all representations $r:\bfG_{\bbC}\rightarrow \GL_{n}$, we have 
	$$r\circ\rho_1\sim_{\GL_n}r\circ\rho_2.$$
\end{introprp}
When $\rho_1,\rho_2$ are unramified, and hence determined by a semisimple element of $\bfG$, this is a result of Steinberg \cite{Steinberg:regular}. Note that this proposition fails without the assumption of (URFS) on the Weil--Deligne representations in question; see Remark \ref{rem: acceptable group}. 

Using Proposition \ref{introprop: WD conj}, it then suffices to show that  for any $\ell,\ell'$, we have 
\begin{equation} \tag{$\dagger\dagger$} r\circ \rho_{\ell,c}^{\WD,\bfG}\sim_{\GL_n}r\circ \rho_{\ell',c}^{\WD,\bfG}\end{equation} 
for all representations $r:\bfG_{\bbC}\rightarrow \GL_n$. To do this we consider the $G(\Q_{\ell})$-local sytem 
$\tilde\pi^*(\mathbb L_{\ell})$ on $C-c,$ and we denote by $\mathcal L_{\ell,r}$ the $\GL_n$-local system induced from $\tilde\pi^*(\mathbb L_{\ell})$ 
via $r.$ By Theorem \ref{introthm: interior}, for $v\nmid \ell,\ell'$ the local systems $\mathcal L_{\ell,r}$ and $\mathcal L_{\ell',r}$ 
are compatible at any point  $x \in C - c,$ in the sense that  
$\mathrm{Frob}_x$ has the same characteristic polynomial when acting on the two local systems $\mathcal L_{\ell,r}$ and $\mathcal L_{\ell',r}.$ 
The results of Lafforgue \cite{Laf} (and \cite{Abe} for $v|\ell$) now imply that $\mathcal L_{\ell,r}$ and $\mathcal L_{\ell',r}$ are compatible 
in the stronger sense that ($\dagger\dagger$) holds.

We now explain the organization of the paper.  \S2 and \S3 contain the main arithmetic and geometric results concerning Shimura varieties which are needed for the applications to abelian varieties in \S5. In \S2, we  recall the construction of integral models of Shimura varieties and prove Theorem \ref{introthm: interior}. In \S3, we recall the construction of compactifications of integral models of Shimura varieties, and prove that they admit good maps from certain smooth curves over the rings of integers of $p$-adic fields. 
\S4 is devoted to the study of $\bfG$-valued Weil--Deligne representations: \S4.1 contains the proof of Proposition \ref{introprop: WD conj}, and in \S4.2, \S4.3 we explain a way to compare Weil--Deligne representations associated to local systems over fields of different characteristic.  \S4.2 deals with $\ell$-adic local systems, and \S4.3 with their $p$-adic counterparts, namely overconvergent isocrystals. In both cases we formulate this  using log geometry, which seems to us to be the most conceptual way to think about such a comparison. Finally in \S5 we put the previous results together to prove Theorem \ref{introthm: main} using the strategy outlined above.

\emph{Acknowledgements:} The authors would like to thank Michael Rapoport, Anthony J. Scholl and Jack Thorne for enlightening discussions. 
M.K is supported by NSF grant DMS-2502675. R. Z. is supported by EPSRC grant ref. EP/Y030648/1.
\section{Integral models of Shimura varieties}\label{sec: integral models}

\subsection{Pappas--Rapoport integral models}
In this section, recall the construction of integral models of Shimura varieties of Hodge type. We use the results of \cite{GLX} and \cite{PRshtukas} to prove the existence of  CM lifts for isogeny classes, and then 
deduce an $\ell$-independence property for these models, generalizing results in \cite{KZ}.

\subsubsection{}\label{sssec: integral models preamble}
Fix an algebraic closure $\bar{\Q}$, and for each place $v$ of $\Q$ (including $v=\infty$) an algebraic closure $\bar{\Q}_v$ together with an embedding $i_v:\bar{\Q}\rightarrow \bar{\Q}_v$ (here $\bar{\bbQ}_\infty\cong \bbC$). 

Let $\bfG$ be a reductive group over $\Q$ and $X$ a $\bfG_{\bbR}$-conjugacy class of homomorphisms
$$h:\mathbb{S}:=\text{Res}_{\mathbb{C}/\R}\mathbb{G}_m\rightarrow \bfG_\mathbb{R}$$
such that $(\bfG,X)$ is a Shimura datum in the sense of \cite{De}.  We write $\{\mu\}$ for the conjugacy class of the  cocharacter
$$\bbC^\times\rightarrow \bbC^\times\times c^*(\bbC^\times)\xrightarrow{h}\bfG(\bbC);$$ here $c$ denotes the complex conjugation. We set $w_h=\mu_h^{-1}\mu_h^{c-1}$.

For a compact open subgroup $\rmK\subset \bfG(\bbA_f)$ of the finite adelic points of $\bfG$, the set  \begin{equation}\label{eqn: points of Shimura var}\Sh_{\rmK}(\bfG,X)_{\bbC}=\bfG(\bbQ)\backslash X\times \bfG(\bbA_f)/\rmK\end{equation} can be identified with the complex points of a smooth algebraic stack (cf. \cite[\S4.1.1]{KZ}). The theory of canonical models implies that $\Sh_{\rmK}(\bfG,X)_{\bbC}$ has a model  $\Sh_{\rmK}(\bfG,X)$ over the reflex field $\bfE\subset \bbC$, which is defined to be the field of definition of the conjugacy class $\{\mu\}$. We may consider $\bfE$ as a subfield of $\bar{\bbQ}$ via the embedding $i_\infty$. If $\rmK$ is sufficiently small (in fact if $\rmK$ is neat, see \cite[p 24.]{Milne}), then $\Sh_{\rmK}(\bfG,X)$ is in fact an algebraic variety.

Now fix a prime $p$ and let $\A_f^p$ be the ring of prime to $p$ adeles, which we consider as a subgroup of $\A_f$ with trivial $p$-component. From now on, we assume that the compact open subgroup $\rmK$ is of the form $\rmK=\rmK_p\rmK^p$, where $\rmK_p\subset \bfG(\bbQ_p)$ and $\rmK^p\subset\bfG(\A_f^p)$ are  compact open subgroups, and we define $$\Sh_{\rmK_p}(\bfG,X):=\lim_{\leftarrow\rmK^p}\Sh_{\rmK_p\rmK^p}(\bfG,X);$$
this is a pro-variety equipped with an action of $\bfG(\bbA_f^p)$.

\subsubsection{} Let $V$ be a $\bbQ$-vector space equipped with a  perfect alternating form $\psi$.  Let $\mathbf{GSp}(V)$ be the corresponding group of symplectic similitudes and let $S^\pm$ be the Siegel double space. The Shimura datum $(\bfG,X)$ is said to be of Hodge type if there exists an embedding of Shimura data $$\iota:(\bfG,X)\rightarrow (\mathbf{GSp}(V),S^\pm);$$
 such an $\iota$ will be called  a Hodge embedding.

We set $G:=\bfG_{\bbQ_p}$.  Let $\Lambda\subset V_{\bbQ_p}$ be a $\bbZ_p$-lattice on which the form $\psi$ takes values in $\bbZ_p$ and set $\Lambda_{\bbZ_{(p)}}=\Lambda\cap V$ which is a $\bbZ_{(p)}$-module in $V$.  We  let $\bar\calG_{\bbZ_{(p)}}$ denote the Zariski closure of $\bfG$ in $\GL(\Lambda_{\bbZ_{(p)}})$ and we set $\bar\calG=\bar\calG_{\bbZ_{(p)}}\otimes_{\bbZ_{(p)}}\bbZ_p$.
We assume the compact open subgroup $\rmK_p\subset G(\bbQ_p)$ is given by $\bar\calG(\bbZ_p)$. Let $\rmK'_p\subset \mathbf{GSp}(\bbQ_p)$ denote the stabilizer of $\Lambda$ and let $\rmK'^{p}\subset \mathbf{GSp}(V_{\bbA_f^p})$ be a compact open subgroup containing $\rmK^p$ and  set $\rmK'=\rmK'_p\rmK'^{p}$.  By \cite[Lemma 2.1.2]{Ki2}, we may choose $\rmK'^{p}$ sufficiently small (but still containing $\rmK^p$) such that the induced map of Shimura varieties $$\Sh_{\rmK}(\bfG,X)\rightarrow \Sh_{\rmK'}(\mathbf{GSp}(V),S^{\pm})_{\bfE}$$ is a closed immersion. 

\subsubsection{}  \label{sssec: construction integral models}Let $v|p$ be a place of $\bfE.$ Upon modifying $i_p:\bar{\bbQ}\rightarrow\bar{\bbQ}_p$, we may assume $v$ is induced by this embedding. Let $E$ be the completion of $\bfE$ at $v$, and let $\calO_E$ (resp. $k_E$)  be its ring of integers (resp. residue field). We fix $k$  an algebraic closure of $k_E$, and we let $\brE$ denote the completion of the maximal unramified extension of $E$. The lattice $\Lambda_{\bbZ_{(p)}}$ gives rise to an interpretation of $\Sh_{\rmK'}(\mathbf{GSp}(V),S^{\pm})$ as a moduli space for abelian varieties up to prime-to-$p$ isogeny, equipped with a weak polarization, and hence to an integral model $\scrS_{\rmK'}(\mathrm{GSp}(V),S^{\pm})$ over $\bbZ_{(p)}$ (see \cite[\S4]{KP} and \cite[\S6]{Z}).   
When $\rmK^p$ is neat, we define $\scrS_{\rmK}(\bfG,X)^-$ to be the Zariski closure of $\Sh_{\rmK}(\bfG,X)$ inside $\scrS_{\rmK'}(\mathrm{GSp}(V),S^{\pm})_{\calO_E}$, and we define $\scrS_{\rmK}(\bfG,X)$ to be the normalization of $\scrS_{\rmK}(\bfG,X)^-$. For a general (not necessarily neat) compact open  subgroup $\rmK^p$, we let $\rmK_1^p\subset \rmK^p$ be a neat normal compact open subgroup. The action of the finite group $\rmK^p/\rmK_1^p$ on $\Sh_{\rmK_p\rmK_1^p}(\bfG,X)$ extends to an action on $\scrS_{\rmK_p\rmK_1^p}(\bfG,X)$, and we define $\scrS_{\rmK}(\bfG,X)$ to be the stack  quotient of $\scrS_{\rmK_p\rmK_1^p}(\bfG,X)$ by this action.

For notational simplicity, we write $\scrS_{\rmK}$ for $\scrS_{\rmK}(\bfG,X)$, and we
set $$\scrS_{\rmK_p}:=\varprojlim_{\rmK^p}\scrS_{\rmK_p\rmK^p}.$$ 
We let $\widehat{\scrS_{\rmK}}$ denote the $p$-adic completion of $\scrS_{\rmK}$. For $x\in\scrS_{\rmK}(k)$, we let $\widehat{\scrS_{\rmK,x}}$ denote the formal neighbourhood of $\scrS_{\rmK}$ at $x$.
 \subsubsection{}\label{sssec: tensors}For a module $M$ over a ring $R$, we let $M^\otimes$ denote the direct sum of all $R$-modules obtained from $M$ by taking duals, tensor products, symmetric and exterior products. Let $(s_\alpha)\in V_{\bbZ_{(p)}}$ be a collection of tensors whose stabilizer is $\calG_{\bbZ_{(p)}}$; such a collection exists by \cite[Lemma 1.3.2]{Ki2}. We let $\calA\rightarrow \scrS_{\rmK}$ denote the pullback of the universal scheme, and let $V_B:=R^1 h_{\mathrm{an},*}\bbZ_{(p)}$, where $h_{\mathrm{an}}$  is the map of complex analytic spaces associated to $h$.  Since the tensors $s_\alpha$ are $\bfG$-invariant, they give rise to sections $s_{\alpha,B}\in V_{B}^\otimes$. For a finite prime $\ell\neq p$, we set $\calV_\ell=R^1h_{\et *}\bbQ_\ell$ and $\calV_p:=R^1h_{\eta,\et*}\bbZ_p$ where $h_{\eta}$ is the generic fiber of $h$.  Then by the comparison between Betti and \'etale cohomology, we obtain tensors $s_{\alpha,\et}\in \calV_p^\otimes$ and $s_{\alpha,\ell}\in \calV_\ell^\otimes$ for $\ell\neq p$ (cf. \cite[\S4.1.6]{KZ}). 
 
 For $T$ an $\calO_E$-scheme and $x\in \scrS_{\rmK}(T)$, we let $\calA_x$ for the pullback of $\calA$ to $x$ and $s_{\alpha,\ell,x}$ the pullback of $s_{\alpha,\ell}$. Similarly for $T$ an $\bfE$-scheme (resp. a $\bbC$-scheme) and $x\in\scrS_{\rmK}(T)$, we write $s_{\alpha,\et,x}$ (resp. $s_{\alpha,B,x}$) for the pullback of $s_{\alpha,\et}$ (resp. $s_{\alpha,B}$) to $x$. If $\xbar\in \scrS_{\rmK}(k)$, we let $\bbD$ denote the Dieudonn\'e module of $\calA_{\xbar}[p^\infty]$. We obtain tensors $s_{\alpha,0,{\xbar}}\in \bbD^\otimes[1/p]$ corresponding to $s_{\alpha,\et,\tilde{x}}\in T_p\calA_{\tilde{x}}^\otimes$ under the $p$-adic comparison isomorphism, where $\tilde{x}\in\scrS_{\rmK}(\calO_K)$ is a lift of ${\xbar}$ for $K/\brQ$ a finite extension. By \cite[Proposition 1.3.7]{KMS}, the $s_{\alpha,0,{\xbar}}$ are independent of the choice of lifting $\tilde{x}$.
\subsubsection{}
In order to the apply the results in \cite{PRshtukas}, we need to impose some further conditions on the compact open $\rmK_p$ and the Hodge embedding $\iota$. We let $\calB(G,\bbQ_p)$ be the extended Bruhat--Tits building for $G$ over $\bbQ_p$. 
For a point $x\in \calB(G,\bbQ_p)$ write $\calG=\calG_x$ for the associated Bruhat--Tits stabilizer scheme constructed in \cite{BT2}; this is a smooth group scheme over $\bbZ_p$ with generic fiber $G,$ which is characterized by the property that $\calG(\brZ)$ is the stabilizer of the point $x\in \calB(G,\brQ)$.  We also write $\calG^\circ=\calG_x^\circ$ for the associated parahoric group scheme. 

We now assume that the following conditions are satisfied.
\begin{enumerate}
	\item 
The Shimura datum $(\bfG,X)$ is of Hodge type.
\item
 $\rmK_p=\calG(\bbZ_p)$ where $\calG=\calG_x$ is a Bruhat--Tits stabilizer scheme with connected special fiber, i.e.   $\calG=\calG^\circ$.
\end{enumerate}
In other words,   the tuple $(p,\bfG,X,\rmK)$ is of \emph{global Hodge type} in the sense of \cite[Definition 4.5.1]{PRshtukas}.

We will assume further that the following two conditions are satisfied, which simplifies some of the discussion and will be sufficient for our purposes.
  
 \begin{enumerate}
 \item[(3)] The centralizer of a maximal $\brQ$-split torus in $G$ is $R$-smooth (cf. \cite[Definition 2.4.3]{KZ})
 
 \item[(4)] If $p=2$, then the relative root system of $G_{\brQ}$ is reduced.
\end{enumerate}
\begin{definition}\label{def: strongly admissible}
Let $(\bfG,X)$ be a Shimura datum and $\calG$ a stabilizer group scheme for $G$. We say the triple $(\bfG,X,\calG)$ is \emph{strongly admissible}, if it satisfies the conditions (1)--(4) above. 
\end{definition}

\subsubsection{} Under the assumption of strong admissibility, the pair $(\bfG,X)$ admits Hodge embeddings with desirable properties.
 
 \begin{lemma}\label{lem: embedding Weil restriction}
 	Let $K$ be a splitting field for $G$. Then under assumptions (3) and (4) above, the natural morphism $G\rightarrow \Res_{K/\bbQ_p}G_K$ extends to a closed immersion $$\calG\rightarrow \mathrm{Res}_{\calO_K/\bbZ_p}\calG_K.$$
 	Here $\calG_K$ is the parahoric group scheme for $G_K$ corresponding to the image of $x$ in the building $\calB(G,K)$ for $G$ over $K$.
 \end{lemma}
 
 \begin{proof}The case $p>2$ is \cite[Proposition 2.4.11]{KZ}. When $p=2$, the result follows from the same proof, noting that assumption (4) implies that the groups $G_\alpha$ appearing in the proof of \cite[Proposition 2.4.11]{KZ} are isomorphic to $\Res_{L/\bbQ_p}\SL_2$.
 	\end{proof}

 \begin{prop}\label{prop: integral LHE}Let $(\bfG,X,\calG)$ be a strongly admissible triple.  There exists a Hodge embedding $\iota:(\bfG,X)\rightarrow (\mathbf{GSp}(V),S^{\pm})$ such that the induced map $G\rightarrow \GL(V_{\bbQ_p})$ extends to a closed immersion $\calG\rightarrow \GL(\Lambda)$, where $\Lambda\subset V_{\bbQ_p}$ is a self-dual $\bbZ_p$-lattice.
 	\end{prop}
 \begin{proof}We may argue as in  \cite[Proposition 3.1.9, Lemma 4.1.4]{KZ} using Lemma \ref{lem: embedding Weil restriction} in the case $p=2$. Explicitly, fix $(\bfG,X)\rightarrow (\mathbf{GSp}(W),S'^\pm)$  any Hodge embedding and let $\rmF$ be a totally real field such that $G$ splits over the completion of $\rmF$ at all $p$-adic places. Then we let $V=W\otimes_{\bbQ}\rmF$ considered as a $\bbQ$-vector space equipped with the alternating bilinear form $\psi:=\mathrm{Tr}_{\rmF/\bbQ}\circ (\psi'\otimes\rmF)$, where $\psi'$ is the alternating form on $W$. We  thus obtain a new Hodge embedding $\bfG\rightarrow \mathbf{GSp}(V).$
Over  $\bbQ_p$, the morphism to $\GL(V_{\bbQ_p})$ factors as \begin{equation}\label{eqn: local embedding}G\rightarrow \prod_{i=1}^r\Res_{F_i/\bbQ_p}G_{F_i}\rightarrow\prod_{i=1}^r \Res_{F_i/\bbQ_p}\GL(W_{F_i})\rightarrow \GL(V_{\bbQ_p});\end{equation} here the $F_i$ are the completions of $\rmF$ at $p$-adic places, $G_{F_i}$ is a split reductive group over $F_i$, and the last map is induced by restriction of structure. By \cite[Proposition 1.3.3]{KP}, the morphism $G_{F_i}\rightarrow \GL(W_{F_i})$ extends to a closed immersion $\calG_{F_i}\rightarrow \GL(\calL_i)$ over $\calO_{F_i}$, where $\GL(\calL_i)$ is the parahoric of $\GL(V_{F_i})$ corresponding an $\calO_{F_i}$-lattice chain $\calL_i$. Upon replacing $V$ and $\psi$ by a direct sum and each $\calL_i$ by a sum of $\mathrm{tot}(\calL_i)$ in the sense of \cite[\S2.3]{KPZ}, we find that \eqref{eqn: local embedding} extends to  a closed immersion \begin{equation}\label{eqn: local embedding integral}\calG\rightarrow \prod_{i=1}^r\Res_{F_i/\bbQ_p}\calG_{F_i}\rightarrow\prod_{i=1}^r \Res_{F_i/\bbQ_p}\GL(\Lambda_i)\rightarrow \GL(\Lambda).\end{equation} 
Here, the first map is a closed immersion by Lemma \ref{lem: embedding Weil restriction}, the second map arises from Weil restriction of a closed immersion $\calG_{F_i}\rightarrow \GL(\Lambda_i)$, and the last map is a closed immersion since it is induced by restriction of structure.

Finally, to ensure that $\Lambda$ is self-dual, we apply Zarhin's trick. Thus we replace $V$ with $V^8$ equipped with the alternating form as in \cite[\S4.5.9]{Keerthi-v4} and $\Lambda$ by the self-dual  lattice $\Lambda^4\oplus\Lambda^{\vee,4}$. The induced Hodge embedding $(\bfG,X)\rightarrow (\mathbf{GSp}(V),S^{\pm})$  then extends to a closed immersion $\calG\rightarrow \GL(\Lambda)$ as desired. 
\end{proof}

If $(\bfG,X,\calG)$ is a strongly admissible triple and  $\iota:(\bfG,X)\rightarrow (\mathbf{GSp}(V),S^\pm)$ is  Hodge embedding as in Proposition \ref{prop: integral LHE}, then $\calG$ coincides with Zariski closure $\bar \calG$ of its generic fiber in $\GL(\Lambda)$ and hence we have the integral model $\scrS_{\rmK}$  constructed in \S\ref{sssec: construction integral models}.

\subsubsection{}
In order to state the main results concerning $\scrS_{\rmK}$ from \cite{PRshtukas}, we recall some notions concerning $p$-adic shtukas from \cite{SW2}, \cite{PRshtukas}. Let $S=\Spa(R,R^+)$ be an affinoid perfectoid space over $k$ and $S^\sharp=\Spa(R^\sharp,R^{\sharp,+})$ an untilt over $\calO_{\breve E}$. Recall that a $\calG$-\emph{shtuka over} $S$ with one leg at the untilt $S^\sharp$ consists of a pair $(\scrP,\phi_{\scrP})$, where 

$\bullet$ $\scrP$ is a $\calG$-torsor over the analytic adic space $S\times \bbZ_p=\Spa(W(R^+))\setminus\{[\varpi]=0\}$ 	

 \ \ ($\varpi\in R^+$ is a pseudouniformizer).

$\bullet $ $\phi_\scrP$ is an isomorphism of $\calG$-torsors
	$$\phi_{\scrP}:\mathrm{Frob}_S^*(\scrP)|_{S\times\bbZ_p\setminus S^\sharp}\xrightarrow{\sim}\scrP|_{S\times\bbZ_p\setminus S^\sharp},$$ 
	
	\ \ which is meromophic along the Cartier divisor $S^\sharp\subset S\times\bbZ_p$.

We say that the leg is bounded by $\mu$ if the relative positive of  $\phi_{\scrP}(\mathrm{Frob}_S^*(\scrP))$ and $\scrP$ along $S^\sharp$ is bounded by $\mu$ (cf. \cite[Definition 20.3.5]{SW2}). 
More generally, if $\calF$ is a $v$-sheaf over $\Spd \calO_E$, a $\calG$-shtuka over $\calF$ with leg bounded by $\mu$, is a functorial rule which assigns to a point $x$ of $\calF$ valued in an affinoid perfectoid space $S$, a $\calG$-shtuka over $S$ with leg at the until $S^\sharp$ corresponding to the morphism $S\xrightarrow{x}\calF\rightarrow \Spd\calO_E$ which is bounded by $\mu$.

Now let $\brQ=W(k)[1/p]$ and let $B(G)$ denote the set of $\sigma$-conjugacy classes in $G(\brQ)$. Let $B(G,\mu^{-1})$ be  the neutral acceptable set, i.e. the set of $b\in B(G)$ with $\nu_{b}\leq \overline{\mu^{-1}}_{\mathrm{dom}}$ and $\kappa(b)=-\mu^\sharp$ in the notation of \cite[\S2.2]{RV}. 
For $b\in B(G,\mu^{-1}),$ the triple $(G,b,\mu)$ is a local Shimura datum in the sense of \cite[Definition 24.1.1]{SW2}. 

We let $\calM_{\calG,b,\mu}^{\mathrm{int}}$ denote the integral model for the local Shimura variety $\mathrm{Sht}_{G,b,\mu,\rmK_p}$ given by \cite[Definition 25.1.1]{SW2}, cf. also \cite[Definition 3.2.1]{PRshtukas}. Explicitly, for $S=\Spa(R,R^+)$  affinoid perfectoid, an $S$-point of $\calM^{\mathrm{int}}_{\calG,b,\mu}$ is an isomorphism class of quadruples $$(S^\sharp,\scrP,\phi_{\scrP},i_r)$$ where $S^\sharp=\Spa(R^\sharp,R^{\sharp+})$ is an untilt of $S$, $(\scrP,\phi_{\scrP})$ is a $\calG$-shtuka over $S$  with one leg at the untilt $S^\sharp$  bounded by $\mu$, and $i_r$ is a framing, i.e. an isomorphism of $G$-torsors $$i_r:G_{\calY_{[r,\infty)}}\xrightarrow{\sim}\scrP_{\calY_{[r,\infty)}}$$
for $r\in \bbQ$ sufficiently large, under which $\phi_{\scrP}$ corresponds to $b\times \mathrm{Frob}_S$. Here $$\calY_{[r,\infty)}\subset S\times \bbZ_p=\Spa(W(R^+))\setminus\{[\varpi]=0\}$$   is given by $\calY_{[r,\infty)}=\{|\cdot|:|p^r|\leq |[\varpi]|\}$.

 \subsubsection{ }\label{sssec: construction tensors}We return to the setting of \S\ref{sssec: construction integral models}, but we now assume  that $(\bfG,X,\calG)$ is strongly admissible and that the Hodge embedding is chosen as in Proposition  \ref{prop: integral LHE}. In particular, $\calG$ is equal to the Zariski closure $\bar\calG$ of its generic fiber in $\GL(\Lambda)$.

 We let $\bbP_{\rmK}$ denote the  $\calG(\bbZ_p)$ local system over $\Sh_{\rmK}(\bfG,X)$ given by $$\bbP_{\rmK}=\mathrm{Isom}_{s_\alpha,s_{\alpha,\et}}(\Lambda,\calV_p),$$ i.e. trivializations of $\calV_p$ which respect the tensors. The  following is proved in \cite[\S4.5--4.7]{PRshtukas}, which shows that these models satisfy \cite[Conjecture 4.2.2]{PRshtukas}.  

\begin{thm}[{\cite[Theorem 4.5.2]{PRshtukas}}]\label{prop: PR shtukas}Let $(\bfG,X,\calG)$ be a strongly admissible triple.
 The models $\{\scrS_{\rmK}\}_{\rmK^p\subset \bfG(\bbA_f^p)}$ satisfy the following properties.
 \begin{enumerate}
 	\item For $R$ a  DVR of mixed characteristic $(0,p)$ over $\calO_E$,  we have a bijection $$\Sh_{\rmK_p}(\bfG,X)(R[1/p])=\scrS_{\rmK_p}(R).$$
 	\item There is a $\calG$-shtuka $\scrE_{\rmK}$  bounded by $\mu$ on $\scrS_{\rmK}$ which extends the shtuka $\scrE_{\rmK,E}$ on $\Sh_{\rmK}(\bfG,X)_E$ associated to the pro-\'etale $\calG(\bbZ_p)$-local system $\bbP_K$ via \cite[Proposition 2.5.3]{PRshtukas}.
 	\item Assume $\rmK^p$ is neat. For $\xbar\in \scrS_{\rmK}(\bfG,X)(k)$, there is an isomorphism of $v$-sheaf completions:
 	$$\Theta_{\xbar}:\widehat{\calM^{\mathrm{int}}_{\calG,b_{\xbar},\mu}}_{/{\xbar}_0}\xrightarrow{\sim}(\widehat{\scrS_{\rmK_{/{\xbar}}}})^\lozenge$$ such that $\Theta_{\xbar}^*(\scrE_{\rmK})$ is the tautological shtuka on $\calM^{\mathrm{int}}_{\calG,b_{\xbar},\mu}$. Here $b_{\xbar}\in G(\brQ)$ is the $\mu$-admissible element corresponding to the $\calG$-shtuka ${\xbar}^*\scrE_{\rmK}$ over $\Spec k$  and ${\xbar}_0$ is the base-point of $\calM^{\mathrm{int}}_{\calG,b_{\xbar},\mu}$.
 \end{enumerate}\qed
  	\end{thm}
 \subsubsection{} 
 We recall the construction of the $\calG$-shtuka $\scrE_{\rmK}$ and the isomorphism $\Theta_{x}$ following \cite[\S4.6-4.7]{PRshtukas}. We assume $\rmK^p$ is neat.

By definition, the extension $\scrE_{\rmK}$ of $\scrE_{\rmK,E}$ is a $\calG$-shtuka on the $v$-sheaf  $\widehat{\scrS_{\rmK}}^\lozenge$ associated to $\widehat{\scrS_{\rmK}}$, compatible with $\scrE_{\rmK,E}$ on the generic fiber. We first extend the vector shtuka \footnote{By a vector shtuka, we just  mean a $\GL(V)$-shtuka.}  $(\scrV_E,\phi_E)$ associated to $\calV_p$ to a vector shtuka$(\scrV,\phi)$  
on $\widehat{\scrS_{\rmK}}^\lozenge$; this is induced by  the BKF-module associated to the pullback of the universal $p$-divisible group $\calA[p^\infty]$. Viewing the tensors $s_{\alpha,\et}$ as morphisms $\mathbf{1}\rightarrow \calV_p^\otimes$ and applying the functor from \cite[Proposition 2.5.2]{PRshtukas}, we obtain tensors $s_{\alpha,0,E}\in \scrV_E^\otimes$, which extend uniquely to tensors $s_{\alpha,0}\in \scrV^\otimes$ by \cite[Theorem 2.7.7]{PRshtukas}. We then set 
$$\scrE=\mathrm{Isom}_{s_{\alpha},s_{\alpha,0}}(\Lambda,\scrV)$$ which is a $\calG$-torsor by the argument in \cite[\S4.6]{PRshtukas}. A key input here is the main result of \cite{Anschutz} on extensions of parahoric torsors. The Frobenius on $\scrV$ preserves $s_{\alpha,0}$ and this gives $(\scrE,\phi)$ the structure of a $\calG$-shtuka with one leg bounded by $\mu$.

 \subsubsection{} 
Recall $k$ was a fixed algebraic closure of $k_E$. For $\xbar\in \scrS_{\rmK}(k)$, we set $\scrG_{\xbar}=\calA_{\xbar}[p^\infty]$ and we let $\bbD:=\bbD(\scrG_{\xbar})(\brZ)$ be the Dieudonn\'e module of $\scrG_{\xbar}$. Then we have an identification between the Frobenius descent  of the linear dual of $\bbD$ and the BKF-module associated to ${\xbar}^*\scrV$. 
The  corresponding tensors $s_{\alpha,0,{\xbar}}\in \bbD^\otimes[1/p]$ (cf. \S\ref{sssec: tensors}) satisfy $s_{\alpha,0,{\xbar}}\in \bbD^\otimes$, and upon fixing a tensor preserving $\brZ$-linear bijection $\bbD\cong\Lambda^\vee\otimes_{\bbZ_p}\brZ$, which again exists by \cite{Anschutz}, the Frobenius on $\bbD$ is given by  $\sigma(b_{\xbar})\sigma$ for some  $b_{\xbar}\in\bigcup_{w\in \Adm_{\calG}(\mu^{-1})}\calG(\brZ)\dot{w}\calG(\brZ)$. Here $\Adm_{\calG}(\mu^{-1})$ denotes the admissible set in the double quotient $W_{\calG}\backslash \widetilde{W}/W_{\calG}$ of the Iwahori Weyl group $\widetilde{W}$ of $G$. Then the $\sigma$-conjugacy class of $b_{\xbar}$ lies in $B(G,\mu^{-1})$. We thus obtain a local Shimura datum $(G,b_{\xbar},\mu^{-1})$, and hence the integral local Shimura variety $\calM_{\calG,b_{\xbar},\mu}^{\mathrm{int}}$.

We write $\widehat{\scrS_{\rmK,{\xbar}}}=\Spf R_{\xbar}$. Let  $\scrG_{\xbar}^\wedge$  denote the pullback of $\calA]p^\infty]$ to $\Spf R_{\xbar}$ and $\scrE_{\rmK,{\xbar}}^\wedge$ the pullback of  $\scrE_{\rmK}$ to $\Spd (R_{\xbar})$. There is unique quasi-isogeny 
$$\scrG_{{\xbar},R_{\xbar}/pR_{\xbar}}\dashrightarrow\scrG^\wedge_{{\xbar},R_{\xbar}/pR_{\xbar}}$$ over $\Spf R_{\xbar}/pR_{\xbar}$ lifting the identity modulo the maximal ideal $\fkm\subset R_{\xbar}$. 

This induces a framing of the vector shtuka associated to $\scrG_{\xbar}^\wedge$, and the argument in \cite[Proof of Proposition 4.7.1]{PRshtukas} shows that this induces a framing of $\scrE^\wedge_{\rmK,\xbar}$. We thus obtain a morphism of $v$-sheaves $$\Psi_{G,{\xbar}}:\Spd(R_{\xbar})\rightarrow\widehat{\calM^{\mathrm{int}}_{\calG,b_{\xbar},\mu/{\xbar}_0}}$$ which is an isomorphism by the argument in \cite[\S4.7.1]{PRshtukas}. The isomorphism $\Theta_{\xbar}$ is then the inverse of $\Psi_{G,{\xbar}}$.

\subsection{Isogeny classes and CM lifts}
\subsubsection{} We keep the notations the previous subsection. Let $\scrS_{\rmK}$ be the integral model associated to a strongly admissible triple $(\bfG,X,\calG)$. Fix a point ${\xbar}\in \scrS_{\rmK}(k)$. We will use the above to obtain a description of the points of $\widehat{\scrS_{\rmK,{\xbar}}}$ valued in finite extensions of $\brQ$ analogous to \cite[Proposition 3.3.13]{KP}, cf. also \cite[Proposition 3.3.4 and Proposition 4.1.9]{KZ}. 

 The proof of the corresponding results in \cite{KP}, \cite{KZ} use as a key input, results on the deformation theory of p-divisible groups equipped with a collection crystalline tensors  developed in \cite{KP}, \cite{KPZ}, which is used to give a concrete description of the pullback of the universal $p$-divisible group to $\widehat{\scrS_{\rmK,{\xbar}}}$. In the approach taken in this paper, these results are replaced with Theorem \ref{prop: PR shtukas} (3). The main point is that after taking $v$-sheaves, we obtain a good description of the universal shtuka over $\widehat{\scrS_{\rmK,{\xbar}}}$. This allows us to prove the desired result (Proposition \ref{prop: G-adapted}) under the less restrictive assumption of strong admissibility considered here.

\subsubsection{} Fix  a tensor preserving $\brZ$-linear bijection \begin{equation}\label{eqn: tensor preserving bijection}\bbD\cong\Lambda^\vee\otimes_{\bbZ_p}\brZ.\end{equation} This gives an identification of $\calG_{\brZ}$ with the group of tensor preserving automorphisms of $\bbD$.

\begin{prop} \label{prop: G-adapted}Assume $(\bfG,X,\calG)$ is a strongly admissible triple and $\rmK^p$ is neat.
	Let $K/\brQ$ be finite and $\tilde{\scrG}$ a deformation of $\scrG_\xbar$ to $\calO_K$. Then $\tilde{\scrG}$ corresponds to a point $\tilde{x}\in\widehat{\scrS_{\rmK,{\xbar}}}(\calO_K)$  if and only if the following two conditions are satisfied.
	\begin{enumerate}
		\item The filtration on $\bbD\otimes_{\brZ}K$ corresponding to $\tilde{\scrG}$ is induced by a $G$-valued cocharacter conjugate to $\mu$.
		
		\item The tensors $s_{\alpha,0,\xbar}\in \bbD^\otimes$ correspond to tensors $s_{\alpha,\et}\in T_p\tilde{\scrG}^\otimes$ under the $p$-adic comparison isomorphism.
	\end{enumerate}
\end{prop}

\begin{proof}If $\tilde{\scrG}$ corresponds to $\tilde{x}\in \widehat{\scrS_{\rmK,{\xbar}}}(\calO_K)$, the same argument as \cite[Proposition 4.1.9]{KZ} shows that the two conditions above are satisfied so it suffices to show the converse.  Thus suppose $\tilde{\scrG}$ satisfies the two conditions.  We write $\widehat{\scrS_{\rmK,{\xbar}}}=\Spf R_{\xbar}$.
	
We write $H=\mathrm{GL}(V_{\bbQ_p})$ and $\calH=\mathrm{GL}(\Lambda)$ a parahoric of $H$, and we let $\iota':G\rightarrow H$ be the composition of $\iota$ and the inclusion $\GSp(V_{\bbQ_p})\rightarrow H$. We obtain a local Shimura datum $(H,\iota'(b_\xbar),\iota'(\mu))$.	Let $A_{\xbar}$ be the complete local ring at the base-point of the RZ-space $\mathrm{RZ}_{\calH,\iota'(b_\xbar),\iota'(\mu)}$ base-changed to $\calO_{\brE}$. By \cite[\S4.7.1]{PRshtukas}, we have a commutative diagram 
	\begin{equation}\label{eqn: comm diagram completions}\xymatrix{\Spd(R_{\xbar})\ar[r]^{\!\!\!\Psi_{G,{\xbar}}}\ar[d]&
	\widehat{\calM^{\mathrm{int}}_{\calG,b_\xbar,\mu}}_{/{\xbar}_0}\ar[d]
\\\Spd(A_{\xbar}) \ar[r]^{\!\!\!\!\!\!\!\!\!\!\!\!\!\!\!\!\!\!\!\!\!\!\!\!\!\!\!\!\!\!\!\!\!\!\!\!\!\!\!\!\!\Psi_{H,{\xbar}}} &\widehat{\calM^{\mathrm{int}}_
{\calH,\iota'(b_\xbar),\iota'(\mu)}}_{/{\xbar}_0}\otimes_{\Spd(\brZ)}\Spd(\calO_{\brE})
}\end{equation}
	where $\Psi_{G,{\xbar}}$  and $\Psi_{H,{\xbar}}$ are isomorphisms and  the vertical maps are closed immersions. More precisely, \cite{PRshtukas} constructs the diagram where the bottom row is replaced by the one for a symplectic group, but the same proof works here.

Let $C$ denote the completion  of an algebraic closure of $K$ and $C^\flat$ its tilt. The deformation $\tilde{\scrG}$ corresponds to a morphism $ A_{\xbar}\rightarrow \calO_K$. Composing this with $\calO_K\rightarrow \calO_C$, we obtain a $(C^\flat,\calO_{C^\flat})$-point $\tilde{x}$ of $\Spd A_{\xbar}$. We write $\tilde{x}':=\Psi_{H,{\xbar}}(\tilde{x})$ for the corresponding $(C^\flat,\calO_{C^\flat})$-point of 
$\widehat{\calM^{\mathrm{int}}_{\calH,\iota'(b_\xbar),\iota'(\mu)}}_{/{\xbar}_0}$. Then it suffices to show that $\tilde{x}'$ factors through 
$\widehat{\calM^{\mathrm{int}}_{\calG,b_{\bar x},\mu}}_{/{\xbar}_0}$. 
	
	By construction of $\Psi_{H,{\xbar}}$, the point $\tilde{x}'$ can be described as follows. By \cite[Theorem 17.5.2]{SW2} (cf. \cite[Example 2.3.4]{PRshtukas}), we can associate to $\tilde{\scrG}_{\calO_C}$ a vector shtuka $\calE'$ over $\Spa(C^\flat,\calO_{C^\flat})$ with leg at the untilt  $\Spa(C,\calO_C)$. 
	Let $$i:\scrG_{{\xbar},\calO_K/p}\dashrightarrow \tilde{\scrG}_{\calO_K/p}$$ denote the unique quasi-isogeny lifting the identity on $\calO_K/\fkm_K=k$. Fix $\pi\in \calO_K$ a uniformizer and let $\pi^\flat\in \calO_{C^\flat}$ be a pseudouniformizer given by choosing a compatible system of roots $(\pi^{1/p^n})_n$ of $\pi$. Then as explained in \cite[Proof of Proposition 4.7.1]{PRshtukas}, evaluating $i_{\calO_C/p}$ on the Dieudonn\'e crystal at $B^+_{\mathrm{crys}}(\calO_{C^\flat}/\pi_{\flat})$ gives a framing  $i_r$ of $\calE'$ on $\calY_{[r,\infty)}(C^\flat,\calO_{C^\flat})$ for $r>>0$. The point $\tilde{x}'$ corresponds to the pair $(\calE',i_r)$.
	
 Thus in order to show  $\tilde{x}'$ factors through $ \widehat{\calM^{\mathrm{int}}_{\calG,b_\xbar,\mu}}_{/x_0}$, it suffices to show $\calE'$ has a reduction to a $\calG$-shtuka $\calE$, and that $i_r$ induces a framing of $\calE$.
	
We let $\fkS:=\brZ[[u]]$ which we equip with the Frobenius $\sigma$ which acts as the  usual Frobenius on $\brZ$ and sends $u$ to $u^p$. By the argument in \cite[Lemma 3.3.5]{KP}, we have a $\calG$-torsor $$\scrP=\mathrm{Isom}_{s_{\alpha,\et},\tilde{s}_\alpha}(T_p\tilde{\scrG}^\vee\otimes_{\bbZ_p}\fkS,\fkM(T_p\tilde{\scrG}^\vee))$$ over $\fkS$. Here $\fkM$ is the functor in \cite[Theorem 1.2.1]{Ki2} and $\tilde{s}_\alpha\in \fkM(T_p\tilde{\scrG}^\vee)^\otimes$ are given by applying 
$\fkM$ to the morphisms $s_{\alpha,\et}:\mathbf{1}\rightarrow T_p\tilde{\scrG}^\vee$. Let $E[u]\in \brZ[u]$ be the Eisenstein polynomial for  $\pi$.  Then the Frobenius on $\fkM(T_p\tilde{\scrG}^\vee)$ induces an isomorphism $$\phi_{\fkM}:\sigma^*(\scrP)[1/E(u)]\rightarrow \scrP[1/E(u)],$$ which gives  $\scrP$ the structure of a $\calG$-BK module (cf. \cite{Pappascanonical}).
We extend the map $\brZ\rightarrow W(\calO_{C^\flat})$ to a map $h:\fkS\rightarrow W(\calO_{C^\flat})$ by sending $u$ to $[\pi^\flat]$. Then  base changing $\scrP$ along $h$, we obtain a $\calG$-BKF module over $W(\calO_{C^\flat})$ (cf. \cite[Proof of Proposition 3.5.1]{PRshtukas}) and hence a $\calG$-shtuka $\calE$ over $\Spa(C^\flat,\calO_{C^\flat})$. As in \cite[Proposition 3.5.1]{PRshtukas}, the push-out of $\calE$ along $\calG\rightarrow \calH$ is the vector shtuka $\calE'$. To conclude, it suffices to show that the framing $i_r$ preserves the $\calG$-structures.

Let $S$ denote the $p$-adic completion of $\fkS[E(u)^i/i!]_{i\geq 1}$, which we equip with the Frobenius $u\mapsto u^p$. Since $S\rightarrow \calO_K/p$ is equipped with divided powers, we may evaluate the quasi-isogeny $i$ on $S$, and hence we obtain a Frobenius equivariant isomorphism  $$\lambda:\bbD\otimes_{\brZ}S[1/p]\xrightarrow{\sim} \bbD(\tilde{\scrG}_{\calO_K/p})(S)[1/p]=\fkM(T_p\tilde{\scrG}^\vee)\otimes_{\fkS,\sigma}S[1/p]$$ which lifts the identity on $\bbD$.

Since the morphism $\fkS\rightarrow B^+_{\mathrm{crys}}(\calO_{C^\flat}/\pi^\flat)$ factors through $S[1/p]$, it suffices to show $\lambda$ takes $s_{\alpha,0,\xbar}$ to $\tilde{s}_\alpha$. Set $s=s_{\alpha,0,\xbar}$ and $s'=\lambda^{-1}\tilde{s}_\alpha$. Then since $\tilde{s}_\alpha$ lifts $s_{\alpha,0,\xbar} $, we have $s-s'\in u\bbD^\otimes_{S[1/p]}$. Since $s_\alpha$ and $\tilde{s}_\alpha$ are Frobenius invariant, and $\lambda$ is Frobenius equivariant, we have $$s-s'=\phi^j(s-s')\in u^{p^j}\bbD^\otimes_{S[1/p]}$$ for all $j\geq 1$, and hence $s=s'$.
\end{proof}

\subsubsection{}Let $\xbar\in \scrS_{\rmK}(k)$ with associated abelian variety $\calA_{\xbar}$.  For $\ell\neq p$ a prime, let $\calV_\ell\calA_{\xbar}^\otimes$ denote the rational $\ell$-adic Tate module of $\calA_{\xbar}$. Then the abelian variety $\calA_{\xbar}$ is equipped with tensors $s_{\alpha,0,{\xbar}}\in \bbD^\otimes$ and $s_{\alpha,\ell,\xbar}\in \calV_\ell\calA_{\xbar}^\otimes$. The \emph{isogeny class} $\scrI_{\xbar}$ of ${\xbar}$ is the set of ${\xbar}'\in \scrS_{\rmK}(k)$ such that there is a quasi-isogeny $\calA_{\xbar}\rightarrow \calA_{{\xbar}'}$ taking $s_{\alpha,0,{\xbar}}$ to $s_{\alpha,0,{\xbar}'}$ and $s_{\alpha,\ell,{\xbar}}$ to $s_{\alpha,\ell,{\xbar}'}$ for all $\ell\neq p$.

We let  $\mathrm{Aut}_{\bbQ}(\calA_\xbar)$ denote the $\bbQ$-group attached to the units in the endomorphism algebra $\mathrm{End}_{\bbQ}(\calA_{\xbar})\otimes_{\bbZ}\bbQ$. We let $I_{\xbar}\subset \mathrm{Aut}_{\bbQ}(\calA_{\xbar})$ be the subgroup which fixes the tensors $s_{\alpha,0,{\xbar}}$ and $s_{\alpha,\ell,{\xbar}}$ for $\ell\neq p$.  If ${\xbar}'\in \scrI_{\xbar}$, the quasi-isogeny $\calA_{\xbar}\rightarrow \calA_{{\xbar}'}$ induces a natural isomorphism $I_{{\xbar}}\cong I_{{\xbar}'}$.

We fix a $\brQ$-linear bijection  \begin{equation}\label{eqn: Q_p linear bijection}
\bbD\otimes_{\brZ}\brQ\cong V^{\vee}\otimes_{\bbQ_p}\brQ
\end{equation}taking $s_{\alpha,0,{\xbar}}$ to $s_{\alpha}$ (for example, we can take the one induced by a $\brZ$-linear bijection $\bbD\cong \Lambda^{\vee}\otimes_{\bbZ_p}\brZ$) so that we obtain an element $b_{\xbar}\in G(\brQ)$ giving rise to the Frobenius on $\bbD$. 
The following description of the set $\scrI_{{\xbar}}$ is proved in \cite{GLX}, cf. also \cite[\S4.10]{PRshtukas}. By the Bruhat decomposition, we have a bijection $$\mathrm{inv}_{{\calG}}:\calG(\brZ)\backslash G(\brQ)/\calG(\brZ)\rightarrow W_{\calG}\backslash \widetilde{W}/W_{\calG},$$ and we let 
$X_{\calG}(\mu^{-1},b_{\xbar})$ denote the affine Deligne--Lusztig set $$\{g\in G(\brQ)/\calG(\brZ)| 
\mathrm{inv}_{\calG}(g^{-1}b_{\xbar}\sigma(g))\in \Adm_{\calG}(\mu^{-1})\}.$$

\begin{thm}[{\cite[Corollary 6.3]{GLX}}]\label{thm:GLX}
	There is a natural map $$i_{\xbar}:X_{\calG}(\mu^{-1},b_{\xbar})\rightarrow \scrI_{\xbar}$$satisfying $\bbD(\scrG_{i_{\xbar}(g)})=\sigma(g)\bbD(\scrG_{\xbar})$, which induces a bijection
	$$\scrI_{\xbar}\cong I_{\xbar}(\bbQ)\backslash X_{\calG}(\mu^{-1},b_{\xbar})\times \bfG(\bbA_f^p)/\rmK^p.$$\qed
\end{thm}

\subsubsection{}	
The action of $I_{\xbar}$ on $\bbD$ gives rise to a natural inclusion $I_{{\xbar},\bbQ_p}\subset J_{b_{\xbar}}$, where $J_{b_{\xbar}}$ is the $\sigma$-centralizer group. By definition, for $R$ a $\bbQ_p$-algebra, we have $$J_{b_{\xbar}}(R)=\{g\in G(\brQ\otimes_{\bbQ_p}R)|g^{-1}b_{\xbar}\sigma(g)=b_{\xbar}\}.$$

 The Newton cocharacter $\nu_{b_{\xbar}}:\bbD\rightarrow G_{\brQ}$ satisfies $b_{\xbar}\sigma(\nu_{b_{\xbar}})b_{\xbar}^{-1}$ and hence induces a fractional cocharacter of $J_{b_{\xbar}}$ which is central. By \cite[Theorem 6]{KMS},  $I_{{\xbar},\bbQ_p}$ and $G$ have the same rank; thus $\nu_{b_{\xbar}}$ induces a central cocharacter of $I_{{\xbar},\bbQ_p}$.

\begin{thm}\label{thm: CM lift}
Let $(\bfG,X,\rmK)$ be  a strongly admissible triple and assume $G=\bfG_{\bbQ_p}$ is quasi-split and $\rmK^p$ is neat.	Let ${\xbar}\in \scrS_{\rmK}(k)$ and  $T\subset I_{\xbar}$ be a maximal torus. There exists ${\xbar}'\in \scrI_{\xbar}$ and $\tilde{x}'\in\scrS_{\rmK}(\calO_K)$ lifting ${\xbar}'$, for $K/\bbQ_p$  finite, such that the action of $T$ lifts to $\calA_{\tilde{x}'}$. Moreover $\tilde{x}'$ is a special point.
\end{thm}
\begin{remark} A version of this theorem is proved in \cite[Corollary 1.4]{GLX} under the additional  assumptions that $p>2$, $p\nmid|\pi_1(G^{\der})|$ and  $G$ tamely ramified. The proof in \emph{loc. cit.} appeals to \cite[Theorem 9.5]{Z}, which applies once we know the existence of $i_{\xbar}$. The main point here is that the arguments in \cite{Z}, which is based on \cite{Ki3}, can be applied in this more general general setting, once Proposition \ref{prop: G-adapted} is known.
\end{remark}

\begin{proof}The proof follows from the same argument as \cite[Theorem 2.2.3]{Ki3} (cf. also \cite[Theorem 9.5]{Z}), 
using Theorem \ref{thm:GLX}. We recall the argument for the convenience of the reader. 
	
Since $G$ is quasi-split, the conjugacy class of the Newton cocharacter $\nu_{b_\xbar}$ contains a representative $\overline{\nu}_{b_{\xbar}}$ which is defined over $\bbQ_p$. We write $M_{[b_{\xbar}]}$ for the centralizer of $\overline{\nu}_{b_{\xbar}}$, which is a Levi subgroup of $G$ defined over $\bbQ_p$ and  is equipped with an inner twisting $J_{b_{\xbar},\bar{\bbQ}_p}\cong M_{[b_{\xbar}],\bar{\bbQ}_p}$. Upon modifying the isomorphism \eqref{eqn: Q_p linear bijection}, we may assume $\nu_{b_{\xbar}}=\overline{\nu}_{b_{\xbar}}$.

	By \cite[Lemma 2.1]{Langlands}, there is an embedding $j:T\rightarrow M_{[b_{\xbar}]}$ over $\bbQ_p$ which is $M_{[b_{\xbar}]}$-conjugate to the embedding \begin{equation}\label{eqn: transfer}T_{\bar{\bbQ}_p}\rightarrow I_{{\xbar},\bar{\bbQ}_p}\rightarrow J_{b_{\xbar},\bar{\bbQ}_p}\xrightarrow{\sim} M_{[b_{\xbar}],\bar{\bbQ}_p}.\end{equation}
	
By \cite[Corollary 1.1.17]{KMS} (cf. also \cite[Lemma 9.2]{Z}), there exists a cocharacter  $\mu_T\in X_*(T)$ such that the following two conditions are satisfied:
	\begin{enumerate}
		\item   $j\circ\mu_T$  lies in the $G$-conjugacy class $\{\mu\}$.
		\item  $\nu_{b_{\xbar}}=\overline{\mu}_T\in X_*(T)_{\bbQ}$; here $\overline{\mu}_T$ denotes the Galois average of the cocharacter $\mu_T$ in $X_*(T)_{\bbQ}$. 
	\end{enumerate} 
	
The scheme of elements in $M_{[b_{\xbar}]}$ conjugating $j$ to \eqref{eqn: transfer} is a torsor for $T$, hence is trivial over $\brQ$ by
Steinberg's theorem. Thus upon further  modifying the isomorphism \eqref{eqn: Q_p linear bijection}  by an element $m\in M_{[b_{\xbar}]}(\brQ)$, we may assume $j$ agrees with the embedding \eqref{eqn: transfer}. 

Now let $T'\subset T$ be the maximal $\bbQ_p$-split torus. Then by definition of $J_{b_{\xbar}}$, $b_{\xbar}$ commutes with $j(T')$, hence lies in the centralizer $M$ of $j(T')$. Since $\overline\nu_{b_{\xbar}}$ is $\bbQ_p$-rational, it factors through $T'$ and hence is central in $M$. It follows that $b_{\xbar}$ is a basic element of $M$.

Note that $T$ is an elliptic torus of $M$; thus by \cite{Ko1}, we may modify \eqref{eqn: Q_p linear bijection}  by an element of $M(\brQ)$ so that $b_{\xbar}$ lies in $j(T(\brQ))$. The pair $(b_{\xbar},\mu_T)$ is then an admissible pair for the torus $T$ in the sense of \cite[\S1.1.5]{KMS}, and hence by \cite[Lemma 2.1]{RZ},
 the cocharacter $j\circ\mu_T$ induces an admissible filtration on   $\bbD\otimes_{\brZ}K'$, where $K'/\brQ$ is a finite extension over which $\mu_T$ is defined. By \cite[2.2.6]{Ki1}, this corresponds to a $p$-divisible group $\tilde{\scrG}'$ over $\calO_{K'}$. We let $s_\alpha'\in (T_p\tilde{\scrG'}^\vee\otimes_{\bbZ_p}\bbQ_p)^\otimes$ denote the image of $s_{\alpha,0,{\xbar}}\in( \bbD\otimes_{\brZ}\brQ)^\otimes\cong (\bbD(\tilde{\scrG}_{{\xbar},k})\otimes_{\brZ}\brQ)^\otimes$ under the $p$-adic comparison isomorphism. 
 
 Let $\tilde{x}\in \widehat{\scrS_{\rmK,{\xbar}}}(\calO_K)$ be any lifting of ${\xbar}$. Then by Proposition \ref{prop: G-adapted} the filtration on $\bbD\otimes_{\brZ}K$ corresponding to $\scrG_{\tilde{x}}$ is induced by a $G$-valued cocharacter conjugate to $\mu_T$. Thus by \cite[Corollary 4.5.3]{Wi}, there is a $\bbQ_p$-linear bijection \begin{equation}\label{eqn: Qp linear}T_p\tilde{\scrG}'\otimes_{\bbZ_p}\bbQ_p\rightarrow T_p\scrG_{\tilde{x}}\otimes_{\bbZ_p}\bbQ_p\end{equation} taking  $s_{\alpha}'$ to $s_{\alpha,\et,\tilde{x}}$. As in \cite[Proposition 1.1.19]{Ki3}, upon replacing $K'$ by a finite extension and $\tilde{\scrG}'$ by an isogenous $p$-divisible group, we may assume \eqref{eqn: Qp linear} is induced by  a $\bbZ_p$-linear bijection $ T_p\tilde{\scrG}'\rightarrow T_p\scrG_{\tilde{x}}$; in particular we have $s_{\alpha}'\in T_p\tilde{\scrG}'^{\vee,\otimes}$.  We thus obtain a representation $$\rho:\Gal(\bar{K'}/K')\rightarrow \calG(\bbZ_p).$$  By the discussion in \cite[\S2.6.2]{PRshtukas}, we can associate to  $\rho$ a $\calG$-shtuka $\scrP$ over $\Spd K$. The specialization of this sthuka $\scrP_0$ to the residue field $k$ is the Frobenius descent of the $\calG$-torsor of tensor preserving trivializations of $\bbD(\tilde{\scrG}'_{k})(\brZ)$, cf \cite[\S3.5]{PRshtukas}. By \cite[Proposition 3.5.1]{PRshtukas}, we have $$\bbD(\tilde{\scrG}'_{k})(\brZ)=\sigma(g)\bbD\subset \bbD\otimes_{\brZ}\brQ$$ where $g\in X_{\calG}(\mu^{-1},b_{\xbar})$.
	 Thus upon replacing ${\xbar}$ by $i_{\xbar}(g)\in\scrI_{\xbar}$, we may assume there is a lift $\tilde{\scrG}$ of $\scrG_{\xbar}$ to $\calO_K$ such that the filtration on $\bbD\otimes_{\brZ}K$ corresponding to $\tilde{\scrG}$ is  induced by $\mu_T$, and $s_{\alpha,0}$ correspond to tensors $s_{\alpha,\et}\in T_p\tilde{\scrG}^\otimes$. Since $\mu_T$ is conjugate to $\mu$, $\tilde{\scrG}$ corresponds to a point $\tilde{x}\in \widehat{\scrS_{\rmK,{\xbar}}}(\calO_K)$ by Proposition \ref{prop: G-adapted}. That $\tilde{x}$ is a special point then follows from the same proof as \cite[Prop. 2.2.3]{Ki3}. Namely, since $\mu_T$ factors through $T$, the action of $T$ on $\bbD(\tilde{\scrG}_{k})$ preserves the filtration and hence lifts to an action (in the isogeny category) on $\tilde{\scrG}$. It follows that the action of $T\subset I_{\xbar}$ lifts to an action on $\calA_{\tilde{x}}$. Since $T$  fixes $s_{\alpha,0,{\xbar}}$, it fixes $s_{\alpha,\et,\tilde{x}}$ and hence also $s_{\alpha,B,{\tilde{x}}}$, so that we obtain a natural embedding $T\subset \bfG$, and $T$ is a maximal torus in $\bfG$.
	The Mumford--Tate group of $\calA_{\tilde{x}}$ is a subgroup of $\bfG$ which commutes with $T$ and hence is contained in $T$.
	\end{proof}

\subsection{Independence of $\ell$}
\subsubsection{}\label{sssec: l-indep SV interior}Let $(\bfG,X)$ be a Shimura datum of Hodge type and $\scrS_{\rmK}$ an integral model  over $\calO_E$ for $\Sh_{\rmK}(\bfG,X)$ as in \S\ref{sssec: construction integral models}. 	Let $\ell\neq p$ be a prime, and assume $\rmK$ is of the form $\rmK_\ell\rmK^\ell$, with $\rmK_\ell\subset \bfG(\bbQ_\ell)$, $\rmK^\ell\subset \bfG(\bbA_f^\ell)$. We let $\widetilde{\bbL}_\ell$ denote the $\bfG(\bbQ_\ell)$-local system (lisse sheaf) arising from the pro-\'etale covering $\varprojlim_{\rmK_\ell'\rmK^\ell}\scrS_{\rmK'_\ell\rmK^\ell}$ of $\scrS_{\rmK}$, and we let $\bbL_\ell$ be the induced local system on the special fiber $\scrS_{\rmK,k_E}$. 

Similarly, for $\ell=p$, there is an  $F$-isocrystal with $G$-structure  $\calE$ on $\scrS_{\rmK,k_E}$ (cf. \cite[Corollary 1.3.13]{KMS}). Explicitly, we let $\calD$ denote the $F$-isocrystal corresponding to the crystalline realization of the abelian variety $\calA$ over $\scrS_{\rmK,k_E}$. By \cite[Corollary A.7]{KMS}, there exists morphisms of $F$-isocrystals $$\mathbf{s}_\alpha: \mathbf{1}\rightarrow \calD^\otimes$$ which restrict to $s_{\alpha,0,{\xbar}}$ at all points  ${\xbar}\in \scrS_{\rmK}(k)$. We may then argue as in \cite[Corollary 1.3.13]{KMS} to obtain the functor $\calE:\Rep_{\bfG}\rightarrow F\text{-}\mathrm{Isoc}(\scrS_{\rmK,k_E})$.

We write $\mathrm{Conj}_{\bfG}$ for the variety of conjugacy classes in $\bfG$ (cf. \cite[\S6.1.3]{KZ}), and $\chi_\bfG:\bfG\rightarrow \Conj_{\bfG}$ for the natural projection map. Let $k_E'=\bbF_q$ be a finite extension of $k_E$ and let $x\in \calS_{\rmK}(k_E').$  We write $\overline{x}\in \scrS_{\rmK}(k)$ for the geometric point lying over $x$. We obtain an element $\gamma_{x,\ell}\in \Conj_{\bfG}(\bbQ_\ell)$ by considering the action of the (geometric) local $q$-Frobenius on the stalks $\bbL_{\ell,\overline{x}}$ of $\bbL_\ell$ at $\overline{x}$. Similarly for $\ell=p$, the local Frobenius acting on the stalk $\calE_{\overline{x}}$ of the $F$-isocrystal with $\bfG$-structure $\calE$ at $\overline{x}$ gives rise to an element $\gamma_{x,p}\in \Conj_{\bfG}(\brQ)$.

\subsubsection{}Now assume $\scrS_{\rmK}$ arises from a strongly admissible triple $(\bfG,X,\calG)$. Let ${\xbar}\in \scrS_{\rmK}(k)$ and  fix $c\in I_{\xbar}(\bbQ)$ a semisimple element. By definition of $I_{\xbar}$, the  action of $c$ on $T_\ell\calA_{\xbar}$  fixes $s_{\alpha,\ell,{\xbar}}$ and hence we obtain an element $c_\ell\in \bfG(\bbQ_\ell)$ well-defined up to conjugation by $\bfG(\bbQ_\ell)$. Similarly, the action of $c$ on $\bbD\otimes_{\brZ}\brQ$  fixes $s_{\alpha,0,{\xbar}}$ and so upon fixing a tensor preserving isomorphim $$\bbD\otimes_{\brZ}\brQ\cong \Lambda^{\vee}\otimes_{\bbZ_p}\brQ,$$ which exists by Steinberg's theorem, we obtain an element $c_p\in \bfG(\brQ)$.

\begin{prop}\label{prop: l-indep point}Let $(\bfG,X,\calG)$ be a strongly admissible triple and assume $G$ is quasi-split and $\rmK^p$ is neat. There exists an element $c_0\in \bfG(\bbQ)$ such that
	
	\begin{enumerate}
		\item  For $\ell\neq p$, $c_0$ is conjugate to $c_\ell$ in $\bfG(\bbQ_\ell)$.
		
		\item  $c_0$ is conjugate to $c_p$ in $\bfG(\bbC_p)$.
		
		\item The image of $c_0\in G(\bbR)$ is elliptic.
	\end{enumerate}

\end{prop}
	\begin{proof} Let $T\subset I_{\xbar}$ be a maximal torus such that $c\in T(\bbQ)$. By Theorem \ref{thm: CM lift}, upon replacing $\xbar$ by a point in its isogeny class, we may assume there exists a special point $\tilde{x}\in \scrS_{\rmK}(\calO_K)$ lifting ${\xbar}$ such that the action of $T$ lifts to $\calA_{\tilde{x}}$. The action of $c$ on the Betti cohomology of $\calA_{\tilde{x}}$ preserves $s_{\alpha,B,\tilde{x}}$, so we obtain an element $c\in \bfG(\bbQ)$. By the Betti-\'etale comparison isomorphism, $c_0$ is conjugate to $c_{\ell}$ for $\ell\neq p$. Similarly, by the comparison between Betti, de Rham and crystalline cohomology, $c_0$ is conjugate to $c_p$ in $\bfG(\bbC_p)$.
		
		The fact that $c_0$ is elliptic follows from the fact that $T(\bbR)/w_h(\bbR^\times)$ is compact (cf. \cite[Proof of Corollary 2.1.7]{Ki3}).
		\end{proof}
	
	\subsubsection{}\label{sssec: compatible system SV} We apply the above to deduce the $\ell$-independence of the elements $\gamma_{x,\ell}$.

	\begin{cor}
		\label{cor: compatible system SV}Let $(\bfG,X,\calG)$ be a strongly admissible triple and assume  $G$ is quasi-split. For any finite extension $k_E'/k_E$ and $x\in \scrS_{\rmK}(k_E')$, there exists an element $\gamma_x\in \Conj_{\bfG}(\bbQ)$ such that 
	$\gamma_{x}=\gamma_{x,\ell}$ for all primes $\ell$ (including $\ell=p$).
	\end{cor}
\begin{proof} Let $\rmK_1^p\subset \rmK^p$ be a neat compact open subgroup and set $\rmK_1=\rmK_p\rmK_1^p$. Let ${\xbar}\in \scrS_{\rmK_1}(k)$ be a point which maps to $x$. Since $x$ is a $k_E'$-point, the abelian variety $\calA_{\xbar}$ is defined over $k_E'=\bbF_q$, and hence the geometric $q$-Frobenius is an element of $c\in\mathrm{Aut}_{\bbQ}(\calA_{\xbar})$. Moreover, $c$ fixes $s_{\alpha,\ell,{\xbar}}$ for all $\ell\neq p$ and fixes  $s_{\alpha,0,{\xbar}}$, and hence $c\in I_{\xbar}(\bbQ)$ is a semi-simple element. By construction of $\bbL_\ell$ and $\calE$, we have $c_\ell=\gamma_{x,\ell}$  for $\ell\neq p$ and $c_p=\gamma_{x,p}$. The result then follows from Proposition \ref{prop: l-indep point} by taking $\gamma_x$ to be the image of the element $c_0\in \bfG(\bbQ)$.
\end{proof}
\begin{remark}\label{rem: compatible system}
	\begin{enumerate}
\item Let $\vartheta:\bfG_{\bar{\bbQ}}\rightarrow \GL_n$ be a representation, and for a prime $\ell\neq p$, let $\bbL_\ell^\vartheta$ denote the rank $n$ local system  induced from $\bbL_{\ell}$ by pushout along $\vartheta$. Corollary \ref{cor: compatible system SV} then implies that for primes $\ell,\ell'$ not equal to $p$, $\bbL_\ell^\vartheta$ and $\bbL_{\ell'}^\vartheta$ are compatible local systems as in \cite[Conjecture 1.2.10]{De3}.

A similar remark holds for the rank $n$ $F$-isocrystal $\calE^\vartheta$ induced from $\calE$. Namely, for $\ell\neq p$, $\calE^\vartheta$ is a `petits camarades cristallines' for $\bbL_\ell^\vartheta$. To make this  precise, one should prove that $\calE$ arises from an overconvergent isocrystal over $\scrS_{\rmK,k_E}$; we will prove a version of this statement for the pullback of $\calE$ to a smooth curve in \S\ref{sec: independence of l}.

\item In \S\ref{sssec: applications}, we will  prove a version of  Corollary \ref{cor: compatible system SV} for Shimura varieties without the strongly admissible assumption. 
\end{enumerate}
\end{remark}

\section{Toroidal compactifications}
In this section, we study the toroidal compactifications of integral models of Shimura varieties constructed  in \cite{Keerthi}. The main result  is Theorem \ref{thm: boundary curves}, which will be used to compare Weil--Deligne representations over local fields of different characteristic.
Thm. 3.3.10]{KZ}, $\Mloc_{\calG,\{\mu\}}$ satisfies the Scholze--Weinstein conjecture \cite[Conj. 21.4.1]{SW2}. In our setting, this means there is an  identification 
amental group $\pi_1(G^{\der})$ of the derived group of $G$.

\subsection{Integral models of toroidal compactifications}

\subsubsection{}\label{subsubsec: toroidal cpct}We keep the notations of the last section. In particular we fix a prime number $p$. Let  $(\bfG,X,\calG)$ be a strongly admissible triple 
and fix a Hodge embedding $$\iota:(\bfG,X)\rightarrow (\mathbf{GSp}(V),S^{\pm})$$  as in Proposition \ref{prop: integral LHE}. Thus there is a self-dual lattice $\Lambda\subset V_{\bbQ_p}$ with stabilizer $\rmK_p$ such that $\iota$ extends to a closed immersion $\calG\rightarrow \GL(\Lambda)$. For $\rmK=\rmK_p\rmK^p$, $\rmK_p=\calG(\bbZ_p)$ and $\rmK^p\subset\bfG(\bbA_f^p)$ a compact open subgroup, we have the integral model $\scrS_{\rmK}$ for $\Sh_{\rmK}(\bfG,X)$ which is a Deligne--Mumford stack over $\calO_E$. By construction, there is a finite map $$\scrS_{\rmK}\rightarrow\scrS_{\rmK'}(\mathbf{GSp}(V),S^{\pm})_{\calO_E}$$
to an  integral model for the Siegel Shimura variety with hyperspecial level structure at $p$

We may now apply the constructions of \cite{Keerthi} to obtain  toroidal compactifications of $\scrS_{\rmK}$. We summarize the main properties in the following proposition. For notational simplicity, we write $\scrS_{\rmK}$ for  $\scrS_{\rmK}(\mathbf{GSp}(V),S^{\pm})_{\calO_E}$.
\begin{prop}[\cite{Keerthi}]\label{prop: Keerthi}
There exists a collection of (quasi-separated) proper  Deligne--Mumford stacks $\scrS_{\rmK}^\Sigma$ indexed by certain complete cone decompositions $\Sigma$ satisfying the following properties:

\begin{enumerate}

		\item Each $\Sigma$ is a set  $\{\Sigma(\Phi)\}_{\Phi}$ of cone decompositions indexed by  \textit{cusp label representatives} $\Phi$, called an admissible rational polyhedral cone decomposition (``rpcd" for short). For a cusp label representative $\Phi$, $\Sigma(\Phi)$ is a complete  cone decomposition  of some $\overline{\bfH}(\Phi)$, which is the union of the interior $\bfH(\Phi)$  of a homogenous self-adjoint cone (cf. \cite[Ch. II]{AMRT}) together with its rational boundary components.

	\item  For each $\Sigma$, $\scrS^{\Sigma}_{\rmK}$ admits a stratification by locally closed substacks $$\scrS_{\rmK}^\Sigma=\bigsqcup_{\Upsilon}\scrZ_{\rmK}(\Upsilon),$$
	where $\Upsilon$ runs over the set $\mathrm{Cusp}^\Sigma_{\rmK}(\bfG,X)$ of equivalence classes of pairs $(\Phi,\sigma)$, where $\Phi$ is a cusp label representative, $\sigma\in \Sigma(\Phi)$ is a cone whose interior lies in $\mathbf{H}(\Phi)$ and the equivalence relation is as in \cite[\S2.1.16]{Keerthi}. The closure relations are described by a partial order $\preccurlyeq$ on $\mathrm{Cusp}^\Sigma_{\rmK}(\bfG,X)$; cf. \cite[Theorem 4.1.5 (3)]{Keerthi}.

	\end{enumerate}

\end{prop}
\begin{proof}This is proved for neat $\rmK^p$ in \cite{Keerthi}; in this case $\scrS_{\rmK}^\Sigma$ is an algebraic space. For $\rmK^p$ general, we let $\rmK_1^p\subset \rmK^p$ be a neat normal compact open subgroup and set $\rmK_1=\rmK_p\rmK_1^p$. We let  $\Sigma_1:=(\mathrm{id},1)^*\Sigma$ be the  admissible rpcd for $(\bfG,X,\rmK_1)$ obtained from pullback along the morphism of triples $$(\mathrm{id},1):(\bfG,X,\rmK_1)\rightarrow(\bfG,X,\rmK),$$ (cf. \cite[2.1.28]{Keerthi}). Then $\Sigma_1$ is invariant under the action of the finite group $\Gamma:=\rmK^p/\rmK_1^p$, i.e. we have $(\mathrm{id},\gamma)^*\Sigma_1=\Sigma_1$ for $\gamma\in \Gamma$. By \cite[Proposition 4.1.13]{Keerthi}, $\Gamma$ acts on $\scrS_{\rmK_1}^{\Sigma_1}$ and we set $\scrS_{\rmK}^\Sigma=[\scrS_{\rmK_1}^{\Sigma_1}/\Gamma]$ to be the stack quotient. The rest of the proposition then follows from the corresponding result for $\scrS_{\rmK_1}^{\Sigma_1}$, and the compatibility of the stratification with taking quotients on the generic fiber (cf. \cite[Proposition 6.25 (b)]{Pink})
\end{proof}

\begin{remark}
	It is possible to show that $\scrS_{\rmK}^\Sigma$ does not depend on the choice of auxiliary neat compact open $\rmK_1^p$. As this is not needed in the sequel, we omit this discussion and always consider $ \scrS_{\rmK}^\Sigma$ as being constructed from a fixed choice of $\rmK'^p$.
\end{remark}
For an admissible rpcd as in the theorem, we write $\partial\scrS_{\rmK}^\Sigma$ for the boundary of $\scrS_{\rmK}^\Sigma$,  i.e. $\partial\scrS_{\rmK}^\Sigma=\scrS_{\rmK}^\Sigma\backslash\scrS_{\rmK}$, and we call $\scrS_{\rmK}$ the \emph{interior} of $\scrS_{\rmK}^\Sigma$.

When $\rmK^p$ is neat, the algebraic spaces $\scrS_{\rmK}^\Sigma$ are defined to be the normalization of certain Faltings--Chai \cite{FC} arithmetic compactifications $\scrS_{\rmK'}^{\Sigma^\dagger}$ for the integral model of the Siegel Shimura variety inside the compactification $\Sh_{\rmK}^\Sigma$  over $E$ constructed by Pink \cite{Pink}. For  general $\rmK^p$, we obtain a finite map 
$\scrS_{\rmK}^\Sigma\rightarrow \scrS_{\rmK'}^{\Sigma^\dagger}$ for suitable $\rmK', \Sigma^\dagger$  via the quotient construction above. As such, the stacks  $\scrS_{\rmK}^\Sigma$ are equipped with a families of semi-abelian schemes given by the pullback of the universal family of semi-abelian schemes, over $\scrS_{\rmK'}^{\Sigma^\dagger}$.
\subsubsection{}\label{subsubsec: local structure cpct}In order to describe the local structure of the boundary strata, we need some preparation. Let 
$\Phi$ be a cusp label representative, then we may associate to $\Phi$ a (pure) Shimura datum $(\bfG_{\Phi,h},D_{\Phi,h})$ in the sense of \cite{Pink}, and hence a Shimura variety $\Sh_{\rmK_{\Phi,h}}(\bfG_{\Phi,h},D_{\Phi,h})$ defined over the reflex field  $\bfE$. Here the compact open subgroup $\rmK_{\Phi,h}\subset \bfG_{\Phi,h}(\bbA_f)$ depends on $\rmK$. For notational convenience we write $\Sh_{\rmK_{\Phi,h}}$ for $\Sh_{\rmK_{\Phi,h}}(\bfG_{\Phi,h},D_{\Phi,h})$. Then the construction in  \cite[\S4]{Keerthi} provides us with an integral model $\scrS_{\rmK_{\Phi,h}}$ for $\Sh_{\rmK_{\Phi,h}}$ over $\calO_E$ together with a tower of algebraic stacks over $\calO_E$
$$\scrS_{\rmK_{\Phi}}\xrightarrow{f_\Phi} \scrS_{\overline{\rmK}_{\Phi}}\xrightarrow{g_\Phi} \scrS_{\rmK_{\Phi,h}}.$$
Here $\scrS_{\rmK_\Phi}$ and $\scrS_{\overline{\rmK}_\Phi}$ are integral models for mixed Shimura varieties $\Sh_{\rmK_\phi}$, $\Sh_{\overline{\rmK}_\Phi}$ associated to certain mixed Shimura data $(Q_\Phi,D_\Phi)$, $(\overline{Q}_\Phi,\overline{D}_\Phi)$, see \cite[\S2.1.7]{Keerthi}.

The integral model $\scrS_{\rmK_{\Phi,h}}$ is equipped with a natural family of abelian varieties $\calA_{\rmK}(\Phi)\rightarrow \scrS_{\rmK_{\Phi,h}}$ and the map $g_\Phi$ is a torsor for $\calA_{\rmK}(\Phi)$. The map $f_{\Phi}$ is a torsor for a certain torus $\bfE_{\rmK}(\Phi)$  over $\bbZ$ with character group $\bfX_{\Phi}$. We let $X_\Phi^\vee$ denote the cocharacter group of $\bfE_{\rmK}(\Phi)$; then by the construction, $X_{\Phi}^\vee\otimes_{\bbZ}\bbR$  canonically contains $\overline{\bfH}(\Phi)$. Therefore for a cone $\sigma\in \Sigma(\Phi)$ whose interior lies in $\bfH(\Phi)$, we obtain a twisted torus embedding over $\scrS_{\overline{\rmK}_{\Phi}}$ (see \cite[\S2.1.17]{Keerthi})$$\scrS_{\rmK_{\Phi}}\rightarrow \scrS_{\rmK_{\Phi}}(\sigma),$$ and we let $\scrZ_{\rmK_\Phi}(\sigma)$ denote the closed stratum of $\scrS_{\rmK_{\Phi}}(\sigma)$. If $\sigma$ is smooth with respect to $\bfX_\Phi^\vee$, then $\scrS_{\rmK_{\Phi}}(\sigma)\rightarrow \scrS_{\overline{\rmK}_\Phi}$ is a smooth morphism.  We write $\partial\scrS_{\rmK_\Phi}(\sigma)$ for the complement $\scrS_{\rmK_\Phi}(\sigma)\setminus\scrS_{\rmK_\Phi}$, which is a relative Cartier divisor over $\calO_E$.

We let $\Delta_{\rmK}(\Phi)$ denote the group associated to $\Phi$ as defined in \cite[(2.1.16.2)]{Keerthi}. This group acts on $\scrS_{\rmK_\Phi}$,  $\scrS_{\overline{\rmK}_\Phi}$ and $\scrS_{\rmK_{\Phi,h}}$, and $f_{\Phi}$ and $g_{\Phi}$ are equivariant for this action.  There is a natural action of $\Delta_{\rmK}(\Phi)$ on $\overline{\bfH}(\Phi)$, and we let $\Delta_{\rmK}(\Phi,\sigma)\subset \Delta_{\rmK}(\Phi)$ denote the stabilizer of $\sigma$. Then the action of $\Delta_{\rmK}(\Phi,\sigma)$ on $\scrS_{\rmK_\Phi}$ extends to an action on $\scrS_{\rmK_{\Phi}}(\sigma)$ which preserves the closed stratum $\scrZ_{\rmK_\Phi}(\sigma)$. For notational simplicity we write $\Delta=\Delta_{\rmK}(\Phi,\sigma)$, and we write $\Delta\backslash\scrS_{\rmK_{\Phi}}(\sigma)$ (resp. $\Delta\backslash\scrZ_{\rmK_{\Phi}}(\sigma)$) for the stack quotient of $\scrS_{\rmK_\Phi}(\sigma)$ (resp. $\scrZ_{\rmK_\Phi}(\sigma)$) by this action.

For an algebraic stack $X$ and $Z\rightarrow X$ a locally closed immersion, we write $X_Z^{(N)}$ for the $N^{\mathrm{th}}$ order infinitesimal neighbourhood of $Z$, and we write $X_Z^\wedge$ for the completion of $X$ along $Z$ (see \cite[Example 5.9]{EmertonFormal}).\textbf{}

\begin{prop}[{\cite[Theorem 4.1.5]{Keerthi}}]\label{prop: local structure of boundary strata} Assume  that the complete  admissible rpcd $\Sigma$ satisfies \cite[Condition 6.2.5.25]{Lan}. Let  $\Upsilon=[(\Phi,\sigma)]\in\mathrm{Cusp}_{\rmK}^\Sigma(\bfG,X)$.
\begin{enumerate} 
	\item There is a canonical isomorphism $$\Delta\backslash\scrZ_{\rmK_\Phi}(\sigma)\cong\scrZ_{\rmK}(\Upsilon).$$
	\item Let $\Delta\backslash\widehat{\scrS}_{\rmK_\Phi}(\sigma)$ be the formal completion of $\Delta\backslash\scrS_{\rmK_\Phi}(\sigma)$ along the closed stratum $\Delta\backslash\scrZ_{\rmK_\Phi}(\sigma)$, and $(\scrS^\Sigma_{\rmK})^\wedge_{\scrZ_{\rmK}(\Upsilon)}$ the formal completion of $\scrS^\Sigma_{\rmK}$ along  $\scrZ_{\rmK}(\Upsilon)$. Then the canonical isomorphism from (1) lifts to an isomorphism of formal algebraic stacks \begin{equation}\label{eqn: formal nbd iso neat}
\Delta\backslash	\widehat{\scrS}_{\rmK_\Phi}(\sigma)\cong (\scrS^\Sigma_{\rmK})^\wedge_{\scrZ_{\rmK}(\Phi)} 	\end{equation}which identifies the formal substacks $$\Delta\backslash\partial	\widehat{\scrS}_{\rmK_\Phi}(\sigma)\cong (\partial\scrS^\Sigma_{\rmK})^\wedge_{\scrZ_{\rmK}(\Phi)}. $$
\end{enumerate}

\end{prop}
\begin{proof}The case of neat $\rmK^p$ is \cite[Theorem 4.1.5]{Keerthi}; note that in this case, the group $\Delta_{\rmK}(\Phi,\sigma)$ is trivial (see \cite[Lemma 2.1.20]{Keerthi}). For general $\rmK^p$, we apply the quotient construction in the proof of Proposition \ref{prop: Keerthi}. Thus $\scrS_{\rmK}^\Sigma=\scrS_{\rmK_1}^{\Sigma_1}/\Gamma$ for a suitable normal neat compact open subgroup $\rmK_1^p\subset \rmK^p$ with $\Gamma=\rmK^p/\rmK_1^p$. Let $\Upsilon_1=[(\Phi_1,\sigma_1)]\in \mathrm{Cusp}_{\rmK_1}^{\Sigma_1}(\bfG,X)$ be an element mapping to $\Upsilon=[(\Phi,\sigma)]$. Then using the description of the action of $\Delta_{\rmK}(\Phi)$ on equivalences classes of cusp label representatives in \cite[\S6.18]{Pink}, we find that the stabilizer of $\Upsilon_1$ is isomorphic to $\Delta_{\rmK}(\Phi,\sigma)\times\Gamma_\Phi$ where $\Gamma_{\Phi}= \rmK_\Phi/\rmK_{1,\Phi_1}$. Thus we obtain the desired  isomorphisms by taking the quotient of the corresponding isomorphisms
	$$\scrZ_{\rmK_{1,\Phi_1}}(\sigma_1)\cong\scrZ_{\rmK_{1}}(\Upsilon_1),\ \ \ \	\widehat{\scrS}_{\rmK_{1,\Phi_1}}(\sigma_1)\cong (\scrS^{\Sigma_1}_{\rmK_1})^\wedge_{\scrZ_{\rmK_1}(\Upsilon_1)}.$$
for $\Upsilon_1$ at level $\rmK_1$ by the action of this stabilizer, noting that $$ \scrS_{\rmK_{1,\Phi}}(\sigma_1)/\Gamma_\Phi=\scrS_{\rmK_\Phi}(\sigma). $$
	\end{proof}

\subsubsection{}\label{subsubsec: equivariant completions} We may now apply Artin approximation  to deduce a result about the \'etale local structure of $\scrS_{\rmK}^\Sigma$.
Thus let $\Sigma$ be a complete admissible rpcd for $(\bfG,X,\rmK)$ satisfying \cite[Condition 6.2.5.25]{Lan} and let $\scrS_{\rmK}^\Sigma$ be the associated compactification over $\calO_E$.
Let $\Upsilon=[(\Phi,\sigma)]\in \mathrm{Cusp}^{\Sigma}_{\rmK}(\bfG,X)$. Let $x_0\in |\scrZ_{\rmK}(\Upsilon)|$  be a  point lying over the closed point of $\Spec\calO_E$ and let  $y_0\in |\Delta\backslash\scrS_{\rmK_\Phi}(\sigma)|$ denote the point corresponding to $x_0$ under the  isomorphism \begin{equation}\label{eqn: boundary strata neat}\scrZ_{\rmK}(\Upsilon)\cong\Delta\backslash \scrZ_{\rmK_\Phi}(\sigma)\end{equation} of Proposition \ref{prop: local structure of boundary strata} (1). By Proposition \ref{prop: local structure of boundary strata} (2), this extends to an isomorphism of formal completions \begin{equation}\label{eqn: equivariant iso}(\scrS_{\rmK}^
{\Sigma})^\wedge_{x_0}\cong(\Delta\backslash\scrS_{\rmK_\Phi}(\sigma))^\wedge_{y_0}\end{equation} 
at $x_0$ and $y_0$ respectively.

\begin{cor}\label{cor: Artin approx for SV}
	
For any $N\geq 1$,	there exists a stack $U$ and a point $u_0\in |U|$ which fits in a diagram of  \'etale pointed morphisms \begin{equation}\label{eqn: Artin approx}\xymatrix{ & (U,u_0) \ar[rd]^{g}\ar[ld]_{f}& \\ (\scrS_{\rmK}^{\Sigma},x_0)& & (\Delta\backslash \scrS_{\rmK_\Phi}(\sigma),y_0) . }\end{equation}
such that the following three conditions are satisfied.
\begin{enumerate}
	\item $f$ and $g$ both induce isomorphisms of residual gerbes  at $u_0$.
	\item We have $f^{-1}(\partial\scrS_{\rmK}^\Sigma)=g^{-1}(\partial\scrS_{\rmK_\Phi}(\sigma))$
\item The induced morphism between the  $N^{th}$-order infinitesimal neighborhoods  coincides with the one induced from \eqref{eqn: equivariant iso}.
\end{enumerate}
\end{cor}
\begin{proof}
	This follows from applying Artin approximation as in the proof of \cite[Theorem 4.19]{AHR1}, cf. also \cite[Theorem 10.10]{AHR2}. We recall the argument here for convenience.
	
We may replace  $(\scrS_{\rmK}^{\Sigma},x_0)$ and $(\Delta\backslash\scrS_{\rmK_\Phi}(\sigma),y_0)$ by \'etale neighborhoods  of the form $([\Spec A/G],a_0)$ and $([\Spec B/G],b_0)$ respectively, where $G$ is the finite automorphism group of the respective points. The closed substacks $\partial\scrS_{\rmK}^\Sigma$, $\partial\scrS_{\rmK_\Phi}(\sigma)$ correspond to $G$-invariant ideals $I\subset A$ and $J\subset B$. Let $A^G$ denote the invariants of $A$  under the action of $G$, and we let $\widehat{A}$ (resp. $\widehat{A^G}$), denote the completion of $A$ at $a_0$ (resp. $A^G$ at the image of $a_0$). Then we have an isomorphism $A\otimes_{A^G}\widehat{A^G}\cong \widehat{A}$.

 We consider the functor $F$ on $A^G$-algebras $R$ given by $$F(R)=\{f\in \Hom_G(B,A\otimes_{A^G}R)| f(J)A\otimes_{A^G}R=IA\otimes_{A^G}R \}$$ where $\Hom_G$ denotes the set of $G$-equivariant homomorphisms. Then $F$ is locally of finite presentation, and we have an element $\widehat\xi\in F(\widehat{A^G})$ arising from \eqref{eqn: equivariant iso}. By Artin approximation, we obtain a residually trivial \'etale nieghborhood $$(\Spec C=U',u'_0)$$ of $\Spec A^G$, and an element $\xi\in F(C)$ which coincides with $\widehat{\xi}$ up to order $N$. We set $U=[\Spec A/G]\times_{\Spec A^G}U'$ which satisfies (1) and (3) by the previous discussion, and satisfies (2) by definition of the functor $F$.
\end{proof}

\subsection{Existence of boundary curves}We now use the result of the previous subsection to show that the algebraic stacks $\scrS_{\rmK}^\Sigma$
 admit certain morphisms from \'etale neighbourhoods of $\Spec\calO_F[u]$ at its boundary.  

We keep the notation of the previous subsection. 
Thus  $\scrS_{\rmK}$ is the integral model over $\calO_E$ associated to a strongly admissible triple $(\bfG,X,\calG)$, and $\rmK^p\subset\bfG(\bbA_f^p)$ is a fixed not necessarily neat compact open subgroup. We let  $\scrS_{\rmK}^\Sigma$ be the compactification associated to a complete admissible rpcd $\Sigma$ satisfying \cite[Condition 6.2.5.25]{Lan}. In this section, we assume in addition that $\Sigma$ is smooth (cf. \cite[\S2.1.17]{Keerthi}). The assumption of smoothness implies that for any $[(\Phi,\sigma)]\in\mathrm{Cusp}_{\rmK}^{\Sigma}(\bfG,X)$, the morphism $\Delta\backslash\scrS_{\rmK_\Phi}(\sigma)\rightarrow \Delta\backslash\scrS_{\overline{\rmK}_\Phi}$ is  smooth.

Let $\Upsilon:=[(\Phi,\sigma)]\in\mathrm{Cusp}_{\rmK}^{\Sigma}(\bfG,X)$ be such that the corresponding stratum $\scrZ_{\rmK}(\Upsilon)$ is contained in the boundary $\partial\scrS_{\rmK}^\Sigma$; this is equivalent to the condition that the associated admissible parabolic $P_{\Phi}$ is proper. Let $x_0\in |\Delta\backslash\scrZ_{\rmK}(\Upsilon)|$ with corresponding point $y_0\in |\scrZ_{\rmK_\Phi}(\sigma)|$ under the identification \eqref{eqn: boundary strata neat}.

Let $\calO_F$ be a finite extension of $\calO_E$ with residue field $k_F$ and uniformizer $\pi\in\calO_F$. Let $y\in \Delta\backslash\scrS_{\rmK_\Phi}(\sigma)(\calO_F)$ such that the induced $k_F$-point represents $y_0$. We write $y_\eta\in \Delta\backslash\scrS_{\rmK_\Phi}(\sigma)(F)$ for the generic fiber of $y$.
We assume $y_\eta\in \Delta\backslash\scrS_{\rmK_\Phi}(F)$, i.e. $y_\eta$ does not lie on the boundary divisor $\partial \scrS_{\rmK_\Phi}(\sigma)$.
\begin{prop}\label{prop: boundary curves local}  Under the assumptions above, 
	 there exists $$\tilde{y}\in \Delta\backslash\scrS_{\rmK_\Phi}(\sigma)(\calO_F[u])$$ satisfying the following two conditions:
		\begin{enumerate}
		\item The specialization of $\tilde{y}$ along $\mathcal{O}_F[u]\rightarrow \calO_F, u\mapsto \pi$ is isomorphic to $y$.
		
		\item The preimage $\tilde{y}^{-1}(\Delta\backslash\scrS_{\rmK_\Phi})$ of the interior $\Delta\backslash\scrS_{\rmK_\Phi}\subset \scrS_{\rmK_\Phi}(\sigma)$ contains the open subscheme $\Spec \calO_F[u,u^{-1}]\subset \Spec\calO_{F}[u]$.
	\end{enumerate}

\end{prop}
\begin{proof}Let $\overline{y}$ denote the image of $y$ in $\Delta\backslash\scrS_{\overline{\rmK}_\Phi}(\calO_F)$ and we write $Y_{\Phi}$  (resp. $Y_\Phi(\sigma)$) for the pullback of  $$\Delta\backslash\scrS_{\rmK_\Phi}\rightarrow \Delta\backslash\scrS_{\overline{\rmK}_\Phi}
	\text{ (resp.  $\Delta\backslash\scrS_{\rmK_\Phi}(\sigma)\rightarrow \Delta\backslash\scrS_{\overline{\rmK}_\Phi}$)}$$  along $\overline{y}$. 
	Then $Y_{\Phi}$ is a torsor for the torus $T:=\bfE_{\rmK}(\Phi)_{\calO_F}$ over $\calO_F$ and  $Y_{\Phi}\rightarrow Y_{\Phi}(\sigma)$ is the  torus embedding corresponding to the cone $\sigma$. 
	
	Since $T$ is a smooth  group scheme over $\calO_F$ with connected special fiber, $Y_{\Phi}$ is a trivial $T$-torsor. Thus $Y_{\Phi}\cong \Spec\calO_F[x_1^\pm,\dotsc,x_n^\pm]$ where $x_1,\dotsc,x_n$ is identified with a basis of  the character lattice $\bfX:=X_*(T)$ of $T$. Since $\scrZ_{\rmK}(\Upsilon)$ lies on the boundary of $\scrS_{\rmK}^\Sigma$, we have $n\geq 1$. Then since $\sigma$ is smooth we may choose the basis of $\bfX$ so that $Y_\Phi(\sigma)\cong \Spec \calO_F[x_1,\dotsc,x_m,x_{m+1}^\pm,\dotsc, x_n^\pm]$, and the boundary divisor $\partial Y_\Phi(\sigma)$ is the union of the closed subschemes $(x_i=0)_{i=1,\dotsc,m}$. The  closed stratum $Z_\Phi(\sigma)$ is the identified with the locus where $x_1=\dotsc=x_m=0$. 
	
	We let $y\in Y_\Phi(\calO_F)$ denote the point corresponding to $y$. Then since  the special fiber $y_0$ of $y$ lies in $ Z_\Phi(\sigma)(k_F)$, we have $y(x_i)\in \pi\calO_F$ for $i=1,\dotsc,m$.  Moreover, since $Y_\eta$ does not lie on the boundary divisor, we have $y(x_i)\neq 0$ for $i=1,\dotsc,m$. Thus for $i=1,\dotsc,m$,  we can write $y(x_i)=a_i\pi^{M_i}$ for some $a_i\in\calO_F^\times$ and $M_i\in \bbZ_{\geq 1}$.  We may then define $\tilde{y}:\calO_F[x_1,\dotsc,x_m,x_{m+1}^\pm,\dotsc,x_n^\pm]\rightarrow \calO_F[u]$  by $$x_i\mapsto \begin{cases} 
a_i u^{M_i}& 	\text{ $1\leq i \leq m$}\\
y(x_i) & \text{ $m+1\leq i \leq n$}
			\end{cases}.$$
			Then $\tilde{y}$ defines an $\calO_F[u]$-point of $Y_\Phi(\sigma)$, and  the specialization of $\tilde{y}$ along $\mathcal{O}_F[u]\rightarrow F, u\mapsto \pi$ is equal to $y$.
			Thus taking $\tilde{y}$ to be the composition $\Spec\calO_F[u]\xrightarrow{\tilde{y}}Y_\Phi(\sigma)\rightarrow \scrS_{\rmK_\Phi}(\sigma)$, we find that $\tilde{y}$ satisfies condition (1) in the statement of the Proposition. By construction, the preimage $\tilde{y}^{-1}(\Delta\backslash \scrS_{K_\Phi})$ is the locus where each $a_iu^M_i$ is invertible, and hence we also obtain (2).
\end{proof}

\subsubsection{}Using Corollary \ref{cor: Artin approx for SV}, we can transfer the above result to deduce the main geometric property of  $\scrS_{\rmK}^{\Sigma}$ that we will need.

Let $x\in \scrS_{\rmK}^{\Sigma}(\calO_F)$. We write $x_0$ (resp $x_\eta$) for the induced $k_F$-point (resp. $F$-point) of $\scrS_{\rmK}^\Sigma$.
\begin{thm}\label{thm: boundary curves} Assume $x_0\in |\partial\scrS_{\rmK}^{\Sigma}|$ lies on the boundary and $x_\eta\in |\scrS_{\rmK}|$ lies in the interior. There exists an \'etale $\calO_F[u]$-scheme $X$  with geometrically connected special fiber $X_{k_F}=X\otimes_{\calO_F}k_F$ and an element $\tilde{x}\in\scrS_{\rmK}^{\Sigma}(X)$ satisfying the following conditions:

	\begin{enumerate}

		\item There exists an $\calO_F$-point $\delta\in X(\calO_F)$  lying above $\calO_F[u]\rightarrow \calO_F, u\mapsto \pi$ such that the  specialization of $\tilde{x}$ along $\delta$ is equal to $x$.
		\item The preimage  $\tilde{x}^{-1}(\scrS_{\rmK})$ of the interior $\scrS_{\rmK}\subset \scrS_{\rmK}^\Sigma$ contains the open subscheme $X[u^{-1}]:=X\setminus \{u=0\}\subset X$.
	\end{enumerate}

	\end{thm}

\begin{proof} Let $\Upsilon=[(\Phi,\sigma)]\in\mathrm{Cusp}_{\rmK}^{\Sigma}(\bfG,X)$  be such that $x_0\in \scrZ(\Upsilon)$ under the decomposition in Proposition \ref{prop: Keerthi}. Let $N\geq 1$ be a sufficiently large integer such that the induced $\calO_F/\pi^{N+1}$-point of $\scrS_{\rmK}^\Sigma$ does not lie on the boundary $\partial\scrS_{\rmK}^\Sigma$. We let $y_0\in \Delta\backslash\scrZ_{\rmK_\Phi}(\sigma)(k_F)$ be the element corresponding to $x_0$ under the isomorphism $\scrZ_{\rmK}(\Upsilon)\cong\Delta\backslash\scrZ_{\rmK_\Phi}(\sigma)$.
	Let $(U,u_0)$ be a common \'etale neighborhood of $(\scrS_{\rmK}^\Sigma,x_0)$ and $(\scrS_{\rmK_\Phi}(\sigma),y_0)$ as in Corollary \ref{cor: Artin approx for SV}. Thus
	we have a diagram of \'etale pointed morphisms of algebraic stacks \begin{equation}\xymatrix{ & (U,u_0) \ar[rd]^{g}\ar[ld]_{f}& \\ (\scrS_{\rmK}^{\Sigma},x_0)  & & (\Delta\backslash\scrS_{\rmK_\Phi}(\sigma),y_0)}\end{equation} inducing isomorphisms of residue gerbes at base points and such that  $$f^{-1}(\partial\scrS_{\rmK}^\Sigma)=g^{-1}(\Delta\backslash\partial\scrS_{\rmK_\Phi}(\sigma)).$$ Moreover, the induced map on $N^{\mathrm{th}}$-order infinitesimal neighborhoods coincides with the map induced by the isomorphism
	 \begin{equation}\label{eqn: inf nbds}\Delta\backslash	\widehat{\scrS}_{\rmK_\Phi}(\sigma)\cong (\scrS^\Sigma_{\rmK})^\wedge_{\scrZ_{\rmK}(\Upsilon)}\end{equation} arising from Proposition \ref{prop: local structure of boundary strata} (b).

 	The pullback of $f$ along $x$ is a residually trivial \'etale neighborhood of the closed point in $\Spec\calO_F$, thus $u_0$ lifts to a point $u\in U(\calO_F)$ which maps to $x$. We let $y$ be the image of $u$ in $\Delta\backslash\scrS_{\rmK_\Phi}(\sigma)(\calO_F)$, which gives a lift of $y_0$.  Then the induced $\calO_F/\pi^{N+1}$-point of $\Delta\backslash\scrS_{\rmK_\Phi}(\sigma)$ does not lie on the boundary $\Delta\backslash\partial\scrS_{\rmK_\Phi}(\sigma)$, and hence $y_\eta\in \Delta\backslash \scrS_{\rmK_\Phi}(F)$.

Thus $y$ satisfies the assumptions of Propostion \ref{prop: boundary curves local}, and let $\tilde{y}\in \Delta\backslash\scrS_{\rmK_\Phi}(\calO_F[u])$ be the element constructed there.
We let $\tilde{u}:X'\rightarrow U$ denote the pullback of $\tilde{y}$ along $g$. Then $X'$ is a residually trivial \'etale neighborhood of the closed point $s$ in $\Spec\calO_F[u]$. We let $\delta_0\in X'(k_F)$ be a lift of $s$,  and we let $X$ denote the scheme obtained from $X'$ by removing all the connected components of $X'\otimes_{\calO_F}k_F$ which do not contain $\delta_0$. Then $X$ has geometrically connected special fiber.
	 We let $\tilde{x}\in \Delta\backslash\scrS_{\rmK}^\Sigma(X)$ denote the composition of $X\xrightarrow{\tilde{u}} U\xrightarrow{f}\Delta\backslash\scrS_{\rmK}^{\Sigma}$.

Consider the base change $\overline{X}$ of $X$ along the map $\calO_F[u]\rightarrow \calO_F,u\mapsto \pi$. Then $\overline{X}$ is an \'etale scheme over $\calO_F$ which contains $\delta_0$. Since $\calO_F$ is Henselian, we obtain a section $\Spec\calO_F\rightarrow \overline{X}$ which maps the closed point of $\Spec\calO_F$ to $\delta_0$. We let $\delta$ denote the composition $\Spec\calO_F\rightarrow \overline{X}\rightarrow X$. Then $\delta$ maps to $u$ in $U(\calO_F)$, and hence maps to $x$ in $\Delta\backslash\scrS_{\rmK}^{\Sigma}(\calO_F)$. This verifies condition (1).
		 
	For condition (2), note that $$\tilde{x}^{-1}(\scrS_{\rmK})=(f\circ\tilde{u})^{-1}(\scrS_{\rmK})=(g\circ\tilde{u})^{-1}(\Delta\backslash\scrS_{\rmK_\Phi})\supset X[u^{-1}]$$ by the corresponding property of the point $\tilde{y}$. The result follows.
	\end{proof}
\begin{remark}
	This Theorem plays the role of \cite[Corollary 5.2.7]{KZ} in our context. It will allows us to relate the Weil--Deligne representation of an abelian variety over $F$ with semistable reduction, to the Weil--Deligne representation for a mod $p$ family of abelian varieties.
\end{remark}

\section{Weil--Deligne representations}
In this section, we discuss Weil--Deligne representations valued in a reductive group. 
\subsection{$\bfG$-valued Weil--Deligne representations}\label{ssec:G WD rep}

\subsubsection{}\label{sssec:G valued WD rep }

Let $F$ be a non-archimedean local field with ring of integers $\calO_F$ and residue field $k_F$ of size $q$.  Fix $\bar{F}$ a separable closure of $F$ and let $\Gamma_F$ denote the absolute Galois group $\mathrm{Gal}(\bar{F}/F)$.  We write $I_F\subset \Gamma_F$ for the inertia subgroup and $W_F$ for the Weil group of $F$. Let $k$ denote the residue field of $\bar{F}$ which is an algebraic closure of $k_F$. We identify $\mathrm{Gal}(k/k_F)$ with $\widehat{\bbZ}$ by sending the arithmetic Frobenius to 1. We then have an exact sequence $$0\rightarrow I_F\rightarrow W_F\xrightarrow{\alpha} \bbZ\rightarrow 0$$
where $\alpha$ sends $w\in W_F$ to the integral power of arithmetic Frobenius it induces.
We equip $W_F$ with the weakest topology for which $I_F$ equipped with the inherited topology from $\Gamma_F$ is an open subgroup of $W_F$.

The Weil--Deligne group $\WD_F$ is the algebraic group $W_F\ltimes \bbG_a$ over $\bbQ$, where we consider $W_F$ as a constant group scheme which acts on $\bbG_a$ via multiplication by $q^{\alpha(\cdot)}$.

We fix an algebraically closed field $\Omega$ of characteristic 0 and we let $\bfG$ be a  connected reductive group over  $\Omega.$

\begin{definition}A \textit{Weil--Deligne $\bfG$-representation}  is an algebraic representation 
$$\rho^{\WD}:(\WD_F)_\Omega\rightarrow \bfG$$ such that the induced map $W_F \rightarrow \bfG(\Omega)$ 
is continuous for the discrete topology on $\bfG(\Omega)$.

Two such representations  $\rho_1^{\WD}$ and $\rho_2^{\WD}$ are equivalent  (written $\rho^{\WD}_1\sim_{\bfG}\rho_2^{\WD}$) if there exists $g\in \bfG(\Omega)$ such that $\mathrm{Ad}(g)(\rho_1^{\WD})=\rho_2^{\WD}$. \end{definition}

A Weil--Deligne representaion will simply be a Weil--Deligne $\GL_n$-representation for some $n\geq 1$. We let  $\Rep_{/\Omega} \WD_F$ denote the category of Weil--Deligne representations over $\Omega$.

A Weil--Deligne $\bfG$-representation is equivalent to the data of  a pair $(\rho',N)$ where $\rho':W_F\rightarrow  \bfG(\Omega)$ is a continuous homomorphism and $N\in \mathrm{Lie}(\bfG)$ is an element satisfying \begin{equation}\label{eqn: WD property}\mathrm{Ad}(\rho'(w))N=q^{\alpha(w)}N\end{equation} for all $w\in W_F$. The condition  \eqref{eqn: WD property} implies that $N$ is a nilpotent element in $\mathrm{Lie}(\bfG)$. In what follows, we use these two points of view on Weil--Deligne $\bfG$-representations interchangeably without comment.

\subsubsection{} We now focus our attention on  Weil--Deligne $\bfG$-representations satisyfing the following two conditions.
\begin{definition}Let $\rho^{\WD}$ be a Weil--Deligne $\bfG$-representation. We say that $\rho^{\WD}$ is 
\begin{enumerate} \item \emph{Frobenius semisimple} if  $\rho^{\WD}(\sigma,0)$ is semisimple for one (equivalently any) lift of Frobenius 
$\sigma\in W_F$. 
		
		\item \emph{unipotently ramified} if $\rho^{\WD}|_{W_F}$ is unramified.
		\end{enumerate} 
If $\rho^{\WD}$ satisfies both the above conditions, we will use the abbreviation  (URFS) to describe $\rho^{\WD}$.
\end{definition}
Given a Weil--Deligne $\bfG$-representation $\rho^{\WD}$, we let $\rho^{\WD,\Fss}$ denote its Frobenius semisimplification. In terms of the pair $(\rho',N)$, this is given by taking $(\rho'^{\Fss},N)$, where $\rho'^{\Fss}:W_F\rightarrow \bfG(\Omega)$ is the semisimplification of $\rho'$.

 Let $\rho^{\WD}$ be a unipotently ramified Weil--Deligne $\bfG$-representation. Then $\rho^{\WD}|_{W_F}$ is determined by the image $s:=\rho^{\WD}(\sigma,0)$ of a lift of the Frobenius element $\sigma\in W_F$. Hence $\rho^{\WD}$ is determined by a pair of elements $(s,N)\in \bfG(\Omega)\times \mathrm{Lie}(\bfG)$ satisfying $\mathrm{Ad}(s)N=qN$. If $\rho_{\WD}$ is in addition Frobenius semisimple, then $s$ is a semisimple element of $\bfG(\Omega)$. We thus define the set  $$\Phi^\square (q,\bfG,\Omega
):=\{(s,N)\in \bfG(\Omega)^{\mathrm{ss}}\times \calN|\mathrm{Ad}(s)N=qN\}$$ where $\bfG(\Omega)^{\mathrm{ss}}$ is the set of semisimple elements of $\bfG(\Omega)$ and  $\calN\subset \mathrm{Lie}(\bfG)$ is the nilpotent cone.
Then we may identify $\Phi^\square (q,\bfG,\Omega)$ with the set of (URFS) Weil--Deligne $\bfG$-representations.  

The  relation $\sim_{\bfG}$ induces  an equivalence relation on $\Phi^{\square}(q,\bfG,\Omega)$. We set
$$\Phi(q,\bfG,\Omega):=\Phi^\square(q,\bfG,\Omega)/\sim_{\bfG}$$ to be the set of equivalence classes in $\Phi^\square(q,\bfG,\Omega)$, which we consider to be the set of isomorphism classes of (URFS) Weil--Deligne $\bfG$-representations. For $\rho^{\WD}\in \phi^{\square}(q,\bfG,\Omega)$, we write $[\rho^{\WD}]$ for its image in  $\Phi(q,\bfG,\Omega)$.

\begin{remark}\label{rem: different local fields} 
As the notation suggests, the set $\Phi(q,\bfG,\Omega)$ only depends on $F$ via the size $q$ of its residue field. We will use this observation to compare (URFS) Weil--Deligne $\bfG$-representations over  non-archimedean local fields of different characteristics.
\end{remark}
\subsubsection{}
Now let $L\subset \Omega$ be a subfield and assume $\bfG$ is defined over $L$.  If $\rho^{\WD}$ is a Weil--Deligne $\bfG$-representation, then for any $\varsigma\in \mathrm{Aut}(\Omega/L)$, we obtain another Weil--Deligne $\bfG$-representation by taking the composition $\varsigma\circ\rho^{\WD}$.

\begin{definition}
 Let $\rho^{\WD}$ be a  Weil--Deligne $\bfG$-representation. We say that $\rho^{\WD}$ is \emph{defined over L} if for all $\varsigma\in \mathrm{Aut}(\Omega/L)$, we have $$\varsigma\circ \rho^{\WD}\sim_{\bfG} \rho^{\WD}.$$
\end{definition} 

It is easy to see that this definition only depends on the equivalence class of $\rho^{\WD}$ under $\sim_{\bfG}$.  In particular, if $[\rho]\in \Phi(q,\bfG,\Omega)$, then we say $[\rho]$ is  defined over $L$  if $\rho^{\WD}$ is defined over $L$ for one (equivalently any) representative $\rho^{\WD}\in \Phi^{\square}(q,\bfG,\Omega)$ of $[\rho].$

\subsubsection{}

The next proposition is a key property of (URFS) representations that we will need. It shows that for such representations,  equivalence can be detected by considering all representations.

\begin{prop}\label{prop: WD rep reduction to GL_n}Let $\rho_1^{\WD}$, $\rho^{\WD}_2\in \Phi^\square(q,\bfG,\Omega)$ be  (URFS) Weil--Deligne $\bfG$-representations. Suppose for every representation $r:{\bfG}\rightarrow \GL_n$,  we have $$r\circ \rho_1^{\WD}\sim_{\GL_n}r\circ\rho_2^{\WD}.$$ Then $$\rho_1^{\WD}\sim_{\bfG}\rho^{\WD}_2,$$  i.e. $[\rho_1^{\WD}]=[\rho_2^{\WD}]\in \Phi(q,\bfG,\Omega)$.
\end{prop}

\begin{proof}For $i=1,2$, we set $$s_i:=\rho_i^{\WD}(\sigma,0)\in \bfG(\Omega)$$ for $\sigma\in W_F$ a lift of Frobenius, and we set $$u_i:=\rho_i^{\WD}|_{\bbG_a}(1)\in \bfG(\Omega).$$ By \cite[Lemma 1.12]{Imai}, there are homomorphisms $\xi_1,\xi_2:\SL_2\rightarrow \bfG$ such that $$\xi_i\left(\left[\begin{matrix}
	1& 1\\ 0& 1
	\end{matrix}\right]\right)=u_i, \ \ \ \ s_i':=\xi_i\left(\left[\begin{matrix}
	q^{\alpha(w)/2} & 0\\ 0 & q^{-\alpha(w)/2}
	\end{matrix}\right]\right)s_i^{-1}\in Z_{\bfG}(\mathrm{Im}(\xi_i))$$ for $i=1,2$. Here $Z_{\bfG}(\mathrm{Im}(\xi_i))$ denotes the centralizer of the image of $\xi_i$. By assumption, for any $r:\bfG\rightarrow\GL_n$, there exists $g\in \GL_n(\Omega)$ such that $$g^{-1}r(s_1)g=r(s_2)\text{ and } g^{-1}r(u_1)g=r(u_2).$$ By the uniqueness statement in \cite[Lemma 1.12]{Imai}, there exists an $h\in \GL_n(\Omega)$ which commutes with $r(s_2)$ and $r(u_2)$ such that $$h^{-1}g^{-1}(r\circ\xi_1)gh=r\circ \xi_2;$$ thus we have $$h^{-1}g^{-1}r(s_1')gh=r(s_2').$$ Since  this is true any for representation $r$, we have that $s_1'$ and $s_2'$ are $\bfG(\Omega)$-conjugate by \cite{Steinberg:regular}. Thus upon replacing $\rho_1^{\WD}$ by a $\bfG(\Omega)$-conjugate, we may choose $\xi_1$ such that $s_1'=s_2'$.
	
	Let $M\subset \bfG$ denote the centralizer of $s_1'$. Our assumption implies that for all representations $r:\bfG\rightarrow \GL_n$, $r\circ\xi_1$ and $r\circ\xi_2$ are conjugate by an element of $Z_{\GL_n}(r(s_1'))$. Thus by Lemma \ref{lem: conjugate SL2} below, upon further replacing $\rho_1^{\WD}$ by an $M(\Omega)$-conjugate, we may assume $\xi_1=\xi_2$. But then $u_1=u_2$ and $s_1=s_2$. Since $u_i$ determines the homomorphism $\rho_i^{\WD}|_{\bbG_a}$, we have $\rho_1^{\WD}=\rho_2^{\WD}$.
\end{proof}

\begin{lemma}\label{lem: conjugate SL2}
Let $\bfG$ be a reductive group over $\Omega$ and  $s\in \bfG(\Omega)$ a semisimple element. Let $M:=Z_{\bfG}(s)$ denote the centralizer of $s$ and let $\xi_1,\xi_2:\SL_2\rightarrow M$ be two algebraic homomorphisms such that for all representations $r:\bfG\rightarrow \GL_n$, $r\circ \xi_1$ and $r\circ\xi_2$ are conjugate by an element of $Z_{\GL_n}(r(s))$. Then $\xi_1$ and $\xi_2$ are conjugate by an element of $M(\Omega)$.
\end{lemma}
\begin{proof} 
For $i=1,2$, let $\alpha_i:\bbG_m\rightarrow M$ be the composite of the  map $d:\bbG_m\rightarrow \SL_2,$ 
which identifies $\bbG_m$ with the diagonal torus, and the morphism $\xi_i$. By \cite[Theorem 4.2]{Kostant}, any homomorphism $\xi:\SL_2\rightarrow M$ is determined by its  composition with $d:\bbG_m\rightarrow \SL_2$ up to conjugation by an element of $M(\Omega)$. Thus it suffices to show that $\alpha_1$ and $\alpha_2$ are conjugate by an element of $M(\Omega)$.
Upon replacing $\xi_2$ by an $M(\Omega)$-conjugate, we may assume that $\alpha_1$ and $\alpha_2$ factor through a common maximal torus in $M$. It follows that there exists a maximal torus $T\subset M$ which contains $s$ and the images of $\alpha_1$ and $\alpha_2$. Then $T$ is also a maximal torus in $G$.

By assumption, for any $t\in \Omega^\times$ and any representation $r:\bfG\rightarrow \GL_n$, $r(\alpha_1(t))$ and $r(\alpha_2(t))$ are conjugate by an element of $Z_{\GL_n}(r(s))(\Omega)$, and hence so are $r(s\alpha_1(t))$ and $r(s\alpha_2(t))$. Thus by \cite{Steinberg:regular}, $s\alpha_1(t)$ and $s\alpha_2(t)$ are conjugate by an element of $\bfG(\Omega)$. 
Since $s\alpha_1(t)$ and $s\alpha_2(t)$ both lie in $T$, they are conjugate by an element $w(t)$ of the Weyl group $W$ for $T$. 
As $W$ is finite, there is an element $w \in W$ such that we may take $w(t)=w$ for $t$ in some Zariski open subset of $\Omega^\times,$ 
and hence we may take $w(t)=w$ for any $t.$ 
In particular, setting $t=1$, we obtain $w(s)=s$ and hence $w$ can be represented by an element $w'\in M(\Omega)$. Thus  $\alpha_1$ and $\alpha_2$ are conjugate by an element of $M(\Omega)$ as desired.
\end{proof}
\begin{remark}\label{rem: acceptable group}
	Proposition \ref{prop: WD rep reduction to GL_n} does not hold for representations which are not unipotently ramified. Indeed it already fails for Weil representations. 
	
	For example, let $\bfG=\mathrm{SO}_6$, $\Gamma=\bbZ/4\bbZ\times \bbZ/4\bbZ$ and $\rho:\Gamma\rightarrow \mathrm{SO}_6(\bbC)$  the representation constructed in \cite[Lemma 23]{Weidner}. We let $\rho':\Gamma\rightarrow \mathrm{SO}_6(\bbC)$ be the conjugate of $\rho$ by an element $g\in\mathrm{O}_6(\bbC)\setminus\mathrm{SO}_6(\bbC)$. Then $\rho$ and $\rho'$ are element-conjugate  in the sense of \cite{Larsen} and hence $r\circ \rho$ and $r\circ \rho'$ are conjugate for all representations $r:\mathrm{SO}_6\rightarrow\GL_n$. However $\rho$ is not conjugate to $\rho'$ in $\mathrm{SO}_6(\bbC)$.  It follows that the proposition fails for the Weil representations which  arise by composing $\rho$ with a surjection $W_F\rightarrow \Gamma$.
	\end{remark}

\subsection{Comparison of tame fundamental groups}\label{ssec: l-adic comparison}

\subsubsection{}\label{sssec: tame fundmental group} In this section, we explain a way to compare Weil--Deligne representations over different local fields. We first recall how to attach Weil--Deligne representations to $\ell$-adic representations.   

We keep the notation of  \S\ref{sssec:G valued WD rep } so that $F$ is a non-archimedean local field.
Let $P_F\subset I_F$ denote the wild inertia, $\Gamma_F^t:=\Gamma_F/P_F$  the tame quotient and  $I_F^t=I_F/P_F\subset\Gamma_F^t$ the tame inertia.
For each Frobenius  lift $\sigma^t\in \Gamma_F^t$ 
there is an isomorphism 
$$\eta_F:\Gamma^t_F\cong \prod_{\ell\neq p}\bbZ_{\ell}(1)\rtimes \widehat{\bbZ},$$
which identifies $I_F^t$ with the first factor. 
For $\ell\neq p$ a prime, we fix an isomorphism $\tau_\ell:\bbZ_\ell(1)\cong \bbZ_\ell$, and we  write $t_\ell$ for the composition $I_F\rightarrow I_F^t\rightarrow \bbZ_\ell(1)\rightarrow \bbZ_\ell$.

\subsubsection{}\label{sssec: l-adic monodromy}Let $\bfG$ be a reductive group over $\bar{\bbQ}_\ell$ and let $\rho:\Gamma_F\rightarrow \bfG(\bar{\bbQ}_\ell)$ be a continuous homomorphism, where $\bfG(\bar{\bbQ}_\ell)$ is equipped with the $\ell$-adic topology. By Grothendieck's $\ell$-adic monodromy theorem \cite[8.2]{De4}, there is an open subgroup $H\subset I_F$ and a nilpotent element $N\in \mathrm{Lie}(\bfG)$ such that $\rho(h)=\exp(t_{\ell}(h)N)$ for all $h\in H$.

We define a map $\rho':W_F\rightarrow \bfG(\bar{\bbQ}_\ell)$ by $$\rho'(\sigma^n\gamma)=\rho(\sigma^n \gamma)\exp(-t_{\ell}(\gamma)N)$$
where $\sigma$ is a lift of Frobenius and $\gamma\in I_F$. Then the pair $(\rho',N)$ gives rise to Weil--Deligne $\bfG$-representation $\rho^{\WD}$, whose $\bfG(\bar{\bbQ}_\ell)$-conjugacy class is independent of  the choice of $\sigma \in W_F$ lifting Frobenius and the isomorphism $\tau:\bbZ_\ell(1)\cong \bbZ_\ell$ (see \cite[8.11]{De4}). 

 We say that  $\rho$  is \textit{unipotently ramified} if it  factors through $\Gamma_F^t$ and $\rho(\gamma)$ is unipotent for all $\gamma\in I_F$. In this case, the corresponding Weil--Deligne $\bfG$-representation $\rho^{\WD}$ is  also unipotently ramified.

\subsubsection{}\label{sssec:logetalegps}

Now let $F$ be a $p$-adic field. Let $\calC=\Spec \calO_F[u],$ and $s_0 \in \calC,$  the closed point given by $u=\pi=0.$ 
Let $R^h$ denote the henselization of $\calO_F[u]$ at $s_0$. We set $\calC^h=\Spec R^h$ and we write 
$\calC[u^{-1}]$ (resp. $\calC^h[u^{-1}]$) for $\Spec \calO_F[u,u^{-1}]$ (resp. $\Spec R^h[u^{-1}]$). 
We equip $\calC$ with the log structure coming from the divisor $u=0.$

Fix an algebraic closure $\bar k_F$ of $k_F.$ Let $\bar s_0 = \Spec \bar k_F.$ 
We equip $s_0$ with the log structure $\bbN^+ \mapsto 0$ and $\bar s_0$ with the induced structure of log geometric point (cf. \cite{Illusieoverview}).

Now let $F'$ be a non-archimedian local field, and suppose we are given a map $R^h \rightarrow \O_{F'}$ 
sending $u$ to a uniformizer $\pi_{F'} \in \O_{F'},$ and inducing an isomorphism on residue fields.  For example, we may take $F'=F$ and the map $\delta:R^h\rightarrow F$, $u\mapsto \pi$ or  $F'=k_F\ls$ and the map $\gamma:R^h\rightarrow k_F\ls $, $u\mapsto u$.
We denote by $\bar \eta$ a geometric point over the generic point $\eta \in \Spec \O_{F'}.$

\begin{lemma}\label{lem:tamefg} There is a canonical commutative diagram of (log) \'etale fundamental group isomorphisms 
\[\xymatrix{ \pi_{1,\et}(\Spec F', \bar \eta)^t \ar[r]\ar[d] & \pi_{1,\et}(\calC^h[u^{-1}], \bar \eta) \ar[d] \\
\pi_{1,\et}^{\log}(s_0, \bar s_0) \ar[r] & \pi_{1,\et}^{\log}(\calC^h, \bar s_0)}
\]
where $\pi_{1,\et}(\Spec F', \bar \eta)^t \simeq \Gamma^t_{F'}$ is the tame quotient of 
$\pi_{1,\et}(\Spec F', \bar \eta).$ Moreover the isomorphisms are compatible with the natural projection 
of each group to $\pi_{1,\et}(s_0,\bar s_0).$
\end{lemma}
\begin{proof} This is contained in   \cite[4.7]{Illusieoverview}. The vertical isomorphisms are induced by the 
specialization map $\bar \eta \rightarrow \bar s.$ Note that, by definition, $\pi_{1,\et}(\calC^h[u^{-1}], \bar \eta)$ is,
identified with the tame fundamental group in \cite[4.7(c)]{Illusieoverview}, as the generic point of the divisor $u=0,$ 
has characteristic $0.$ 
\end{proof}

\subsubsection{}Now, let $F$ again be a $p$-adic {\em local} field. The previous result allows us to relate  Weil--Deligne representations for $F$ and $k_F\ls$ arising from an $\ell$-adic local system over $\calC^h[u^{-1}]$. 

We let $\mathcal{L}$ be a $\bfG(\bar{\bbQ}_\ell)$-local system over $\calC^h[u^{-1}]$. This corresponds to a continuous representation 
$$\rho:\pi_1(\calC^h[u^{-1}],\overline{\eta})\rightarrow \bfG(\bar{\bbQ}_\ell),$$ and we write $$\rho_F:\Gamma_F\rightarrow \bfG(\bar{\bbQ}_\ell),\ \  \rho_{k_F\ls}:\Gamma_{k_F\ls}\rightarrow \bfG(\bar{\bbQ}_\ell) $$
for the representations obtained by pullback along the maps $\delta$ and $\gamma$ respectively. We write $\rho^{\WD}_{F}$ (resp. $\rho^{\WD}_{k_F\ls}$) for  the  Weil--Deligne $\bfG$-representation associated to $\rho_F$ (resp. $\rho_{k_F\ls}$).

\begin{cor}\label{cor: comparison mixed/equal WD reps} Suppose $\rho_F^{\WD}$ is (URFS). Then $\rho^{\WD}_{k_F\ls}$ is (URFS) and we have an equality $$[\rho^{\WD}_{F}]=[\rho^{\WD}_{k_F\ls}]\in \Phi(q,\bfG,\bar\bbQ_\ell),$$ where $q$ is the size of the residue field $k_F$ of both local fields $F$ and $k_F\ls$ (cf. Remark \ref{rem: different local fields}).
\end{cor}
\begin{proof} Using Lemma \ref{lem:tamefg}, we have an identification 
$$ \theta: \Gamma_F^t \xrightarrow{\sim}  \pi_{1,\et}^{\log}(\calC^h, \bar s_0) \xrightarrow{\sim} \Gamma_{k_F\ls}^t$$ 
which identifies the tame inertia subgroups $I_F^t\cong I_{k_F\ls}^t,$ and is compatible with Frobenius elements. 

Let  $\sigma_F\in \Gamma_F$ and $\sigma_{k_F\ls}\in \Gamma_{k_F\ls} $ be any Frobenius lifts which are compatible with $\theta,$ 
and $t_{F,\ell}:I_F\rightarrow \bbZ_\ell$ and $t_{k_F\ls,\ell}:I_{k_F\ls}\rightarrow\bbZ_\ell$ the homomorphisms corresponding to a choice of isomorphism $\tau_\ell:\bbZ_\ell(1)\cong\bbZ_\ell$. Using these choices, one sees from construction of $\rho^{\WD}_{F}$ and $\rho^{\WD}_{k_F\ls}$ that  we have an equality $\rho_F^{\WD}=\rho_{k_F\ls}^{\WD}$ in $\Phi^\square(q,\bfG,\bar{\bbQ}_\ell)$. The result follows a  fortiori.	
\end{proof}

\subsection{Weil--Deligne representations associated to isocrystals}\label{ssec: p-adic comparison}

\subsubsection{}\label{sssec:isocrystalintro} In this subsection we discuss Weil--Deligne representations associated to isocrystals. 
In particular, we compare two definitions of the Weil--Deligne representation 
associated to an isocrystal. This will allow us to compare Weil--Deligne representations arising from crystalline cohomology 
for local fields of mixed characteristic and equal characteristic $p$.

\subsubsection{}\label{sssec:micsetup}
Let $k$ be a perfect field of characteristic $p,$  $W = W(k)$ its ring of Witt vectors, and $F_0 = W[1/p].$
We denote by $\O \subset F_0\lps u \rps$ the subring of power series which converge for $|u| < 1.$ 
The ring $\O$ has a continuous Frobenius $\varphi,$ extending the Frobenius on $W,$ given by $\varphi(u) = u^p.$

We denote by $\Isoc^{\varphi}_{/\O}$ the category of finite free $\O$-modules $\M,$ equipped with a 
 $\varphi$ semi-linear map $\varphi_{\M},$ which induces an isomorphism $\varphi^*\M \simeq \M,$ 
 and a connection $\nabla_{\M}: \M \rightarrow \M\otimes_{\O}\Omega^1_{\O/F_0}[1/u],$ which has at most a logarithmic pole 
  with nilpotent residue at $u=0.$ 
 We require that the isomorphism $\varphi^*\M \simeq \M$ is compatible with connections. 

Let $\partial$ denote the differential operator $u\frac{d}{du}.$ Then $\partial$ acts on $\M$ in $\Isoc^{\varphi}_{/\O}$ 
as the operator $u \frac{d}{du}\circ \nabla_{\M}.$ We denote by $\ell_u$ a formal variable, and set $\partial(\ell_u) = 1,$ so that 
$\ell_u$ is a formal logarithm on $\O.$ We may then extend $\partial$ to a derivation on the polynomial ring 
$\O[\ell_u]$ over $\O.$ We extend $\varphi$ to $\O[\ell_u]$ by setting $\varphi(\ell_u) = p\ell_u,$ and we define 
an $\O$-linear derivation $N$ on $\O[\ell_u]$ by $N(\ell_u) = 1.$

\begin{lemma}\label{lem:delignemoduleI} Let $\M$ be in $\Isoc^{\varphi}_{/\O},$ and 
set $D(\M) =  (\O[\ell_u]\otimes_{\O} \M)^{\partial=0}.$ 
Then 
\begin{enumerate} 
\item The canonical map 
$$ D(\M) \otimes_{F_0}\O[\ell_u] \rightarrow \M\otimes_{\O}\O[\ell_u]. $$
is an isomorphism.
\item The operators $\varphi\otimes \varphi_{\M}$ and $N\otimes 1$ on $\O[\ell_u]\otimes_{\O} \M$ 
induce operators $\varphi = \varphi_{D(\M)}$ and $N$ on $D(\M).$ Moreover, $\varphi$ on $D(\M)$ is an automorphism, 
and satisfies 
$$N\varphi = p\varphi N.$$
\end{enumerate}
\end{lemma}
\begin{proof} This is well known, but we sketch a proof. 

For $r \leq 1,$ denote by $D(0,r)$ the open $p$-adic disc of radius $r$ centered at $0,$ over $F_0.$  
Let $\O(D(0,r))$ be the ring of functions on $D(0,r).$ 
By the $p$-adic version of Fuchs' theorem \cite[Thm.~13.2.2]{Kedlayapadic}, there exists $r > 0,$ and an isomorphism 
$\O(D(0,r))^d \simeq \M|_{D(0,r)},$ such that with respect to this basis we have for $m \in  \M|_{D(0,r)},$ 
$$ \nabla_{\M}(m) = dm + N_0 \frac {du} u $$
for some nilpotent matrix $N_0 \in M_d(F_0).$ 

Let $M \subset \M|_{D(0,r)},$ be the $F_0$-vector space spanned by the chosen basis. 
Then  $M = \ker(\partial^d).$ 
Using the isomorphism of $\O[\partial]$-modules, $\varphi^*(\M) \simeq \M,$ one sees that $\varphi_{\M}$ 
induces an automorphism of $M,$ and that any section in $M$ actually 
converge on $D(0,r^{1/p});$ that is $M \subset \M|_{D(0,r^{1/p})}.$ 
Continuing the argument, we get $M \subset \M = \cap_{r< 1} \M|_{D(0,r)}.$ 

Now using that $M = \ker(\partial^d),$ one sees easily that 
$$  (F_0[\ell_u]\otimes_{F_0}M)^{\partial=0} \simeq (F_0[\ell_u]\otimes_{F_0}\M)^{\partial=0} 
\simeq (\O[\ell_u]\otimes_{\O}\M)^{\partial=0} = D(\M)  $$
has dimension $d$ over $F_0,$ and that 
$$ D(\M)\otimes_{F_0}F_0[\ell_u] \simeq M\otimes_{F_0}F_0[\ell_u].$$
This implies (1). For (2), the compatibility of $\varphi^*\M \simeq \M$ with $\nabla_{\M}$ implies 
that $\partial \varphi = p\varphi\partial$ on $\M$ and hence on $F_0[\ell_u]\otimes_{F_0}\M.$ 
This implies that $D(\M)$ is stable under $\varphi.$ Moreover, 
$\varphi_{D(\M)}$ is an automorphism as $\varphi_{\M}$ induces an automorphism of $M.$ 
The relation $N\varphi = p\varphi N$ follows from the corresponding relation on $F_0[\ell_u].$
\end{proof}

\subsubsection{}\label{sssec:Delginemodules} We denote by $\Mod^{\varphi,N}_{/F_0}$ the category of finite dimensional 
$F_0$-vector spaces equipped with a Frobenius semi-linear operator $\varphi,$ and a nilpotent, linear 
operator $N$ satisfying $\varphi N = p\varphi N.$ Thus Lemma \ref{lem:delignemoduleI} yields a functor 
$$ \Isoc^{\varphi}_{/\O} \rightarrow \Mod^{\varphi,N}_{/F_0}\; \quad \M \mapsto D(\M) =  (\O[\ell_u]\otimes_{\O} \M)^{\partial=0} .$$

Let $D(\M)' = \M/u\M,$ equipped with the Frobenius induced from $\varphi_{\M}.$ 
The operator $\partial$ induces a nilpotent operator on $D(\M)',$ which we again denote  by $\partial.$
We make $D(\M)'$ into an object of $\Mod^{\varphi,N}_{/F_0}$ by letting $N$ act as $-\partial.$ 

\begin{lemma}\label{lem:delignemoduleII} There is a natural isomorphism 
$$ D(\M) \simeq D(\M)'$$
in $\Mod^{\varphi,N}_{/F_0}.$ 
\end{lemma}
\begin{proof} We will show that the composite of the natural maps 
$$ D(\M) =  (\O[\ell_u]\otimes_{\O} \M)^{\partial=0} \rightarrow \M \rightarrow \M/u\M = D(\M)' $$
is an isomorphism in $\Mod^{\varphi,N}_{/F_0}.$ 

Note that, using the notation of Lemma \ref{lem:delignemoduleI}, the bijection
$$ D(\M) = (F_0[\ell_u]\otimes_{F_0}M)^{\partial=0} \overset {\ell_u \rightarrow 0} \rightarrow M $$
is compatible with Frobenius and intertwines the operators $N$ on $D(\M)$ and $-\partial$ on $M.$ 
Now the lemma follows, because we already saw above that $M \rightarrow D(\M)'$ is a bijection, 
which is compatible with $\varphi$ and $\partial,$ by construction.
\end{proof}

\subsubsection{}\label{sssec:Fisocrystals} 
Let $C$ be a smooth scheme of finite type over $k.$ We refer to \cite{Kedlayaisoc}, \cite{Shihoisoc} for different notions 
(convergent, overconvergent ...) of $F$-isocrystal and log $F$-isocrystals on $C,$ which we will use below. 
Suppose that $C$ is equipped with a fine log structure \cite{KatoFI}. We denote by $\Isoc^{\dag,\varphi}_{C/k}$ (resp.~$\Isoc^{c,\varphi}_{C/k}$) 
the category of overconvergent (resp.~convergent) log $F$-isocrystals on $C$ over $k.$ These are naturally Tannakian categories, 
and there is a functor 
$$ \Isoc^{\dag,\varphi}_{C/k} \rightarrow \Isoc^{c,\varphi}_{C/k}.$$
To simplify the discussion, we work with $\Isoc^{\dag,\varphi}_{C/k}$ in this section, although some of the constructions 
could also be made for $\Isoc^{c,\varphi}_{C/k}.$

\subsubsection{}\label{sssec:assphiNmodule}
As above, let $s_0 = (k, \bbN^+)$ denote the log point and equip $\O_{F_0}$ with the log structure coming from the maximal ideal, 
and $\O_{F_0}[u]$ with the log structure coming from $u=0.$ 
Recall that the module of logarithmic differentials $\Omega^1_{\O_{F_0}[u]/\O_{F_0}}(\log)$ is defined to be the $\O_{F_0}[u]$-submodule 
of $\Omega^1_{\O_{F_0}[u,u^{-1}]/\O_F}$ spanned by $\frac{du}u,$ and that we then have 
$$\Omega^1_{\O_{F_0}}(\log) \simeq \Omega^1_{\O_{F_0}[u]/\O_{F_0}}(\log)\otimes_{\O_{F_0}[u]}\O_{F_0}.$$
An object of $\Isoc^{\dag,\varphi}_{s_0/k}$ consists of an $F_0$-vector space $V$ together with a semi-linear Frobenius $\varphi_V,$ 
and a connection $$\nabla_V:V \rightarrow V\otimes_{F_0}\Omega^1_{\O_F}(\log),$$ which is compatible with $\varphi_V.$ 
For $v \in V,$ define $$\partial(v) = \langle \nabla_V(v), u\frac d {du} \rangle.$$ Then $V$ is determined 
by $\varphi_V,$ and the linear operator $\partial,$ subject to the condition $$p \varphi_V \partial = \partial \varphi_V.$$ 
We may thus view $V$ as an object of $ \Mod^{\varphi,N}_{/F_0},$ by letting $N$ act as $-\partial.$

\subsubsection{}\label{sssec:WDrepsisoc} 
Let $k'/k$ be a finite extension, $s'_0 = (k', \bbN^+),$ and $F_0' = W(k')[1/p].$ 
A morphism of log schemes $x: s'_0 \rightarrow C$ gives rise to a $\otimes$-functor 
$$ x^*:  \Isoc^{\dag,\varphi}_{C/k} \rightarrow  \Isoc^{\dag,\varphi}_{s'_0/k'} \simeq  \Mod^{\varphi,N}_{/F_0'}; \quad \M \mapsto x^*\M.$$

We now assume that $k$ is a finite field, and we write $q = p^s = |k'|.$ Fix an embedding $F_0' \rightarrow \bar \Q_p.$
Then there is a $\otimes$-functor 
$$ \Mod^{\varphi,N}_{/F'_0} \rightarrow \Rep_{/\bar \Q_p} \WD_{k'((u))}; \quad \N \mapsto \omega(\N)\otimes_{F'_0}\bar \Q_p,$$ 
where $\omega(\N)$ is the underlying $F_0'$-vector space of $\N,$ equipped with the linear operator $N,$ 
and the unramified $W_F$ action obtained by sending the arithmetic Frobenius to $\varphi^{-s}.$ 
Thus, if $x: s'_0 \rightarrow C$ is as above, we obtain a $\otimes$-functor 
 $$ \omega^{\log}_x:  \Isoc^{\dag,\varphi}_{C/k} \rightarrow \Rep_{/\bar \Q_p} \WD_{k'((u))}.$$ 
Finally, as the category on the right is $\bar \Q_p$-linear, we may extend this to a $\otimes$-functor 
$$ \omega^{\log}_x:  \Isoc^{\dag,\varphi}_{C/k,\bar \Q_p}:= \Isoc^{\dag,\varphi}_{C/k}\otimes_{\Q_p}\bar \Q_p \rightarrow \Rep_{/\bar \Q_p} \WD_{k'((u))}.$$

\subsubsection{}\label{sssec:WDrepsisoccurve}
Now let $\bar C$ be a smooth curve over $k,$ equipped with a finite collection of closed points $Z,$ 
and let $C = \bar C - Z.$ We equip $\bar C$ with the log structure arising from the divisor $Z.$ 

Fix $x \in Z,$ let $k'$ be the residue field of $x,$ and identify the complete local ring of $\bar C$ at $x$ 
with $k' \lps u \rps.$ Recall that Marmora \cite[\S 3]{Marmoraepsilon} has defined a $\otimes$-functor 
$$ \omega^\dag_x:  \Isoc^{\dag,\varphi}_{C/k,\bar \Q_p}  \rightarrow \Rep_{/\bar \Q_p} \WD_{k'((u))}.$$

\begin{lemma}\label{lem:compatisoc} With the above assumptions the diagram of $\otimes$-functors
\[\xymatrix{\Isoc^{\dag,\varphi}_{\bar C/k,\bar \Q_p}  \ar[r]^{\!\!\!\!\!\!\!\omega^{\log}_x}\ar[d] &  \Rep_{/\bar \Q_p} \WD_{k'((u))} \ar@{=}[d] \\
\Isoc^{\dag,\varphi,}_{C/k,\bar \Q_p} \ar[r]^{\!\!\!\!\!\!\!\omega^\dag_x} & \Rep_{/\bar \Q_p} \WD_{k'((u))}}
\]
commutes up to natural equivalence.\
\end{lemma}
\begin{proof} This follows from Lemma \ref{lem:delignemoduleII}, and unwinding the definitions above, and 
those in \cite[\S 3]{Marmoraepsilon}.
\end{proof} 

\subsubsection{}\label{sssec:FisocrystalsG} Let $\bfG$ be a connected reductive group over $\bar \Q_p.$ 
For $C$ as above, an overconvergent log $F$-isocrystal with $\bfG$-structure on $C,$ is an exact $\otimes$-functor 
$$ \Rep_{\bar \Q_p} \bfG \rightarrow \Isoc^{\dag,\varphi}_{C/k,\bar \Q_p}, $$
where $\Rep_{\bar \Q_p} \bfG$ denotes the category of $\bar \Q_p$-linear representations of $\bfG.$ 
We denote by $\bfG$-$\Isoc^{\dag,\varphi}_{C/k,\bar \Q_p}$ the category of 
overconvergent log $F$-isocrystals with $\bfG$-structure on $C,$ and we define $\bfG$-$\Isoc^{c,\varphi}_{C/k,\bar \Q_p}$ 
in an analogous way.

For $x:s_0' \rightarrow \bar C,$ as above, and $\calL$ in $\bfG$-$\Isoc^{\dag,\varphi}_{\bar C/k,\bar \Q_p},$ we 
can consider the composite 
$$ \Rep_{\bar \Q_p} \bfG \overset {\calL} \rightarrow \Isoc^{\dag,\varphi}_{\bar C/k,\bar \Q_p} 
\overset {\omega^{\log}_x} \rightarrow \Rep_{/\bar \Q_p} \WD_{k'((u))} .$$
By the Tannakian formalism, this corresponds to a map of algebraic groups 
$$ \rho_{\calL,x}^{\log}: \WD_{k'((u))} \rightarrow \bfG$$ 
over $\bar \Q_p.$
We may likewise define a representation 
$\rho_{\calL,x}^{\dag}: \WD_{k'((u))} \rightarrow \bfG$ using the functor $\omega^\dag_x$ in place of $\omega^{\log}_x.$

\begin{lemma}\label{lem:WDGreps} Let $Z \subset \bar C,$ and $\calL$ be as above. 
Then for $x \in Z,$ there is an equivalence of $\bfG$-representations 
 $$\rho_{\calL,x}^{\log} \simeq \rho_{\calL,x}^{\dag}.$$
  \end{lemma}
  \begin{proof} This follows immediately from Lemma \ref{lem:compatisoc}.
  \end{proof}

\subsection{Local global compatibility in Lafforgue's theorem}\label{ssec: Lafforgue LGC}
The aim of this subsection is to prove a slight refinement of the local-global compatibility theorem for $\ell$-adic companions in \cite{Laf}, 
and its crystalline analogue \cite{Abe}. 
For the rest of the paper we fix an isomorphism $\iota_\ell:\bar{\bbQ}_\ell\cong \bbC$ for every prime $\ell$. 

\subsubsection{}
Let $C$ be a smooth, connected curve over a finite field $\bbF_q$ of characteristic $p$ with algebraic closure 
$\bar \bbF_q,$ and let $F:=\bbF_q(C)$ be the function field of $C.$ Let $\ell$ be a prime.  Let $\calL_{\ell}$ be a lisse 
$\bar{\bbQ}_\ell$-sheaf on $C$ if $\ell \neq p,$ and an object of $\Isoc^{\dag,\varphi}_{C/\bbF_q,\bar \Q_p},$ if $\ell = p;$ (here $C$ is equipped with the trivial log structure).  We let $n$ be the rank of $\calL_\ell$.
The following lemma is well known.

\begin{lemma}\label{lem:constdet} The restriction of $\det \calL_{\ell}$ to $C_{\bar \bbF_q}$ has finite order.
\end{lemma}
\begin{proof} When $\ell \neq p,$ this is \cite[Proposition 1.3.4]{De3}. For $\ell = p,$ this is \cite[Lemma 6.1]{Abe2}. 
\end{proof}

\subsubsection{}
Suppose $\ell \neq p,$ and let $\beta\in \bar{\bbQ}_\ell^\times$ be an $\ell$-adic unit. Then there is a continuous homomorphism 
$$\mathrm{Gal}(\bar{\bbF}_q/\bbF_q)\rightarrow \bar{\bbQ}_\ell^\times$$ 
taking the arithmetic $q$-Frobenius to $\beta^{r},$ where $q = p^r.$ 
This representation corresponds to a rank 1 lisse $\bar{\bbQ}_\ell$-sheaf on $\Spec\bbF_q$ and we let $\beta$ denote its pullback to $C.$

Now suppose $\ell = p.$ Then for any $\beta\in \bar{\bbQ}_\ell^\times,$ 
we let  $\calL_\ell\otimes\beta$ denote the overconvergent isocrystal obtained from $\calL_{\ell}$ by multiplying 
the isomorphism $\varphi^*(\calL_{\ell}) \simeq \calL_{\ell}$ by $\beta^{-1}.$

By Lemma \ref{lem:constdet}, for any $\ell$ and $\calL_{\ell},$ there exists a $\beta \in \bar{\bbQ}_\ell^\times$ 
which is an $\ell$-adic unit if $\ell \neq p,$ such that $\det \calL_{\ell}\otimes \beta$ has finite order. We denote this element by 
$\beta(\calL_{\ell}).$

\subsubsection{}
Now  let $\overline{C}$ denote the smooth compactification of $C.$ 
For $x\in \overline{C}$ a closed point,  let $F_x$ denote the completion of $F$ at $x,$ and $\kappa(x)$ the residue 
field of $x.$  
We can associate a Weil--Deligne $\GL_{n,\bar{\bbQ}_\ell}$-representation $\rho_{\mathcal{L}_\ell,x}^{\WD}$ to 
$\calL_{\ell}$ and $x:$ For $\ell \neq p,$ this is done by considering the corresponding 
$\ell$-adic representation 
$$\rho_{\mathcal{L}_\ell,x}:\mathrm{Gal}(\bar{F}_x/F_x)\rightarrow \GL_n(\bar{\bbQ}_\ell),$$
and applying  \S\ref{sssec: l-adic monodromy}. For $\ell = p$ this is done in \ref{sssec:WDrepsisoccurve}.
Via the isomorphism $\iota_\ell: \bar{\bbQ}_\ell\cong \bbC$ we may consider 
$\rho_{\mathcal{L}_\ell,x}^{\WD}$ and its Frobenius semisimplification $\rho_{\mathcal{L}_\ell,x}^{\WD,\Fss}$ as a $\GL_{n,\bbC}$-representation.

\begin{thm}\label{thm: Lafforgue local-global} 
Suppose that $\calL_{\ell}$ is irreducible (resp.~semi-simple), and satisfies the following condition 
$$\text{\rm For every factor  $\calL'_{\ell} \subset \calL_{\ell},$ $\beta(\calL'_{\ell})$ is algebraic and a $p$-unit} $$
Then there exists a number field $E \subset \mathbb C,$ and for each prime $\ell'$ an irreducible (resp. semisimple) lisse 
$\bar \Q_{\ell'}$-sheaf  on $C$ (resp.~object of $\Isoc^{\dag,\varphi}_{C/\bbF_q,\bar \Q_p}$ if $\ell' = p$) $\calL_\ell'$ such that 
for every closed point $x \in \bar C$, $\rho_{\calL_{\ell},x}^{\WD,\Fss}$ is defined over $E$ and we have 
$$ \rho_{\mathcal{L}_{\ell'},x}^{\WD,\Fss} \sim_{\GL_{n,\bbC}} \rho_{\mathcal{L}_\ell,x}^{\WD,\Fss}.$$
	\end{thm}

\begin{proof} We remark that for $x \in C,$ $\rho_{\mathcal{L}_\ell,x}^{\WD}$ and $\rho_{\mathcal{L}_{\ell'},x}^{\WD}$ 
are unramified, and the equality in the Theorem amounts to the statement that under these representations, 
the characteristic polynomials of Frobenius are equal and in $E[T].$
It suffices to consider the case  where $C$ is geometrically irreducible (in which case $\overline{C}$ is also geometrically irreducible) and $\calL_{\ell}$ is an irreducible lisse sheaf.

We first consider the case that $\calL_\ell$ has determinant of finite order. 
In this case $\calL_{\ell}$ corresponds to a cuspidal automorphic representation $\pi$ of $\GL_n(\bbA_F),$ 
and the association $\calL_{\ell} \mapsto \pi$ is compatible with $L$ and $\varepsilon$-factor of pairs 
\cite[Thm.~VII.3]{Laf}, \cite[Thm.~4.2.2]{Abe}. As in \cite[Thm.~VII.6]{Laf} and \cite[Cor.~VII.5]{Laf}, respectively, 
this implies the existence of the number field $E$ such that  $\rho_{\mathcal{L}_\ell,x}^{\WD,\Fss} $ is defined over $E$ 
for all $x,$ and that the association $\calL_{\ell} \mapsto \pi$ is compatible 
with the local Langlands correspondence at all closed points $x \in \bar C.$ That is, $\rho_{\mathcal{L}_\ell,x}$ corresponds to 
the local factor $\pi_x$ of $\pi,$ under the Local Langlands correspondence. Now applying the correspondence 
of \cite[Thm.~VII.3]{Laf} and \cite[Thm.~4.2.2]{Abe} to $\pi,$ we obtain an $\calL_{\ell'}$ such that 
$\calL_{{\ell'},x}$ corresponds to $\pi_x$ for all closed points $x.$ This implies 
$ \rho_{\mathcal{L}_{\ell'},x}^{\WD,\Fss} \sim_{\GL_{n,\bbC}} \rho_{\mathcal{L}_\ell,x}^{\WD,\Fss}.$

Now for any $\mathcal{L}_\ell$ satisfying the condition of the theorem, let $\beta = \beta(\calL_{\ell})\in \bar\Q^\times$ and $\beta^{1/n} \in \bar \Q^\times$ be any $n$th root of $\beta.$ Then $\calL_{\ell}^\beta: = \calL_\ell\otimes \beta^{1/n}$ has 
determinant of finite order. Applying the case considered above to $\calL_{\ell}^\beta,$ we obtain a number field 
$E^{\beta} \subset \mathbb C$ and an $\bar \Q_{\ell'}$-sheaf (resp. object of 
$\Isoc^{\dag,\varphi}_{C/\bbF_q,\bar \Q_p}$) $\calL_{\ell'}^\beta$. We may now take 
$E = E^\beta(\beta^{1/n}),$ and $\calL_{\ell'} = \calL^{\beta}_{\ell'}\otimes \beta^{-1/n}.$
\end{proof}

\section{Independence of $\ell$}\label{sec: independence of l}

In  \S5.1  and \S5.2, we use the results of the previous section to prove an independence of $\ell$ result for the monodromy at the boundary in the special fiber of Shimura varieties. This is used to prove the main theorem concerning abelian varieties in \S5.3.

\subsection{Local monodromy at the boundary}\label{subsec: local monodromy}

\subsubsection{}We return to the setting of \S\ref{sec: integral models}.	Thus let $(\bfG,X,\calG)$ be a strongly admissible Shimura datum and fix a Hodge embedding 
$\iota: (\bfG,X) \rightarrow (\mathbf{GSp}(V), S^{\pm})$ as in Proposition \ref{prop: integral LHE}. For $\rmK=\calG(\bbZ_p)$, $\rmK^p\subset \bfG(\bbA_f^p)$ a compact open subgroup and  $\rmK=\rmK_p\rmK^p$, we let $\scrS_{\rmK}$ denote the associated  integral model over $\calO_E$. Throughout this subsection, we will make the assumption  that $G=\bfG_{\bbQ_p}$ is quasi-split.

Now let $\ell\neq p$ be a prime, and  assume as in \S\ref{sssec: l-indep SV interior} that  $\rmK$ is of the form $\rmK_\ell\rmK^\ell$, with $\rmK_\ell\subset \bfG(\bbQ_\ell)$, $\rmK^\ell\subset \bfG(\bbA_f^\ell)$. Then we have a $\bfG(\bbQ_\ell)$-local system $\bbL_{\ell}$ on $\scrS_{\rmK,k_E}$ arising from the pro-\'etale covering  $\varprojlim_{\rmK_\ell'\subset \rmK_\ell}\scrS_{\rmK'_\ell\rmK^\ell,k_E}$.
By Corollary \ref{cor: compatible system SV}, the $(\bbL_\ell)_{\ell \neq p}$ form a compatible system of $\bfG$-local systems on $\scrS_{\rmK,k_E}$ defined over $\bbQ$ in the following sense: 
For each $x\in \scrS_{\rmK}(\bbF_q)$,  we let $\gamma_{x,\ell}\in \Conj_{\bfG}(\bbQ_\ell)$ be the element corresponding to the action of local Frobenius on the geometric  stalk $\bbL_{\ell,\overline{x}}$. Then there exists $\gamma_x\in \mathrm{Conj}_{\bfG}(\bbQ)$ such that for all $\ell\neq p$, we have $\gamma_x=\gamma_{x,\ell}$ in $\Conj_{\bfG}(\bbQ_\ell)$.
	
	For notational simplicity, we will write $\calS_{\rmK}$ for the special fiber $\scrS_{\rmK,k_E}$ of $\scrS_{\rmK}$.

\subsubsection{}\label{subsubsec:curvedefn}
Let $\scrS_{\rmK}^\Sigma$ denote the toroidal compactification of $\scrS_{\rmK}$ associated  to a choice of complete smooth admissible rational polyhedral cone decomposition, and  let $\calS_{\rmK}^\Sigma$ denote the special fiber of $\scrS_{\rmK}^\Sigma$. Then we have an open immersion $\calS_{\rmK}\hookrightarrow \calS_{\rmK}^\Sigma$ and we let $\partial\calS_{\rmK}$ denote the complement $\calS_{\rmK}^\Sigma\setminus \calS_{\rmK}$.

Let $k_E'=\bbF_q$ be a finite extension of $k_E$ and fix $x\in \partial\calS_{\rmK}(k_E')$. Let $\pi:\bar C\rightarrow \calS_{\rmK,k_E'}^\Sigma$ be a morphism from a smooth geometrically connected curve $\bar C$ over $k_E'$ satisfying the following two properties:
	
	\begin{enumerate}\item There exists $c\in \bar C(k_E')$ with $\pi(c)=x$.
		\item $\pi(\bar C)$ intersects the interior $\calS_{\rmK,k_E'}$.
		\end{enumerate}
	We let $C=\pi^{-1}(\calS_{\rmK,k_E'})$ which is an open subscheme of $\bar C$. Fix a geometric point of 
	$\overline{u} \in C$. Then the pullback of $\bbL_\ell$ to $C$ gives rise to a representation 
	$$\rho_{C,
		\ell}:\pi_1(C,\overline{u})\rightarrow \bfG(\bbQ_\ell),$$ 
	which is known to be semisimple by a result of Zarhin \cite{Zarhin1}, \cite{Zarhin2}.
	
	We let $F_c$ denote the completion of the fractional field $k(C)$ at the place corresponding to $c$, and we let $\rho_{C,\ell,c}^{\WD}$ denote the Weil--Deligne $\bfG_{\bbC}$-representation associated to $\rho_{C,\ell}|_{F_c}$, induced by our fixed choice of isomorphism $\iota_\ell:\bar\Q_\ell\cong \bbC$ for every prime $\ell$. 
	
	\begin{lemma}\label{lem: WD rep is URFS}
		For any $\ell\neq p$, $\rho^{\WD}_{C,\ell,c}$ is Frobenius semisimple and unipotently ramified.
	\end{lemma}

\begin{proof} It suffices to prove the result for the Weil--Deligne $\GL_{n,\bbC}$-representation $r\circ \rho^{\WD}_{C,\ell,c}$ induced by a faithful representaion $r:\bfG\rightarrow \GL_{n}$. We take $r$ to be the representation induced by the Hodge embedding $\iota.$ Then $r\circ \rho^{\WD}_{C,\ell,c}$ arises from the the action of $\Gamma_{F_c}$ on the $\ell$-adic Tate module of an abelian variety over $F_c$ with semistable reduction. The result then follows from \cite[IX, 3.5]{SGA7}.
	\end{proof}

By the lemma above, for each $\ell\neq p$, we can associate to $\pi: \bar C\rightarrow\calS_{\rmK,k_E'}^{\Sigma}$ as above an element 
$$[ \rho^{\WD}_{C,\ell,c}]\in \Phi(q,\bfG,\bbC).$$
	
\subsubsection{} Now let $\vartheta:\bfG_{\mathbb C} \rightarrow \GL_n$ be any representation, defined over $\mathbb C.$ 
The  $\bfG(\bbQ_\ell)$-local system $\bbL_\ell$ induces  a $\bar \Q_{\ell}$-local system of rank $n$ on 
$\calS_{\rmK}$ via $\vartheta,$ and we write $\calL_{\ell}^\vartheta = \calL_{\ell}^{C,\vartheta}$ for its pullback to $C$. The representation $\vartheta$ also gives rise to a corresponding Weil--Deligne $\GL_n$-representations $\vartheta\circ \rho_{C,\ell,c}^{\WD}$.
	
\begin{prop}\label{prop: l-indep boundary monodromy GLn}
For any $\ell, \ell'\neq p$, we have  $$[\vartheta\circ \rho_{C,\ell,c}^{\WD}]=[\vartheta\circ \rho_{C,\ell',c}^{\WD}]\in \Phi(q,\GL_n,\bbC),$$ i.e. $\vartheta\circ \rho_{C,\ell,c}^{\WD}$ and $\vartheta\circ \rho_{C,\ell',c}^{\WD}$ are conjugate by an element of $\GL_n(\bbC)$. 
	\end{prop}

\begin{proof} By \cite[Lemma 5.3.3]{KZ}, which applies in our setting,  the local system $\calL^\vartheta_{\ell}$ satisfies the condition of Theorem 
\ref{thm: Lafforgue local-global}. 
 We let $\calL'_{\ell'}$ denote the semisimple $\bar{\bbQ}_{\ell'}$-local system associated to $\calL^\vartheta_{\ell}$ by Theorem \ref{thm: Lafforgue local-global}. Then $\calL'_{\ell'}$ is $\ell'$-compatible for $\calL^\vartheta_{\ell}$ in the sense of \cite[\S5.3.1]{KZ}\footnote{In \cite[\S5.3.1]{KZ}, $\calL'_{\ell'}$ is denoted $\calK_{\ell'}$.}.

By Corollary \ref{cor: compatible system SV} (cf. Remark \ref{rem: compatible system}),
 $\mathcal{L}^\vartheta_{\ell'}$ is also $\ell'$-compatible for $\calL^\vartheta_{\ell}$.  Thus by the Chebotarev density theorem $\calL^\vartheta_{\ell'}$ and $\calL'_{\ell'}$  are isomorphic, since they are both semisimple.  Thus, 
using  Theorem \ref{thm: Lafforgue local-global}, we have 
$$ [\vartheta\circ\rho^{\WD}_{C,\ell',c}] = [\rho^{\WD}_{\calL_{\ell'}^\vartheta,c}] =  [\rho^{\WD}_{\calL'_{\ell'},c}] 
= [\rho^{\WD}_{\calL_{\ell}^\vartheta,c}] = [\vartheta\circ\rho^{\WD}_{C,\ell,c}] \in \Phi(q,\GL_n,\bbC).$$ 
\end{proof}
Combining with Proposition \ref{prop: WD rep reduction to GL_n}, we obtain the following corollary.
\begin{cor}\label{cor: l-indep boundary monodromy}
	With the notation and assumptions as above, for any $\ell,\ell'\neq p$, we have  $$[\rho^{\WD}_{C,\ell,c}]=[\rho^{\WD}_{C,\ell',c}]\in \Phi(q,\bfG,\bbC).$$
	Moreover $[\rho^{\WD}_{C,\ell,c}]$ is defined over $\bbQ$ for any $\ell$.
\end{cor}
\begin{proof} Let $\ell,\ell'\neq p$. By Proposition  \ref{prop: l-indep boundary monodromy GLn}, for any representation $\vartheta:\bfG_{\bbC}\rightarrow \GL_n$, we  have that $\vartheta\circ \rho_{C,\ell,c}^{\WD}$, $\vartheta\circ \rho_{C,\ell',c}^{\WD}$ are $\GL_n(\bbC)$-conjugate. Thus by Proposition \ref{prop: WD rep reduction to GL_n}, we have $$[\rho^{\WD}_{C,\ell,c}]=[\rho^{\WD}_{C,\ell',c}]\in \Phi(q,\bfG,\bbC).$$

 To show that $[\rho^{\WD}_{C,\ell,c}]$ is defined over $\bbQ$,  we first show it is defined over a number field $L$. By \cite[Proposition 1.13]{Imai}, we may view $\rho^{\WD}_{C,\ell,c}$ as a pair $(s',\xi)$ where $s'\in \bfG(\bbC)^{\Fss}$ and $$\xi:\SL_2\rightarrow Z_{\bfG}(s')$$ is a homomorphism to the centralizer of $s'$.
 
 Let $r:\bfG\rightarrow \GL_{2g}$ denote the representation arising from the Hodge embedding $\iota$. Then $r\circ \rho^{\WD}_{C,\ell,c}$ is the Weil--Deligne representation associated to an abelian variety with semistable reduction over $k_E'\ls$, hence is defined over $\bbQ$. In particular, the conjugacy class of $r(s')\in \GL_n(\bbC)^{\Fss}$ is defined over $\bbQ$. Since the map $$\Conj_{\bfG}\rightarrow \Conj_{\GL_n}$$ is a finite morphism of varieties, the conjugacy class of $s'$ is defined over $\bar\bbQ$. Thus upon replacing $\rho_{C,\ell,c}^{\WD}$ by a $\bfG(\bbC)$-conjugate, we may assume $s'\in \bfG(L)^{\Fss}$ for $L$ a number field.
 
 Now as in Lemma \ref{lem: conjugate SL2}, the homomorphism  $\xi:\SL_2\rightarrow Z_{\bfG}(s')$ is determined up to $Z_{\bfG}(s')$-conjugacy by the cocharacter $\alpha:\bbG_m\rightarrow Z_{\bfG}(s')$ given by the composition with the map $d:\bbG_m\rightarrow \SL_2$ from the diagonal torus. Upon extending $L$, the $Z_{\bfG}(s')$-conjugacy class of $\alpha$ has a representative defined over $L$. It follows that $\rho_{C,\ell,c}^{\WD}$ is defined over $L$.

We may assume $L$ is minimal for which $\rho_{C,\ell,c}^{\WD}$ is defined over $L$.   For every $\ell'\neq p$, $[\rho^{\WD}_{C,\ell,c}]=[\rho^{\WD}_{C,\ell',c}]$ arises from a $\bfG(\bbQ_{\ell'})$-local system, thus $[\rho^{\WD}_{C,\ell,c}]$ is defined over $\bbQ_{\ell'}$ for  all $\ell'\neq p$; here we consider $\bbQ_{\ell'}\subset \bbC$ via the isomorphism $i_{\ell'}:\bar\bbQ_{\ell'}\cong \bbC$. It follows that for every prime $\ell'\neq p$, there is a place of $L$ above $\ell$ which splits in $L$. The Chebotarev density theorem then implies $L=\bbQ$ (cf. \cite[Proof of Theorem 5.1.4]{KZ}).
	\end{proof}

\begin{definition}
	We define $$[\rho_{C,c}^{\WD}]\in \Phi(q,\bfG,\bbC)$$ to be the element $[\rho_{C,\ell,c}^{\WD}]\in\Phi(q,\bfG,\bbC)$ for one, equivalently any, choice of prime $\ell\neq p$.
\end{definition}
\subsection{Local monodromy for $\ell = p$}\label{subsec: l=p}

\subsubsection{} We keep the notation of the previous section. 
Fix a collection of tensors $\{s_{\alpha}\}_{\alpha} \subset V^{\otimes},$ such that $\bfG$ is the pointwise stabilizer of $\{s_{\alpha}\}_{\alpha}.$ 
The generic fiber $\Sh_{\rmK}=\scrS_{\rmK}\otimes E$  carries a $p$-adic local system $\calV_p[1/p],$ and each $s_{\alpha}$ 
give rise to a morphism $s_{\alpha,\et}:{\bf 1} \rightarrow \calV_p[1/p]^\otimes$ from the constant local system. 

Let $E''/E$ be a finite extension, and $x: \Spec E''\rightarrow \scrS_{\rmK}$ an  ${E''}$-point which we assume extends to a morphism $ x:\Spec \O_{E''} \rightarrow \scrS_{\rmK}^\Sigma.$ 
 Then $ x$ corresponds to an abelian variety over $E''$ with semi-stable reduction, and so the pullback $ x^*(\calV_p[1/p])$ 
is a semi-stable Galois representation. We may therefore apply Fontaine's functor to obtain $\calE_{ x} = D_{{\rm st}}( x^*(\calV_p)),$ 
which is a filtered $(\varphi,N)$-module. Forgetting the filtration, we may view $\calE_{ x}$ as a log $F$-isocrystal over 
the log point $(k_{E''}, \mathbb N^+).$ The tensors $s_{\alpha,\et},$ then give rise to tensors $s_{\alpha,0, x} \in \calE_{ x}^\otimes,$ 
which are invariant by $\varphi$ and killed by $N.$ Moreover, the $p$-adic comparison map implies that the scheme of isomorphisms 
$$\calP_{ x} = \SIsom_{\{s_{\alpha}\}}(V^{\vee},\calE_{ x})$$ 
taking $s_{\alpha}$ to $s_{\alpha,0, x}$ is a $\bfG$-torsor over $W(k_{E''})[1/p].$
Thus $\calE_{ x}$ may be promoted to an $F$-isocrystal with $\bfG$-structure 
$$ \calE_{ x}^{\bfG}: \Rep_{\bar \Q_p} \bfG \rightarrow \Isoc^{\dag,\varphi}_{(k_{E''},\mathbb N^+)/k_{E''},\bar \Q_p}; \quad 
W \mapsto W\times \calP_{ x}/\bfG.
$$

\subsubsection{} We want to show that all the $\calE_{x}^{\bfG}$ 
arise from a log $F$-isocrystal with $\bfG$-structure over $\calS_{\rmK}^\Sigma.$ 
For technical reasons we do this only after pulling back to a curve. 

We now assume that the map $\pi: \bar C \rightarrow \calS_{\rmK,k_{E}'}^\Sigma,$ defined in \ref{subsubsec:curvedefn}, 
lifts to a map $\tilde \pi:\bar \calC\rightarrow \scrS_{\rmK,\O_{E'}}^\Sigma,$ where $\O_{E'}$ is the ring of integers of 
a finite extension $E'/E$ with residue field $k_{E'} = k_{E}',$ and $\bar \calC$ is a smooth curve over $\O_{E'}$ 
with $\bar\calC_{k_{E'}} = \bar C.$ 
We equip $\bar \calC$ with the log structure coming from the complement $\scrS_{\rmK}^\Sigma\backslash \scrS_{\rmK},$ and 
we let $\calC =  \pi^{-1}(\scrS_{\rmK, \O_{E'}}).$

\begin{lemma}\label{lem:GIsoc} There is an object $\calE^{\bfG}_C$ in $\bfG$-$\Isoc^{\dag,\varphi}_{\bar C/k_{E}',\bar \Q_p}$ such 
that for any finite extension $E''/E',$ and any point $ x \in \bar\calC(\O_{E''}) \cap \calC(E''),$ there is a canonical isomorphism 
$ x^*(\calE^{\bfG}_C) \simeq \calE^{\bfG}_{ x},$ in $\Isoc^{\dag,\varphi}_{(k_{E''},\mathbb N^+)/k_{E''},\bar \Q_p}.$
\end{lemma}
\begin{proof} By \cite[1.3.5]{Keerthi}, the universal abelian scheme (arising from the chosen Hodge embedding), 
on $\scrS_{\rmK}$ corresponds to a log Dieudonn\'e crystal on $\calS_{\rmK,k_{E}'}^\Sigma.$ 
In particular, this gives rise to an object $\calE$ in $\Isoc^{\dag,\varphi}_{\bar C/k_{E}'},$ such that for all $ x$ as above 
$ x^*(\calE)$ is canonically isomorphic to $\calE_{ x}.$

We claim that there exists morphisms $s_{\alpha,0}: \mathbf 1 \rightarrow \calE^\otimes,$ in 
$\Isoc^{c,\varphi}_{\bar C/k_{E}'},$ such that, for all $ x$ as above, $ x^*(s_{\alpha,0}) \simeq s_{\alpha,0, x}.$ 
The proof of this is essentially contained in \cite[Prop.~A.6]{KMS}. There are two differences between the statement of 
{\em loc.~cit}, and the one we are using here: The base scheme $S$ there (which corresponds to $\bar \calC$) is assumed to be proper, 
and the compatibility $ x^*(s_{\alpha,0}) \simeq s_{\alpha,0,x}$ is only stated for $ x \in \calC(\O_{E''}).$ 
The properness is not used in the proof; indeed the argument is made \'etale locally. The compatibility for more general $ x$ 
is proved in the same way, by invoking the functoriality of Faltings' construction \cite{FaltingsAE}.

Next we show that $\{s_{\alpha,0}\}_{\alpha}$ gives rise to a {\em convergent} object $\calE^\bfG_C,$ in 
$\bfG$-$\Isoc^{c,\varphi}_{C/k_E',\bar \Q_p}.$ 
For an enlargement $C \hookrightarrow T,$ let $\calP_T = \SIsom_{\{s_{\alpha}\}}(V^{\vee},\calE(T))$ be the 
scheme of isomorphisms taking $s_{\alpha}$ to $s_{\alpha,0}|_T.$ This is a finite type 
$T$-scheme.  We claim that that $\calP_T$ is a $\bfG$-torsor. Let $\widehat{\calC}$ denote the $p$-adic completion of $\calC.$ 
Locally on $T,$ $C \rightarrow \widehat \calC $  lifts to $T \rightarrow \widehat{\calC},$ so it suffices to check this when 
$T = \widehat \calC.$   

Now $\calP_{\widehat\calC}$ is an affine scheme of finite type over the adic space $\widehat \calC \otimes \Q_p,$ such that for any finite extension $E''/E$ and $ x \in \widehat \calC(E''),$ $ x^*(\calP_{\widehat{\calC}}) \simeq \calP_{ x}$ is a $\bfG$-torsor. 
Since the sections 
$s_{\alpha,0}(\widehat\calC)$ are parallel for the connection on $\calE(\widehat{\calC}),$ 
$\calP_{\widehat{\calC}}$ is flat over  $\widehat \calC \otimes \Q_p,$ with constant fiber dimension. Thus 
$\calP_{\widehat{\calC}}$ is a $\bfG$-torsor, as claimed.

We consider $\calP_T$ as a functor on the category of enlargements, and denote by $\varphi^*(\calP)_T$ the value 
of its Frobenius pullback on $T.$ Since the $s_{\alpha,0}$ are Frobenius invariant, we have an natural isomorphism 
$\varphi^*(\calP) \simeq \calP.$ For $W \in \Rep_{\bar \Q_p} \bf G,$ we set $\calE^{\bfG}_C(W) = W\times\calP/\bfG,$ which defines an object in 
$\bfG$-$\Isoc^{c,\varphi}_{C/k_E',\bar \Q_p}.$

To show that $\calE^{\bfG}_C$ corresponds to an object in $\bfG$-$\Isoc^{\varphi, \dag}_{\bar C/k_E',\bar \Q_p},$ 
note that the functors 
$$\Isoc^{\dag,\varphi}_{\bar C/k_E',\bar \Q_p} \rightarrow \Isoc^{\dag,\varphi}_{C/k_E',\bar \Q_p} \rightarrow \Isoc^{c,\varphi}_{C/k_E',\bar \Q_p}   $$ 
are fully faithful by \cite[Thm.~1.1]{Kedlayafaithful} and \cite[Thm.~6.4.5]{Kedlayasemistab}, so it suffices to show that for 
$W \in \Rep_{\bar \Q_p} \bf G,$ $\calE^{\bfG}_C(W)$ is in 
$\Isoc^{\dag,\varphi}_{\bar C/k_E',\bar \Q_p}.$ 
But $\calE^{\bfG}(W)$ is a direct summand in $\calE^\otimes,$ so this follows from the fact that $\calE$ is in 
$\Isoc^{\dag,\varphi}_{\bar C/k_E',\bar \Q_p}.$

It remains to check that for $ x \in \bar\calC(\O_{E''}) \cap \calC(E''),$ there is a canonical isomorphism 
$ x^*(\calE^{\bfG}_C) \simeq \calE^{\bfG}_{ x}.$ First note that $\calE^{\bfG}_C(s_{\alpha}) = s_{\alpha,0}.$ 
Indeed, this holds over $C$ by construction, and hence over $\bar C.$ 
This implies that 
$$ \calP_{ x} = \SIsom_{\{s_{\alpha}\}}(V^{\vee},\calE_{ x}) = \SIsom_{\{s_{\alpha}\}}(V^{\vee},{ x}^*(\calE^{\bfG}_C(V))),$$
so both $\calE_{ x}$ and $ x^*(\calE^{\bfG}_C)$ can be recovered by twisting $V^{\vee}$ by the same $\bfG$-torsor, and hence they are canonically isomorphic.
\end{proof}

\subsubsection{} Recall that associated to $\calE^{\bfG}_C,$ and $c \in \bar C(k_E'),$ we have the Weil-Deligne representation 
$\rho_{\calE^{\bfG}_C,c}^{\log}$ constructed in \ref{sssec:FisocrystalsG}

\begin{cor}\label{cor: l-indep boundary monodromyII} 
$\rho_{\calE^{\bfG}_C,c}^{\log}$ is (URFS), and we have $$[\rho_{\calE^{\bfG}_C,c}^{\log}] = [\rho^{\WD}_{C,c}]\in\Phi(q,\bfG,\bbC).$$
\end{cor}
\begin{proof} The fact that $\rho_{\calE^{\bfG}_C,c}^{\log}$ is (URFS) follows from the same argument as Lemma \ref{lem: WD rep is URFS}, using the corresponding result for the Weil--Deligne representation associated to an abelian variety with semistable reduction over a local function field; see eg. \cite[Lemm 3.12]{CL}
	
	Let $\ell' \neq p$ be a prime. By Proposition \ref{prop: WD rep reduction to GL_n}, it suffices to show that 
for any representation $\vartheta: \bfG_{\bbC} \rightarrow \GL_n,$ 
$\vartheta\circ \rho_{C,\ell',c}^{\WD}$ and $\vartheta\circ \rho_{\calE^{\bfG}_C,c}^{\log}$ are $\GL_n(\bbC)$-conjugate.
The proof of this is similar to that of Proposition \ref{prop: l-indep boundary monodromy GLn}.

Let  $\calL'_{\ell'}$ denote the semisimple $\bar{\bbQ}_{\ell'}$-local system associated to $\calE^{\bfG}_C(\vartheta)$ by Theorem \ref{thm: Lafforgue local-global}. Then $\calL'_{\ell'}$ is $\ell'$-compatible for $\calE^{\bfG}_C(\vartheta).$  
By Corollary \ref{cor: compatible system SV}, cf. Remark \ref{rem: compatible system},
$\mathcal{L}_{\ell'}^\vartheta$ is also $\ell'$-compatible for $\calE^{\bfG}_C(\vartheta)$.  Thus by the Chebotarev density theorem 
$\calL^\vartheta_{\ell'}$ and $\calL'_{\ell'}$  are isomorphic, since they are both semisimple.  It follows by Theorem 
\ref{thm: Lafforgue local-global} that 
$$[\vartheta\circ \rho_{\calE^{\bfG}_C,c}^{\dag}] = [\rho^{\WD}_{\calL'_{\ell'},c}] = [\rho^{\WD}_{\calL^\vartheta_{\ell'},c}] 
= [\vartheta\circ \rho_{C,\ell',c}^{\WD}] \in \Phi(q,\GL_n,\bbC). $$
 Now the corollary follows as 
$\rho_{\calE^{\bfG}_C,c}^{\log} \simeq \rho_{\calE^{\bfG}_C,c}^{\dag}$ by Lemma \ref{lem:WDGreps}.
\end{proof}

\subsection{WD-representations attached to abelian varieties}
We now prove our main result concerning abelian varieties.

\subsubsection{}\label{sssec: MT groups} Let $A$ be an abelian variety over a number field $\rmE$. Recall we have fixed an embedding $i_\infty:\bar{\bbQ}\rightarrow\bbC$; using this we may consider $\rmE$ as a subfield of $\bbC$.  We write $V_B$ for the Betti cohomology $\rmH^1_B(A(\C),\Q)$ which is equipped with a Hodge structure of type $((0,-1),(-1,0))$. This Hodge structure is induced by a morphism 
$$h:\mathbb{S}:=\text{Res}_{\C/\R}\mathbb{G}_m\rightarrow \GL(V_B).$$
We write 
$$\mu:\C^\times\xrightarrow{z\mapsto (z,1)}\C^\times\times c^*(\C^\times)\xrightarrow{h} {\GL}(V_B\otimes\C)$$ for the Hodge cocharacter. 

\begin{definition} 
	The Mumford--Tate group $\bfG$ of $A$ is  the smallest algebraic subgroup of $\GL(V_B),$ defined over $\Q,$ such that  $\bfG(\bbC)$ contains the image of $\mu$. \end{definition}

We remark that $\bfG$ depends  on the embedding $\rmE\hookrightarrow \bbC$; if $\bfG_1$ is the group defined by a different embedding then there is a canonical inner twisting $\bfG_{\bar{\bbQ}}\cong\bfG_{1,\bar{\bbQ}}$ induced by the  torsor of tensor preserving isomorphisms between the Betti cohomology groups (see \cite[Proof of Theorem 3.8]{De1}. In particular, there is a canonical $\mathrm{Aut}(\bbC/\bbQ)$-equivariant identification 
$\Phi(q,\bfG,\bbC)\cong \Phi(q,\bfG_1,\bbC)$ for any $q=p^s$.
\subsubsection{}\label{subsubsec:notation} For a prime number $\ell$, let $T_\ell A$ be the $\ell$-adic Tate module of $A$. 
The action of 
$\Gamma_{\rmE}:=\mathrm{Gal}(\bar{\rmE}/\rmE)$ on $T_\ell A$ gives rise to a representation 
$\rho_{A,\ell}:\Gamma_{\rmE}\rightarrow \GL(T_\ell A)$ and the Betti-\'etale comparison gives us a canonical isomorphism 
$$\rmH^1_B(A(\C),\Q)\otimes_{\bbQ}\Q_\ell\cong T_\ell A^\vee\otimes_{\bbZ_\ell}\Q_\ell.$$

Deligne's theorem that Hodge cycles are absolutely Hodge \cite{De1}, implies that upon replacing $\rmE$ by a finite extension, the map 
$\rho_{A,\ell}$ factors through $\bfG(\Q_\ell)$; see \cite[Remarque 1.9]{Noot}. By \cite[Lemma 7.1.4]{KZ}, if $\rho_{A,\ell}$ factors through $\bfG(\Q_\ell)$ for some prime $\ell,$ then this holds for all primes $\ell.$
We replace $\rmE$ by the smallest extension such that $\Gamma_{\rmE}$ maps to $\bfG(\Q_\ell)$, and
we write $\rho_{A,\ell}^{\bfG}$ for the induced map $\Gamma_{\rmE}\rightarrow \bfG(\Q_\ell).$

\subsubsection{}\label{subsubsec:repsetup}
Let $v$ be a prime of $\rmE$ where $A$ has semistable reduction lying above a rational prime $p$. Upon modifying the embedding $i_p:\bar{\bbQ}\rightarrow \bar{\bbQ}_p$ fixed in \S\ref{sssec: integral models preamble}, we may assume that $v$ is induced by $i_p$. We write $E' = \rmE_v$, and let $k_{E'}=\bbF_q$ be its residue field.
For $\ell\neq p$ a prime,  the restriction of $\rho_{A,\ell}^{\bfG}$ to $\Gamma_{E'}:=\mathrm{Gal}(\bar{E}'/E')$ gives rise to a Weil--Deligne $\bfG$-representation $\rho^{\WD}_{A,\ell,v}$ which is known to be Frobenius semisimple (cf. \cite{Noot2} Remark 1.9), and hence (URFS) by our assumption of semistable reduction at $v$.  If $v$ is a place of good reduction for $A$, $\rho_{A,\ell}^{\bfG}$ is unramified at $v$ and  $\rho^{\WD}_{A,\ell,v}$ factors through the surjection
$\bbG_a\rtimes W_{E'}\rightarrow W_{E'}\xrightarrow{\alpha}\bbZ$. 

For $\ell = p,$ we have an $F$-isocrystal with $\bfG$-structure 
$$\calE^{\bfG}: \Rep_{\bar \Q_p} {\bf G} \rightarrow \Isoc^{\varphi,\dag}_{(k_{E'},\mathbb N^+)/k_{E'}};\quad W \mapsto D_{\mathrm{st}}(W),$$ 
where we view $W \in \Rep_{\bar \Q_p} \bf G$ as a $\Gamma_{E'}$-representation via $\rho_{A,p}^{\bfG}.$ 
We denote by $\rho^{\WD}_{A,p,v}$ the Weil--Deligne associated to $\calE^{\bfG},$ as in 
\ref{sssec:assphiNmodule} and \ref{sssec:WDrepsisoc}. Again, $\rho^{\WD}_{A,p,v}$ is (URFS), and unramified if $A$ has good reduction  reduction at $v$.

\subsubsection{}\label{subsubsec: restriction of scalars} We would like to apply the considerations in \S\ref{subsec: local monodromy} and \S\ref{subsec: l=p} to the current setting. To do this, we will make use of the following auxiliary construction. 
Let $\rmF/\bbQ$ be a totally real field, and let $\bfH\subset \bfH':=\text{Res}_{\rmF/\Q}\bfG_{\mathrm{F}}$ denote the subgroup constructed in \cite[\S6.1.6]{KZ}. Then $\bfH$ is a reductive group and there is a $\bfH(\bbR)$ conjugacy class  $X_{\bfH}$ of homomorphisms $\bbS\rightarrow \bfH_{\bbR}$ which makes $(\bfH,X_{\bfH})$ a Shimura datum. Moreover, there are morphisms of Shimura data 
$$ (\bfG, X) \hookrightarrow (\bfH, X_{\bfH}) \hookrightarrow (\mathbf{GSp}(W), S'^{\pm}), $$
where $W$ is the symplectic space with underlying vector space given by $V\otimes_{\bbQ}F$ considered as a $\bbQ$-vector space.

\begin{lemma}\label{lemma: WD injective}
The natural inclusion $\bfG\rightarrow \bfH$ induces an injective map $$\Phi(q,\bfG,\bbC)
\rightarrow \Phi(q,\bfH,\bbC)$$ which is equivariant for the action of $\mathrm{Aut}(\bbC/\bbQ)$.
		\end{lemma}
\begin{proof}Let $(s_1,N_2),(s_2,N_2)\in \Phi^\square(q,\bfG,\bbC)$ and suppose  there exists  $h\in \bfH({\bbC})$ such that $(s_2,N_2)=(hs_1h^{-1},\mathrm{Ad}(h)(N_1))$ in $\Phi^{\square}(q,\bfH,\bbC)$.  Then under the identification $$\bfH'_\bbC=\prod_{\tau:\rmF\rightarrow \bbC}\bfG_\bbC,$$
$(s_i,N_i)$ corresponds to $(s_i,\dotsc,s_i)\times (N_i,\dotsc,N_i)\in \prod_{\tau:\rmF\rightarrow \bbC}\bfG_{\bbC}\times \mathcal{N}$, and we let $h=(h_1,\dotsc,h_n)$. Then $(s_2,N_2)=(hs_1h^{-1},\mathrm{Ad}(h_1)(N_1))$ implies $$(s_2,N_2)=(h_1s_1h_1^{-1},\mathrm{Ad}(h)(N_1))$$ with $h_1\in \bfG(\bbC)$,   and so $(s_1,N_1)$ and $(s_2,N_2)$ have the same image in $\Phi(q,\bfG,\bbC)$. 

The equivariance under $\mathrm{Aut}(\bbC/\bbQ)$ follows from the fact that the map $\bfG\rightarrow \bfH$ is defined over $\bbQ$.
\end{proof}

\subsubsection{} Let $\widetilde{\sigma}_q\in \Gamma_{\rmE}$ be the image of a lift of the geometric Frobenius in $\mathrm{Gal}(\bar{E}'/E')$.  
The following proposition serves as the analogue of \cite[Proposition 6.2.4]{KZ} in our setting, and is proved in a similar way. 
\begin{prop}\label{prop: existence strongly adm}
	There exists a totally real field $\rmF$ such that if $(\bfH,X)$ is the Shimura datum arising from the construction in \S\ref{subsubsec: restriction of scalars}, then $H:=\bfH_{\bbQ_p}$ is quasi-split and there exists a parahoric group scheme $\calH$ for $H$ such that 	\begin{enumerate}
		\item[(A)] The image of $\rho_{A,p}^{\bfG}(\widetilde{\sigma}_q)$ in $H(\bbQ_p)$ lies in $\calH(\bbZ_p)$.
		\item[(B)] The triple $(\bfH,X_{\bfH},\calH)$ is strongly admissible.
	\end{enumerate}
	
\end{prop}
\begin{proof}By \cite[Lemma 6.2.1]{KZ}, there exists $F/\bbQ_p$ finite such that $G_{F}$ is split and there exists a special parahoric $\calG_F$ of $G_F$ such that the image of $\rho_p^{\bfG}(\tilde{\sigma}_q)$ lies in $\calG_F(\calO_F)$. We take $\rmF$ a totally real field with $\rmF_w=F$ for all places $w|p$ of $\rmF$, and let $\calH$ be the parahoric of $H$ corresponding to  the parahoric $\prod_{w|p}\calG_{\calO_F}(\calO_F)$ of $H'$. Then $H^{\der}$ is the Weil-restriction of a split group, hence quasi-split, and has reduced relative root system so that condition (4) in the definition of strongly admissible triple is satisfied. (1) is satisfed by construction, and (3) is satisfied since any centralizer of a maximal $\brQ$-split torus is an extension of $\bbG_m$ by an induced torus, hence $R$-smooth by \cite[Proposition 2.4.6]{KZ}.
	
	For condition (2), note that $X_*(H^{\ab})$ is an extension of  $\bbG_m$ by an induced torus, and hence $X_*(H^{\ab})_I$ is torsion-free. The result then follows from \cite[Lemma 4.2.4]{KZ}, noting that $\calH$ is a very special parahoric of $H$.
\end{proof}
\subsubsection{} We now prove our main theorem.  
\begin{thm}\label{thm: main AV}
Let $A$ be an abelian variety over a number field $\rmE$ and let $v|p$ be a prime of $\rmE$ where $A$ has semistable reduction. Then there exists an element  $[\rho^{\WD}_{A,v}]\in \Phi(q,\bfG,\bbC)$ satisfying the following two properties. \begin{enumerate}
	  \item $[\rho^{\WD}_{A,v}]$ is defined over $\bbQ$.
	\item For all primes $\ell$, we have 
  $$[\rho_{A,v}^{\WD}]=[\rho_{A,\ell,v}^{\WD}]\in \Phi(q,\bfG,\bbC).$$
\end{enumerate}
\end{thm}

\begin{proof}  If $\bfG$ is a torus, $A$ has CM and hence has everywhere potentially good reduction. The assumption of semistability then implies $A$ has good reduction at $v$, and in this case the result is a theorem of Shimura--Taniyama. 
	
	We now assume $\bfG$ is not a torus. 	Let $(\bfH, X_\bfH,\calH)$ be a strongly admissible triple  arising from Proposition \ref{prop: existence strongly adm} above.  
 By construction, $\rho_{A,p}^{\bfG}(\widetilde\sigma_q)$ lies in $\rmK_p:=\calH(\bbZ_p)$. Hence there is a finite extension $\rmE'/\rmE$  such that $\rho_{A,p}^{\bfG}|_{\Gamma_{\rmE'}}$ factors through $\rmK_p$, and such that there is a prime $v'|v$ such that $\rmE'_{v'}$ is a totally ramified extension of $\rmE_v$. By Lemma \ref{lem: tot ram extension OK} below, it suffices to prove the result with $\rmE$ replaced by $\rmE'$, and $v$ replaced by $v'$, thus upon replacing $\rmE$ by $\rmE'$, we may assume $\rho_{A,p}^{\bfG}$ factors through $\rmK_p$.

	By our assumption on $\rmE$, the representation $\rho^p:\Gamma_{\rmE}\rightarrow \GL(\widehat{V}(A))$ factors through $\bfG(\bbA_f^p)\subset \bfH(\bbA_f^p)$ and hence through a compact open subgroup $\rmK^p\subset \rmH(\bbA_f^p)$. We set $\rmK=\rmK_p\rmK^p$ and write $\scrS_{\rmK}:=\scrS_{\rmK}(\bfH,X_{\bfH})$ for the integral model. As in \cite[Theorem 6.2.7]{KZ}, the abelian variety $A^\rmF:=A\otimes_{\bbQ}\rmF$ given by the Serre tensor construction \cite[\S7]{Conrad} corresponds to an $E':=\rmE_v$-point $x_A\in \Sh_{\rmK}(E')$. 
	
\emph{Case (1): $A$ has good reduction at $v$.} In this case $[\rho^{\WD}_{A,\ell,v}]$ is unramified and hence is determined by the conjugacy class of Frobenius $\gamma_{A,\ell,v}\in \Conj_{\bfG}(\bbC)$. The statement then amounts  to showing there exists $\gamma_{A,v}\in \Conj_{\bfG}(\bbQ)$ such that $$\gamma_{A,v}=\gamma_{A,\ell,v}\text{ for all primes } \ell.$$
When $p>2$, and $\ell\neq p$, this is proved in \cite[Theorem 6.2.7]{KZ}. The remaining cases in the good reduction case ($p=2$ and $\ell=p$) can be handled in the same way using Corollary \ref{cor: compatible system SV} in place of \cite[\S5.1.4]{KZ}. We recall the argument here for completeness. 

By construction, the triple $(\bfH,X_{\bfH},\calH)$ satisfies the assumptions in Corollary \ref{cor: compatible system SV}. The assumption that $A$ has good reduction implies that $x_A$ extends to a point $x\in \scrS_{\rmK}(\calO_{E'})$. Let $x_0\in \scrS_{\rmK}(k_{E'})$ denote the special fiber of $x$.  Then we have $\gamma_{x_0,\ell}=\gamma_{A,\ell,v}\in \Conj_{\bfH}(\bbC)$, where $\gamma_{x_0,\ell}$ is the element associated to $x_0$ in \S\ref{sssec: compatible system SV}.
Then by Corollary \ref{cor: compatible system SV}, there exists $\gamma\in \Conj_{\bfH}(\bbQ)$ such that $\gamma=\gamma_{x_0,\ell}\in \Conj_{\bfH}(\bbC)$ for all $\ell$. The result then follows from Lemma \ref{lemma: WD injective}.

\emph{Case (2): $A$ has bad reduction at $v$.} Let $\Sigma$ be a complete smooth admissible rpcd for the triple $(\bfH,X_{\bfH},\rmK)$, and $\scrS_{\rmK}^\Sigma$ the associated compactification. The assumption of semistability implies that $x_A$ extends to  point $x\in \scrS_{\rmK}^\Sigma(\calO_{E'})$, with special fiber $x_0$ contained in the boundary $\partial \scrS_{\rmK}^\Sigma(k_{E'})$. 
	
Fix $\varpi$ a uniformizer for $E'$.  Let $(\bar\calC,c)$ be a residually trivial \'etale neighbourhood of $\Spec\calO_{E'}[u]$ 
and $\tilde\pi: \bar\calC \rightarrow  \scrS_{\rmK}^\Sigma$ a map satisfying the properties in the conclusion of 
Theorem \ref{thm: boundary curves}; in particular, there is a point $\delta\in \bar{\calC}(\calO_{E'})$ lying above $u\mapsto \varpi$ which maps to $x_A$. 
By assumption, the pre-image of the interior $\scrS_{\rmK}$ in $\bar\calC$ contains $\bar\calC[u^{-1}]$. 
Thus for each $\ell\neq p$, we obtain a $\bfH(\bbQ_\ell)$-local system  on $\bar\calC[u^{-1}]$  given by the pullback of $\widetilde{\mathbb L}_\ell$. This corresponds to a continuous homomorphism $$\rho_\ell:\pi_1(\bar\calC[u^{-1}],\overline{y})\rightarrow \bfH(\bbQ_\ell).$$
We let $\gamma$ denote the morphism $\Spec \bbF_q\ps\rightarrow \bar{\calC}$ lifting $\calO_{E'}[u]\rightarrow \bbF_q\ps,u\mapsto u$.
We write $\rho_{E',\ell}:\Gamma_{E'}\rightarrow \bfH(\bbQ_\ell)$ and $\rho_{\bbF_q \ls,\ell}:\Gamma_{\bbF_q\ls}\rightarrow \bfH(\bbQ_\ell)$ for the representations obtained from $\rho_\ell$ by pullback along $\delta$ and $\gamma$ respectively. 

Note that we have an equality $\xi\circ\rho^{\WD}_{A,\ell,v}=\rho^{\WD}_{E',\ell}$, where $\xi:\bfG\rightarrow \bfH$ is the natural map. It follows that  $\rho^{\WD}_{E',\ell}$ is (URFS).  By Corollary \ref{cor: comparison mixed/equal WD reps}, we have 
$$[\xi\circ\rho^{\WD}_{A,\ell,v}]=[\rho^{\WD}_{E',\ell}]=[\rho^{\WD}_{\bbF_q\ls,\ell}]\in \Phi(q,\bfH,\bbC).$$ 
Here we use Remark \ref{rem: different local fields} to identify the sets $\Phi(q,\bfH,\bbC)$ for the different local fields $E'$ and $\bbF_q\ls$.
Let $\bar C = \bar\calC\otimes k_{E'},$ and $C \subset \bar C$ the preimage of $\scrS_{\rmK}.$ Then $\rho^{\WD}_{\bbF_q\ls,\ell}=\rho^{\WD}_{C,\ell,c}$ with the notation in \S\ref{subsubsec:curvedefn}, and hence
applying Corollary  \ref{cor: l-indep boundary monodromy} to $\bar C$, we find that  there exists $[\rho^{\WD}_{A,\bfH,v}]\in \Phi(q,\bfH,\bbC)$ defined over $\bbQ$ such that  
$$[\rho^{\WD}_{A,\bfH,v}]=[\rho^{\WD}_{\bbF_q\ls,\ell}]=[\xi\circ\rho^{\WD}_{A,\ell,v}]\in \Phi(q,\bfH,\bbC)$$ for all $\ell\neq p.$  
It follows by Lemma \ref{lemma: WD injective}, there is an element $[\rho^{\WD}_{A,v}]\in \Phi(q,\bfG,\bbC)$ defined over $\bbQ$ such that 
 $[\rho^{\WD}_{A,\bfH,v}]=[\xi\circ\rho^{\WD}_{A,v}],$ and 
$$[\rho^{\WD}_{A,v}]=[\rho^{\WD}_{A,\ell,v}]\in \Phi(q,\bfG,\bbC)$$ for all $\ell\neq p.$ 
This proves the theorem for $\ell \neq p.$

For the case $\ell = p,$
let $\calE^{\bfH}_C$ be the object of $\bfH$-$\Isoc^{\dag,\varphi}_{\bar C/k_E,\bar \Q_p}$ constructed in Lemma 
\ref{lem:GIsoc}. 
Then we have 
$$ [\xi\circ\rho^{\WD}_{A,p,v}] = [\rho^{\log}_{\calE^{\bfH}_C}] = [\xi\circ\rho^{\WD}_{A,v}] \in \Phi(q,\bfH,\bbC),$$ 
where the first equality is part of  Lemma \ref{lem:GIsoc} - note that it involves the equality in $\Phi(q,\bfH,\bbC)$ of representations 
of $\WD_{E'}$ and $\WD_{\bbF_q\ls}$ - and the second equality follows from Corollary \ref{cor: l-indep boundary monodromyII}.  Arguing  as above using Lemma \ref{lemma: WD injective}, we find that $$ [\rho^{\WD}_{A,p,v}] =[\rho^{\WD}_{A,v}] \in \Phi(q,\bfG,\bbC)$$ as desired.
\end{proof}
	
\begin{lemma} 	\label{lem: tot ram extension OK}
	Let $F$ be a non-archimedean local field with residue field $\bbF_q$ and let $\rho:\Gamma_F^t\rightarrow \bfG(\bbQ_\ell)$ be a unipotently ramified continuous homomorphism. Let $F'/F$ be a totally ramified extension, and we let $\rho^{\WD}$ (resp. $\rho'^{\WD}$)  denote the unipotently ramified Weil--Deligne representations associated to $\rho$ (resp. $\rho|_{\Gamma_{F'}}$). Then we have an equality $$[\rho^{\WD}]=[\rho'^{\WD}]\in \Phi(q,\bfG,\bbC).$$
\end{lemma}
\begin{remark}
	Note that $F$ and $F'$ have the same residue field, so that it makes sense to compare $[\rho^{\WD}]$ and $[\rho'^{\WD}]$ as elements of $\Phi(q,\bfG,\bbC)$.
\end{remark}

\begin{proof} Note that $\rho^{\WD}$ (resp. $\rho'^{\WD}$) depends on a choice of Frobenius lift $\sigma$ in $\Gamma_F^t$ (resp. $\sigma'\in \Gamma_{F'}$) and a choice of homomorphism $t_\ell:I^t_F\rightarrow \bbZ_\ell$ (resp. $t_{\ell}':I_{F'}^t\rightarrow \bbZ_\ell$). Since $F'/F$ is totally ramified, we may choose $\sigma=\sigma'$,  and we let $t_\ell'=m^{-1}t_\ell|_{I_{F'}^t}$, where $m$ is the index of $I_{F'}^t$ in $I^t_F$.
	
We let $(s,N), (s',N')\in \Phi^\square(q,\bfG,\bbC)$ denote the elements corresponding to $\rho^{\WD}$ and $\rho'^{\WD}$ respectively. Then with our above choices, we have that $$(s',N')=(s,aN)$$ for some $a\in \bbC$. The argument in \cite[Variante 8.11]{De4} shows that  $(s',N')$ and $(s,N)$ are $\bfG(\bbC)$-conjugate, and hence their images in $\Phi(q,\bfG,\bbC)$ are equal.
\end{proof}
\begin{remark}
The proof of 	Theorem \ref{thm: main AV} applies to abelian varieties defined over more general fields: Let $K$ be a finite extension of $\bbQ_p$ with residue field $\bbF_q$ equipped with an embedding $K\subset \bbC$ and let $A$ be an abelian variety over $K$ with Mumford--Tate group $ \bfG$.  Then the proof of Theorem \ref{thm: main  AV} shows that upon replacing $K$  by a finite extension,  we obtain Weil--Deligne $\bfG$-representations $\rho^{\WD}_{A,\ell,v}$ for each prime $\ell$ whose image in $\Phi(q,\bfG,\bbC)$ is defined over $\bbQ$ and does not depend on the choice of $\ell$. \end{remark}
\subsection{Applications}\label{sssec: applications}
	In this last section, we use our main result on abelian varieties to deduce 
 $\ell$-independence results for Shimura varieties without the assumption of strong admissibility. 
 
 We now let $(\bfG,X)$ be a Shimura datum of Hodge type and let  $\scrS_{\rmK}$ be an integral model for $\Sh_{\rmK}(\bfG,X)$ over $\calO_E$ as constructed in \S\ref{sssec: construction integral models}. In particular, we don't necessarily assume  that $(\bfG,X)$ arises as part of a strongly admissible triple. For $k'_E=\bbF_q$ a finite extension of $k_E$ and $x_0\in \scrS_{\rmK}(k'_E),$ we obtain elements $\gamma_{x_0,\ell}\in\Conj(\bbQ_\ell)$ for $\ell\neq p$ and  $\gamma_{x_0,p}\in \Conj_{\bfG}(\brQ)$  corresponding to the local Frobenius acting on the stalk of the $\bfG(\bbQ_\ell)$-local system $\bbL_{\ell}$ (resp. $F$-isocrystal with $\bfG$-structure $\calE$); see \S\ref{sssec: l-indep SV interior}. For $m\geq 1$, we let $\gamma_{x_0,\ell}^{(m)}\in \Conj_{\bfG}(\bbQ_\ell)$ for $\ell\neq p$ (resp. $\gamma_{x_0,p}^{(m)}\in \Conj_{\bfG}(\brQ)$) denote the corresponding elements for the $q^{m}$-Frobenius.

\begin{thm}\label{thm: application}
	Let $k_E'=\bbF_q$ be a finite extension of the residue field $k_E$ and let $x_0\in \scrS_{\rmK}(k_E')$. For sufficiently divisible $m$, there exists an element $\gamma_{x_0}^{(m)}\in \Conj_{\bfG}(\bbQ)$ such that $$\gamma_{x_0}^{(m)}=\gamma^{(m)}_{x_0,\ell} \text{ for all $\ell$ (including $\ell=p$)}.$$
\end{thm}
\begin{proof}Let $x\in \scrS_{\rmK}(\bfG,X)(K)$ be a lift of $x_0$, where  $K$ is a finite extension of $\bbQ_p$ with residue field $\bbF_{q^s}$. Then we obtain an abelian variety $\calA_{x}$ over $K$, and for a fixed embedding $K\subset \bbC$, its Mumford--Tate group is a subgroup $\bfG'$ of $\bfG$. Upon replacing $K$ by a finite extension, we may assume the action of $\Gamma_K$ on $\calV_{\ell}\calA_{x}$ factors through $\bfG'(\bbQ_\ell)$. It follows that the elements $\gamma_{x_0,\ell}^{(s)}$ can be refined to an element (also denoted $\gamma_{x_0,\ell}^{(s)}$) of $\Conj_{\bfG'}(\bbQ_\ell)$ (resp. $\Conj_{\bfG'}(\brQ)$ for $\ell=p$). By Theorem \ref{thm: main AV}, we obtain an element $\gamma_{x_0}^{(s)}\in \Conj_{\bfG'}(\bbQ)$ such that $\gamma_{x_0}^{(s)}=\gamma_{x_0,\ell}^{(s)}$ for all $\ell$. The theorem follows a fortiori for any $m$ which is divisible by $s$.
	\end{proof}

\begin{remark}
	This theorem verifies  \cite[Hypothesis 2.3.1]{vHoftenOrdinary} for these Shimura varieties. In \emph{loc. cit.}, this is used to prove instances of the Hecke-orbit conjecture.
\end{remark}
\bibliographystyle{amsalpha}
\bibliography{bibfile}

\providecommand{\bysame}{\leavevmode\hbox to3em{\hrulefill}\thinspace}
\providecommand{\MR}{\relax\ifhmode\unskip\space\fi MR }
% \MRhref is called by the amsart/book/proc definition of \MR.
\providecommand{\MRhref}[2]{%
  \href{http://www.ams.org/mathscinet-getitem?mr=#1}{#2}
}
\providecommand{\href}[2]{#2}
\begin{thebibliography}{AMRT10}

\bibitem[Abe18a]{Abe}
T.~Abe, \emph{Langlands correspondence for isocrystals and the existence of
  crystalline companions for curves}, J. Amer. Math. Soc. \textbf{31} (2018),
  no.~4, 921--1057.

\bibitem[Abe18b]{Abe2}
\bysame, \emph{Langlands program for {$p$}-adic coefficients and the petits
  camarades conjecture}, J. Reine Angew. Math. \textbf{734} (2018), 59--69.

\bibitem[AHR]{AHR2}
J.~Alper, J.~Hall, and D.~Rydh, \emph{The {\'e}tale local structure of
  algebraic stacks}, arXiv:1912.06162.

\bibitem[AHR20]{AHR1}
\bysame, \emph{A {L}una {\'e}tale slice theorem for algebraic stacks}, Ann. of
  Math. \textbf{191} (2020), no.~3, 675--738.

\bibitem[AMRT10]{AMRT}
A.~Ash, D.~Mumford, M.~Rapoport, and Y.~Tai, \emph{{S}mooth compactification of
  locally symmetric varieties}, 2nd ed. ed., Cambridge University Press,,
  Cambridge, 2010.

\bibitem[Ans22]{Anschutz}
J.~Ansch\"{u}tz, \emph{Extending torsors on the punctured {${\rm
  Spec}(A_{\inf})$}}, J. Reine Angew. Math. \textbf{783} (2022), 227--268.
  \MR{4373246}

\bibitem[BT84]{BT2}
F.~Bruhat and J.~Tits, \emph{Groupes r{\'e}ductifs sur un corps local. {II}.
  sch{\'e}mas en groupes. existence d'une donn{\'e}e radicielle valu{\'e}e},
  Inst. Hautes {\'E}tudes Sci. Publ. Math. (1984), no.~60, 197--376.

\bibitem[CL17]{CL}
B.~Chiarellotto and C.~Lazda, \emph{Around $\ell$-independence}, Compositio
  Mathematica \textbf{154} (2017), no.~1, 223--248.

\bibitem[CM04]{Conrad}
B.~Conrad and W.~R. Mann, \emph{Gross--zagier revisited}, pp.~67--164,
  Cambridge University Press, June 2004.

\bibitem[Del71]{De}
P.~Deligne, \emph{Travaux de {S}himura}, S{\'e}minaire {B}ourbaki, 23{\`e}me
  ann{\'e}e (1970/71), {E}xp. {N}o. 389, Springer, Berlin, 1971, pp.~123--165.
  Lecture Notes in Math., Vol. 244. \MR{0498581 (58 \#16675)}

\bibitem[Del73]{De4}
\bysame, \emph{Les constantes des {\'e}quations fonctionnelles des fonctions
  $l$}, Modular functions of one variable, II, Lecture Notes in Math.,
  Springer, Berlin, 1973, pp.~501--597.

\bibitem[Del80]{De3}
\bysame, \emph{La conjecture de {W}eil {II}}, Publ. Math. Inst. Hautes
  {\'E}tudes Sci. (1980), no.~52, 137--252.

\bibitem[Del82]{De1}
\bysame, \emph{Hodge cycles, motives, and {S}himura varieties}, Lecture Notes
  in Math., ch.~Hodge cycles on abelian varieties, pp.~9--100, Springer,
  Berlin, 1982.

\bibitem[DG72]{SGA7}
M.~Demazure and A.~Grothendieck, \emph{S{\'e}minaire de {G}{\'e}om{\'e}trie
  {A}lg{\'e}brique du {B}ois {M}arie - {G}roupes de monodromie en
  g{\'e}om{\'e}trie alg{\'e}brique - ({SGA} 7)}, Lecture Notes in Math.,
  Springer-Verlag, 1972.

\bibitem[Eme]{EmertonFormal}
M.~Emerton, \emph{Formal algebraic stacks}, (unpublished), available at
  https://www.math.uchicago.edu/{~}emerton/pdffiles/formal-stacks.pdf.

\bibitem[Fal02]{FaltingsAE}
G.~Faltings, \emph{Almost \'{e}tale extensions}, no. 279, 2002, Cohomologies
  $p$-adiques et applications arithm\'{e}tiques, II, pp.~185--270.

\bibitem[FC90]{FC}
G.~Faltings and C.~L. Chai, \emph{Degeneration of abelian varieties},
  Ergebnisse der Mathematik und ihrer Grenzgebiete (3), vol.~22,
  Springer-Verlag, 1990.

\bibitem[Fon94]{Fontaine}
J.~M. Fontaine, \emph{Expos\'e {VIII~:} {Repr\'esentations} $\ell$-adiques
  potentiellement semi-stables}, P\'eriodes $p$-adiques - S\'eminaire de Bures,
  1988 (Jean-Marc Fontaine, ed.), Ast\'erisque, no. 223, Soci\'et\'e
  math\'ematique de France, 1994, talk:8, pp.~321--347 (fr). \MR{1293977}

\bibitem[GLX]{GLX}
I.~Gleason, D.~Lim, and Y.~Xu, \emph{The connected components of affine
  {D}eligne--{L}usztig varieties}, arxiv:2208.07195.

\bibitem[GRR72]{GRR}
A.~Grothendieck, M.~Raynaud, and D.~S. Rim, \emph{Groupes de monodromie en
  g{\'e}om{\'e}trie alg{\'e}brique. i}, S{\'e}minaire de G{\'e}om{\'e}trie
  Alg{\'e}brique du Bois-Marie 1967--1969 (SGA 7 I),, Springer-Verlag, 1972.

\bibitem[Ill02]{Illusieoverview}
L.~Illusie, \emph{An overview of the work of {K}. {F}ujiwara, {K}. {K}ato, and
  {C}. {N}akayama on logarithmic \'{e}tale cohomology}, no. 279, 2002,
  Cohomologies $p$-adiques et applications arithm\'{e}tiques, II, pp.~271--322.

\bibitem[Ima24]{Imai}
N.~Imai, \emph{Local langlands correspondences in $\ell$-adic coefficients},
  Manuscripta Math. \textbf{175} (2024), no.~1-2, 345--364.

\bibitem[Kat87]{KatoFI}
K.~Kato, \emph{On {$p$}-adic vanishing cycles (application of ideas of
  {F}ontaine-{M}essing)}, Algebraic geometry, {S}endai, 1985, Adv. Stud. Pure
  Math., vol.~10, North-Holland, Amsterdam, 1987, pp.~207--251.

\bibitem[Ked03]{Kedlayaisoc}
K.~S. Kedlaya, \emph{Semistable reduction for overconvergent {$F$}-isocrystals
  on a curve}, Math. Res. Lett. \textbf{10} (2003), no.~2-3, 151--159.

\bibitem[Ked04]{Kedlayafaithful}
\bysame, \emph{Full faithfulness for overconvergent {$F$}-isocrystals},
  Geometric aspects of {D}work theory. {V}ol. {I}, {II}, Walter de Gruyter,
  Berlin, 2004, pp.~819--835.

\bibitem[Ked07]{Kedlayasemistab}
\bysame, \emph{Semistable reduction for overconvergent {$F$}-isocrystals. {I}.
  {U}nipotence and logarithmic extensions}, Compos. Math. \textbf{143} (2007),
  no.~5, 1164--1212.

\bibitem[Ked10]{Kedlayapadic}
\bysame, \emph{{$p$}-adic differential equations}, Cambridge Studies in
  Advanced Mathematics, vol. 125, Cambridge University Press, Cambridge, 2010.

\bibitem[Kis06]{Ki1}
M.~Kisin, \emph{Crystalline representations and {$F$}-crystals}, Algebraic
  geometry and number theory, Progr. Math., vol. 253, Birkh{\"a}user Boston,
  Boston, MA, 2006, pp.~459--496. \MR{2263197 (2007j:11163)}

\bibitem[Kis10]{Ki2}
\bysame, \emph{Integral models for {S}himura varieties of abelian type}, J.
  Amer. Math. Soc. \textbf{23} (2010), no.~4, 967--1012. \MR{2669706
  (2011j:11109)}

\bibitem[Kis17]{Ki3}
\bysame, \emph{Mod$p$ points on {S}himura varieties of abelian type}, J. Amer.
  Math. Soc. \textbf{30} (2017), no.~3, 819--914.

\bibitem[KMS22]{KMS}
M.~Kisin, K.~{Madapusi Pera}, and S.~Shin, \emph{Honda-{T}ate theory for
  {S}himura varieties}, Duke Math. J. \textbf{171} (2022), no.~7, 1559--1614,
  preprint.

\bibitem[Kos59]{Kostant}
B.~Kostant, \emph{The principal three-dimensional subgroup and the {B}etti
  numbers of a complex simple {L}ie group}, Amer. J. Math. \textbf{81} (1959),
  973--1032. \MR{114875}

\bibitem[Kot97]{Ko1}
R.~Kottwitz, \emph{Isocrystals with additional structure. {II}}, Compositio
  Math. \textbf{109} (1997), no.~3, 255--339. \MR{1485921 (99e:20061)}

\bibitem[KP18]{KP}
M.~Kisin and G.~Pappas, \emph{Integral models of {S}himura varieties with
  parahoric level structure}, Publ. Math. Inst. Hautes {\'E}tudes Sci.
  \textbf{128} (2018), no.~1, 121--218.

\bibitem[KPZ]{KPZ}
M.~Kisin, G.~Pappas, and R.~Zhou, \emph{Integral models of {S}himura varieties
  with parahoric level structure {II}}, arxiv:2402.05727.

\bibitem[KZ]{KZ}
M.~Kisin and R.~Zhou, \emph{Independence of $\ell$ for {F}robenius conjugacy
  classes attached to abelian varieties}, Ann. of Math., To appear.

\bibitem[Laf02]{Laf}
L.~Lafforgue, \emph{Chtoucas de {D}rinfeld et correspondance de langlands},
  Invent. Math. \textbf{147} (2002), 1--241.

\bibitem[Lan89]{Langlands}
R.~Langlands, \emph{Representation theory and harmonic analysis on semisimple
  {L}ie groups}, Math. Surveys Monogr., vol.~31, ch.~On the classification of
  irreducible representations of real algebraic groups, pp.~101--170,
  Providence, Rhode Island: Amer. Math. Soc., 1989.

\bibitem[Lan13]{Lan}
Kai-Wen Lan, \emph{Arithmetic compactifications of {PEL}-type {S}himura
  varieties}, London Mathematical Society Monographs, vol.~36, Princeton
  University Press, 2013.

\bibitem[Lar94]{Larsen}
M.~Larsen, \emph{On the conjugacy of element-conjugate homomorphisms}, Israel
  Journal of Mathematics \textbf{88} (1994), no.~1--3, 253--277.

\bibitem[{Mad}]{Keerthi-v4}
K.~{Madapusi Pera}, \emph{Toroidal compactifications of integral models of
  {S}himura varieties of {H}odge type (version 4)}, available at
  https://arxiv.org/abs/1211.1731v4.

\bibitem[{Mad}19]{Keerthi}
\bysame, \emph{Toroidal compactifications of integral models of {S}himura
  varieties of {H}odge type}, Ann. Sci. {\'E}cole Norm. Sup. \textbf{52}
  (2019), no.~2, 393--514.

\bibitem[Mar08]{Marmoraepsilon}
A.~Marmora, \emph{Facteurs epsilon {$p$}-adiques}, Compos. Math. \textbf{144}
  (2008), no.~2, 439--483.

\bibitem[Mil92]{Milne}
J.~Milne, \emph{The points on a {S}himura variety modulo a prime of good
  reduction}, ch.~The zeta functions of Picard modular surfaces, pp.~151--253,
  1992.

\bibitem[Noo09]{Noot}
R.~Noot, \emph{Classe de conjugaison du {F}robenius d'une vari{\'e}t{\'e}
  ab{\'e}lienne sur un corps de nombres}, J. Lond. Math. Soc. \textbf{79.1}
  (2009), 53--71.

\bibitem[Noo13]{Noot2}
R.~Noot, \emph{The system of representations of the {W}eil-{D}eligne group
  associated to an abelian variety}, Algebra Number Theory \textbf{7} (2013),
  no.~2, 243--281. \MR{3123639}

\bibitem[Pap23]{Pappascanonical}
G.~Pappas, \emph{On integral models of {S}himura varieties}, Math. Ann. (2023),
  no.~3-4, 2037--2097.

\bibitem[Pin89]{Pink}
R.~Pink, \emph{Arithmetical compactification of mixed {S}himura varieties},
  Ph.D. thesis, Rheinische Friedrich-Wilhelms-Universit{\"a}t Bonn, 1989.

\bibitem[PR21]{PRshtukas}
G.~Pappas and M.~Rapoport, \emph{p-adic shtukas and the theory of global and
  local shimura varieties}, 2021.

\bibitem[RV14]{RV}
M.~Rapoport and E.~Viehmann, \emph{Towards a local theory of {S}himura
  varieties}, M{\"u}nster J. of Math \textbf{7} (2014), 273--326.

\bibitem[RZ96]{RZ}
M.~Rapoport and Th. Zink, \emph{Period spaces for {$p$}-divisible groups},
  Annals of Mathematics Studies, vol. 141, Princeton University Press,
  Princeton, NJ, 1996. \MR{1393439 (97f:14023)}

\bibitem[Sai97]{Saito}
T.~Saito, \emph{Modular forms and $p$-adic {H}odge theory}, Invent. Math.
  \textbf{129} (1997), no.~3, 607--620.

\bibitem[Ser94]{Serre}
J-P. Serre, \emph{Propri{\'e}t{\'e}s conjecturales des groupes de galois
  motivique et des repr{\'e}sentations $\ell$-adiques}, Motives (J-P.~Serre
  U.~Jannsen, S.~Kleiman, ed.), vol.~1, Proc. Sympos. Pure Math., no.~55,
  American Mathematical Society, Providence, RI, 1994, pp.~377--400.

\bibitem[Shi02]{Shihoisoc}
A.~Shiho, \emph{Crystalline fundamental groups. {II}. {L}og convergent
  cohomology and rigid cohomology}, J. Math. Sci. Univ. Tokyo \textbf{9}
  (2002), no.~1, 1--163.

\bibitem[Ste65]{Steinberg:regular}
R.~Steinberg, \emph{Regular elements of semisimple algebraic groups}, Inst.
  Hautes \'{E}tudes Sci. Publ. Math. (1965), no.~25, 49--80.

\bibitem[SW20]{SW2}
P.~Scholze and J.~Weinstein, \emph{Berkeley lectures on {$p$}-adic geometry},
  Annals of Mathematics Studies, vol. 207, Princeton University Press,
  Princeton, NJ, 2020. \MR{4446467}

\bibitem[Tat79]{TateNTB}
J.~Tate, \emph{Number theoretic background}, Automorphic forms, representations
  and {$L$}-functions ({P}roc. {S}ympos. {P}ure {M}ath., {O}regon {S}tate
  {U}niv., {C}orvallis, {O}re., 1977), {P}art 2, Proc. Sympos. Pure Math.,
  XXXIII, Amer. Math. Soc., Providence, R.I., 1979, pp.~3--26.

\bibitem[vH24]{vHoftenOrdinary}
P.~van Hoften, \emph{On the ordinary hecke orbit conjecture}, Algebra \&;
  Number Theory \textbf{18} (2024), no.~5, 847--898.

\bibitem[Wei20]{Weidner}
M.~Weidner, \emph{Pseudocharacters of homomorphisms into classical groups},
  Transformation Groups \textbf{25} (2020), no.~4, 1345--1370.

\bibitem[Win97]{Wi}
J.~P. Wintenberger, \emph{Propri{\'e}t{\'e}s du groupe tannakien des structures
  de hodge p-adiques et torseur entre cohomologies cristalline et {\'e}tale},
  Ann. Inst. Fourier (1997), no.~47, 1289--1334.

\bibitem[Zar74a]{Zarhin1}
Y.~Zarhin, \emph{Isogenies of abelian varieties over fields of finite
  characteristic}, Mathematics of the USSR-Sbornik \textbf{24} (1974), no.~3,
  451--461.

\bibitem[Zar74b]{Zarhin2}
\bysame, \emph{A remark on endomorphisms of abelian varieties over function
  fields of finite characteristic}, Mathematics of the USSR-Izvestiya
  \textbf{8} (1974), no.~3, 477--480.

\bibitem[Zho20]{Z}
R.~Zhou, \emph{Mod p isogeny classes on {S}himura varieties with parahoric
  level}, Duke Math. J. \textbf{169} (2020), no.~15, 2937--3031.

\end{thebibliography}

\end{document}